\documentclass[12pt,a5paper]{article}

\usepackage{natbib,amsmath,defns,amssymb,eulervm,reducecode}
\allowdisplaybreaks
\IfFileExists{ajr.sty}{\usepackage{ajr}}{}
\title{Macroscale, slowly varying, models emerge from the microscale dynamics in long thin domains}

\author{A.~J. Roberts}
\date{\today}

\def\llangle{\langle\!\langle}
\def\rrangle{\rangle\!\rangle}

\Cal A\Cal B\Cal E\Cal G\Cal L\Cal M\Cal P\Cal W\Cal V\Cal Z

\Frak L\Frak V\Frak k\Frak u
\newcommand{\ix}{\ensuremath{\mathit x}}
\newcommand{\fx}{\ensuremath{
    \ifcase0x
    \or\mathrm x
	\or\mathchoice%
	  {\text{\footnotesize$\mathcal X$}}%
	  {\text{\footnotesize$\mathcal X$}}%
	  {\text{\scriptsize  $\mathcal X$}}%
	  {\text{\tiny        $\mathcal X$}}%
	\or\bar x
	\fi}}
	
\Bb C\Bb E\Bb N\Bb X\Bb Y\Bb R\Bb U

\newcommand{\diag}{\operatorname{diag}}

\DeclareMathOperator{\Pe}{Pe}
\newcommand{\ydim}{Y}
\newcommand{\pt}[2][X]{\ifx0#2%
  \else\ifx1#2(x-#1)%
  \else\frac{(x-#1)^{#2}}{#2!}%
  \fi\fi}

\newcommand{\xit}[1]{\ifx0#1 \else\ifx1#1\xi
  \else\frac{\xi^{#1}}{#1!}\fi\fi}

\newcommand{\tr}[1]{#1^\dagger}

\newcommand{\xec}{\ensuremath{x_c}}%{\hat x}} % for uncertain evaluation point
\newcommand{\xey}{\hat x} % for uncertain evaluation point
\newcommand{\xed}{\ensuremath{x_d}}%{\breve x}} % for uncertain evaluation point

\Vec u\Vec d\Vec c\Vec D\Vec C\Vec r\Vec R\Vec F
\newcommand{\rvd}{\rv'}
\newcommand{\nuv}{\ensuremath{\|\vec u\|}}
\newcommand{\tc}{\ensuremath{\tilde c}}
\newcommand{\td}{\ensuremath{\tilde d}}
\newcommand{\tC}{\ensuremath{\tilde C}}
\newcommand{\tD}{\ensuremath{\tilde D}}
\newcommand{\tF}{\ensuremath{\tilde F}}
\newcommand{\tu}{\ensuremath{\tilde u}}
\newcommand{\tL}{\ensuremath{\tilde\cL}}
\newcommand{\tA}{\ensuremath{\tilde\cA}}
\newcommand{\tB}{\ensuremath{\tilde\cB}}
\newcommand{\tV}{\ensuremath{\tilde\cV}}
\newcommand{\tW}{\ensuremath{\tilde\cW}}

\def\ou\big(#1,#2,#3\big){{e^{\if#31\else#3\fi t}\star}#1\,}
\newcommand{\Z}[1]{{e^{-t}\star}#1\,}

\newcommand{\bigtimes}{\mathop{\text{\LARGE$\times$}}}

\atBegin{endexample}{\qed}

\begin{document}

\maketitle

\begin{abstract}
Many practical approximations in physics and engineering invoke a relatively long physical domain with a relatively thin cross-section.  In this scenario we typically expect the system to have structures that vary slowly in the long dimension.  Extant mathematical approximation methodologies are typically self-consistency or limit arguments as the aspect ratio becomes unphysically infinite.  The proposed new approach is to analyse the dynamics based at each cross-section in a rigorous Taylor polynomial.  Then centre manifold theory supports the local modelling of the system's dynamics with coupling to neighbouring locales treated as a non-autonomous forcing.  The union over all cross-sections then provides powerful new support for the existence and emergence of a centre manifold model global in the long domain, albeit finite sized.  Our resolution of the coupling between neighbouring locales leads to novel quantitative estimates of the error induced by long slow space variations.  Four examples help develop and illustrate the approach and results.  The approach developed here may be used to quantify the accuracy of known approximations, to extend such approximations to mixed order modelling, and to open previously intractable modelling issues to new tools and insights.
\end{abstract}

\tableofcontents

\section{Introduction}

System of large spatial extent in some directions and relatively thin extent in other dimensions are important in engineering and physics. 
Examples include thin fluid films, flood and tsunami modelling \cite[e.g.]{Noakes06, Bedient88, LeVeque2011}, pattern formation in systems near onset \cite[e.g.]{Newell69, Cross93, Westra2003}, and wave interactions \cite[e.g.]{Nayfeh71b, Griffiths05}.
There are many formal approaches to mathematically describe,  by means of modulation or amplitude equations, the relatively long time and space evolution of these systems~\cite[e.g.]{Vandyke87}.
This article develops a new general approach to illuminate and enhance such practical approximations, albeit limited here to one long spatial direction.

The new approach is to examine the dynamics in the locale around any cross-section.
We find that a Taylor polynomial representation of the dynamics is only coupled to neighbouring locales via the highest order resolved derivative. 
Treating this coupling as an `uncertain forcing' of the local dynamics we in essence apply non-autonomous centre manifold theory \cite[e.g.]{Potzsche2006, Haragus2011} to prove the existence and emergence of a local model.
This theoretical support applies for all cross-sections and so establishes existence and emergence of a centre manifold model globally over the spatial domain to form an `infinite' dimensional centre manifold \cite[e.g.]{Gallay93, Aulbach96, Aulbach2000}.
Sections~\ref{sec:mdhx}--\ref{sec:pmid} develop the approach for linear systems, and then sections~\ref{sec:nhem}--\ref{sec:mndcd} generalise the approach to nonlinear systems.

In addition to existence and emergence proofs, we also establish a practical construction procedure based upon a polynomial generating function.
One result is that the new construction recovers symbolically the traditional multiple scale modelling as a special case (Corollary~\ref{cor:msm}), and justifies rigorously an established but previously formal procedure that derives `mixed order' models (Corollary~\ref{cor:robformproc}).
Further, the new approach derives a novel quantitative estimate of the leading error, equation~\eqref{eq:pderemain}, that results from the assumption of slow variations in space.
Interestingly, the theory is still valid in boundary layers and shocks, it is just that then the error terms are so large that the assumption of slow space variations is inappropriate.

Note that this article is not about finding and characterising steady solutions in long thin domains as explored, for example, by \cite{Haragus2011} or \cite{Mielke86}.  
Instead, this article focusses on the time evolution of structures that `vary slowly' in space.

Throughout, examples illustrate the concepts.  
Sections~\ref{sec:mdhx} and~\ref{sec:nhem} develop the basic concepts on a simple heat exchanger, linear and nonlinear respectively.
Sprinkled through the development of general linear theory, section~\ref{sec:pmid}, is the application to dispersion along a long thin channel \cite[e.g.]{Taylor53, Mercer90}.
The nonlinear theory developed in section~\ref{sec:mndcd}
is applied by sections~\ref{sec:egpes} and~\ref{sec:anpf} to  derive the Ginzburg--Landau model of patterns governed by the Swift--Hohenberg \pde, but now complete with a new quantitative error estimate.
The computer algebra code of Appendices~\ref{sec:camhe}, \ref{sec:camnhe} and~\ref{sec:campfshe} implements practical construction algorithms for these  examples and confirms the modelling extends to arbitrary order.

This article is \(\epsilon\)-free.  \let\epsilon\varepsilon
Although the analysis is based upon a fixed reference equilibrium (taken to be at the origin without loss of generality), crucially the subspace and centre manifold theorems guarantee the existence and emergence of models in a finite region about this reference equilibrium.
Sometimes such a finite region of applicability is large.
The only epsilons in this article appear in comparisons with other methodologies.

\section{Macroscale dynamics of a heat exchanger}
\label{sec:mdhx}

This section introduces the novel approach in perhaps the simplest example of the evolution of fields which slowly vary in space.  
The next section~\ref{sec:pmid} develop the approach for general linear \pde{}s.

\begin{figure}
\centering
\setlength{\unitlength}{0.01\linewidth}
\begin{picture}(100,14)
%\put(0,0){\framebox(100,14){}}
\put(20,1){\(b(x,t)\)}
\put(20,11){\(a(x,t)\)}
\put(45,6.5){exchange}
\multiput(20,0)(20,0)4{\color{magenta}
\put(0,6){\vector(0,+1){3}}
\put(2,9){\vector(0,-1){3}}
}
\thicklines
\multiput(90,10)(-20,0)4{\color{red}\vector(-1,0){20}}
\multiput(10,5)(+20,0)4{\color{blue}\vector(+1,0){20}}
\end{picture}
\caption{schematic diagram of two pipes (red and blue) carrying `heat' to the left and the right, with `temperature' fields~\(a\) and~\(b\), and exchanging heat.}
\label{fig:hescheme}
\end{figure}
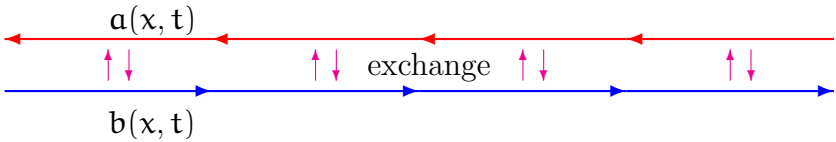

Consider the idealistic heat exchanger of Figure~\ref{fig:hescheme}. 
Say hot fluid enters the top pipe from the right having temperature field~$a(x,t)$, and cold fluid enters the bottom pipe from the left with temperature field~$b(x,t)$.  
Straightforward modelling gives that the governing \pde{}s are
\begin{equation}
\D ta=+U\D xa+\frac R2(b-a)
\quad\text{and}\quad 
\D tb=-U\D xb+\frac R2(a-b),
\label{eq:hedim}
\end{equation}
for flow to the left and right at equal and opposite velocities~$\pm U$, and for some inter-pipe exchange at rate~$R$.
Equivalently, \(a(x,t)\)~and~\(b(x,t)\) could be the probability density function of a random walker who walks steadily at constant speed~\(U\) but changes direction at random times, the changes occur at a rate~\(R\).
Our challenge is to find a description of the large time heat distribution, or equivalently the large time probabilty distribution of the random walker.

Non-dimensionalise space and time by choosing the reference time~$1/R$ and the reference length~$U/R$ so that the \pde{}s~\eqref{eq:hedim} are equivalent to the non-dimensional \pde{}s
\begin{equation}
\D ta=+\D xa+\rat12(b-a)
\quad\text{and}\quad 
\D tb=-\D xb+\rat12(a-b).
\label{eq:hepde}
\end{equation}
These \pde{}s are to be modelled with boundary conditions, for example that $a={}$hot at $x=L$, and $b={}$cold at $x=0$.
We aim to find the model that the mean temperature, $c(x,t)=\rat12(a+b)$, satisfies the diffusion \pde 
\begin{equation}
\D tc\approx\DD xc \quad\text{for }0<x<L\,.
%\quad\text{with some \textsc{bc}s.}
\label{eq:appdiff}
\end{equation}
Many extant mathematical methods, such as homogenisation and multiple scales \cite[e.g.]{Engquist08, Pavliotis07}, will straightforwardly derive this diffusion \pde.
The challenge here is to {rigorously} derive this \pde\ from a local analysis, 
complete with a novel quantitative error estimate,
and as a naturally emergent phenomena from a wide domain of initial conditions.

A future challenge is to determine the boundary conditions on the mean field~$c$,

The analysis here is clearer in `cross-pipe' modes.
Thus transform to mean and difference fields:
\begin{equation}
c(x,t):=\rat12(a+b)
\quad\text{and}\quad 
d(x,t):=\rat12(a-b),
\label{eq:meandiff}
\end{equation}
that is, \(a=c+d\) and \(b=c-d\).
Considering the mean and difference of the \pde{}s~\eqref{eq:hepde} gives the equivalent \pde\ system for these mean and difference fields
\begin{equation}
\D tc=\D xd
\quad\text{and}\quad 
\D td=-d+\D xc\,.
\label{eq:hecan}
\end{equation}
In this form we readily see that the difference field~$d$ tends to decay exponentially quickly, but that interaction between gradients of the mean and difference fields generates other effects: effects that are crucial in deriving the approximate model \pde~\eqref{eq:appdiff}.

Our approach is to expand the fields in their local spatial structure based around a station $x=X$.
As commented earlier, this approach is \(\epsilon\)-free.

\subsection{In the interior}

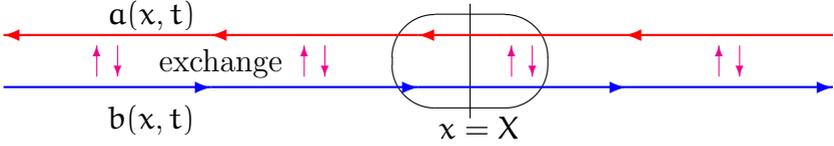
\begin{figure}
\centering
\setlength{\unitlength}{0.01\linewidth}
\begin{picture}(100,14)
%\put(0,0){\framebox(100,14){}}
\put(55,2){\line(0,1){11}}
\put(55,7.5){\oval(15,9)}
\put(52,0){\(x=X\)}
\put(20,1){\(b(x,t)\)}
\put(20,11){\(a(x,t)\)}
\put(25,6.5){exchange}
\multiput(19,0)(20,0)4{\color{magenta}
\put(0,6){\vector(0,+1){3}}
\put(2,9){\vector(0,-1){3}}
}
\thicklines
\multiput(90,10)(-20,0)4{\color{red}\vector(-1,0){20}}
\multiput(10,5)(+20,0)4{\color{blue}\vector(+1,0){20}}
\end{picture}
\caption{schematic diagram of the heat exchanger showing that we focus on modelling the dynamics in the locale of a fixed station \(x=X\)\,.}
\label{fig:hestation}
\end{figure}

Fix upon a station along the pipes at $x=X$ as shown in Figure~\ref{fig:hestation}.  
Consider the temperature fields in the vicinity of $x=X$.  
In the vicinity let's invoke Taylor's Remainder Theorem to express the fields exactly:
\begin{subequations}\label{eqs:trt}%
\begin{eqnarray}
c(x,t)&=&c_0(X,t)+c_1(X,t)\pt1 +c_2(X,t)\pt2
\nonumber\\&&{}
+c_3(X,t)\pt3
 +c_4(X,x,t)\pt4,
\label{eq:ctrt}
\\
d(x,t)&=&d_0(X,t)+d_1(X,t)\pt1 +d_2(X,t)\pt2
\nonumber\\&&{}
+d_3(X,t)\pt3 
+d_4(X,x,t)\pt4,
\label{eq:dtrt}
\end{eqnarray}
\end{subequations}
where by Taylor's Theorem we define
\begin{itemize}
\item $c_n(X,t):=\Dn xnc|_{x=X}$ for \(n=0,1,2,3\);
\item $c_4(X,x,t):=\Dn x4c|_{x=\xec}$ for some \(X \lessgtr\xec(X,x,t) \lessgtr x\);
\item $d_n(X,t):=\Dn xnd|_{x=X}$ for \(n=0,1,2,3\); and
\item $d_4(X,x,t):=\Dn x4d|_{x=\xed}$ for some \(X \lessgtr\xed(X,x,t) \lessgtr x\).
\end{itemize}
That is, \(c_4(X,x,t)\) and \(d_4(X,x,t)\) are fourth order derivatives but evaluated at some nearby but uncertain and typically moving locations.

For definiteness, this section truncates to a quartic approximation, \(N=4\); Appendix~\ref{sec:camhe} lists computer algebra code that not only derives the results summarised here, but also derives corresponding results for general truncation order~\(N\).

Substituting the Taylor expansions~\eqref{eqs:trt} into the governing \pde{}s~\eqref{eq:hecan} leads to (Appendix~\ref{sec:sts})
\begin{subequations}\label{eqs:tay4}%
\begin{eqnarray}
\sum_{n=0}^4 \D t{c_n}\pt{n}
&=&\sum_{n=0}^3d_{n+1}\pt{n}+\D x{d_4}\pt{4}\,,
\label{eq:ctay4}\\
\sum_{n=0}^4 \D t{d_n}\pt{n}
&=&
\sum_{n=0}^3(-d_{n}+c_{n+1})\pt{n}
\nonumber\\&&{}
+\left(-d_4+\D x{c_4}\right)\pt{4}\,.
\label{eq:dtay4}
\end{eqnarray}%
\end{subequations}

\paragraph{Local ODEs}

The derived equations~\eqref{eqs:tay4} are exact everywhere, but some places (namely near the station~\(X\)) they are useful in that the remainder terms \(c_{4x}:=\D x{c_4}\) and \(d_{4x}:=\D x{d_4}\) are negligibly small.
We derive a set of linearly independent equations for the coefficient functions \(c_n\) and~\(d_n\) simply by differentiation and evaluation at $x=X$ (Appendix~\ref{sec:lode}):
this process is almost the same as equating coefficients of~\((x-X)^n\), but with care to maintain exactness one finds extra terms generated by the remainders \(c_{4x}\) and~\(d_{4x}\).
The various derivatives of~\eqref{eq:ctay4} evaluated at \(x=X\) lead to the sequence of five \ode{}s for the $c_n$ coefficients:
\begin{subequations}\label{eqs:5pde}%
\begin{equation}
\dot c_{0}=d_{1},\quad
\dot c_{1}=d_{2},\quad
\dot c_{2}=d_{3},\quad
\dot c_{3}=d_{4},\quad
\dot c_{4}=5d_{4x}.
\label{eq:5cpde}
\end{equation}
Similarly, the various derivatives of~\eqref{eq:ctay4} evaluated at \(x=X\) lead to the five \ode{}s
\begin{equation}
\dot d_{0}=-d_0+c_{1},\quad
\ldots,\quad
%\dot d_{1}=-d_1+c_{2},\quad
%\dot d_{2}=-d_2+c_{3},\quad
\dot d_{3}=-d_3+c_{4},\quad
\dot d_{4}=-d_4+5c_{4x}.
\label{eq:5dpde}
\end{equation}%
\end{subequations}
Hereafter, because of the evaluation at the station \(x=X\), the symbols \(c_4\) and~\(d_4\) denote \(c_4(X,X,t)=\Dn x4c|_{x=X}\) and \(d_4(X,X,t)=\Dn x4d|_{x=X}\) respectively.
Further, the symbols \(c_{4x}\) and~\(d_{4x}\) denote the definite but uncertain `fifth-order' derivatives \(\D x{c_4}|_{x=X}\) and~\(\D x{d_4}|_{x=X}\).
\footnote{The `uncertain' derivatives \(c_{4x}\) and~\(d_{4x}\) might appear to be simple fifth-order derivatives, but they are a little more subtle.  For example, recall $c_4(X,x,t):=\Dn x4c|_{x=\xec}$ so by the chain rule \(c_{4x}=\left(\Dn x5c\times \D x\xec\right)|_{x=X}\) and hence is a fifth-derivative multiplied by an uncertain rate of change of location~\(\xec\).}
The functions $d_{4x}$~and~$c_{4x}$ are part of the closure problem for the local dynamics: 
the derivatives $d_{4x}$~and~$c_{4x}$ couple the dynamics at a station~$X$ with the dynamics at neighbouring stations.
It is by treating the terms $d_{4x}$~and~$c_{4x}$ as `uncertain' inputs into the local dynamics that we notionally make the vast simplification in \emph{apparently} reducing the problem from one of an infinite dimensional dynamical system to a tractable finite dimensional system.

\subsection{The slow subspace emerges}
\label{sec:sse}

For a dynamical system approach to modelling the local dynamics, define the state vector $\uv=(c_0,d_0, c_1,d_1, c_2,d_2, c_3,d_3, c_4,d_4)$ and group the ten \ode{}s~\eqref{eqs:5pde} into the matrix-vector system, of the form $\de t{\uv}=\cL\uv+\rv(t)$,
\begin{equation}
\de t{\uv}=\underbrace{\begin{bmatrix} 
0&0&0&1\\ 0&-1&1&0\\
&&0&0&0&1\\ &&0&-1&1&0\\
&&&&0&0&0&1\\ &&&&0&-1&1&0\\
&&&&&&0&0&0&1\\ &&&&&&0&-1&1&0\\
&&&&&&&&0&0\\ &&&&&&&&0&-1\\
\end{bmatrix}}_{\cL}\uv
+\underbrace{\begin{bmatrix} 0\\0\\0\\0\\0\\0\\0\\0\\
5d_{4x}\\5c_{4x} \end{bmatrix}}_{\vec r}
\label{eq:odes}
\end{equation}
where $d_{4x}$~and~$c_{4x}$ are some definite but uncertain functions.

\paragraph{Local slow subspace}
The system~\eqref{eq:odes} appears in the form of a `forced' linear system, so our usual first task is to understand the corresponding linear homogeneous system obtained by omitting the `forcing' (although here the the `forcing' is uncertain coupling with neighbouring locales).
The corresponding homogeneous system is upper triangular (also block toeplitz), so its eigenvalues are the diagonal, namely $0$~and~$-1$ each with multiplicity five.
The five eigenvalues~$-1$ indicates that after transients decay, roughly like~\Ord{e^{-t}}, the system evolves on the 5D slow subspace of the five eigenvalues~\(0\).

Let's construct this 5D slow subspace.
Two eigenvectors corresponding to the zero eigenvalue are found immediately, namely
\begin{equation*}
\vec v_0:=(1,0,\ldots,0),\quad \vec v_1:=(0,1,1,0,\ldots,0).
\end{equation*}
Other eigenvectors are generalised and come from solving $\cL \vec v_2=\vec v_0$, $\cL\vec v_3=\vec v_1$ and $\cL \vec v_4=\vec v_2-\vec v_0$:\footnote{An advantage of this choice of eigenvectors is that they are one in their $c$ components.}
\begin{align*}&
\vec v_2:=(0,0,0,1,1,0,0,0,0,0),\\&
\vec v_3:=(0,-1,0,0,0,1,1,0,0,0),\\&
\vec v_4:=(0,0,0,-1,0,0,0,1,1,0).
\end{align*}
Setting the matrix $\cV :=\begin{bmatrix} \vec v_0&\cdots&\vec v_4 \end{bmatrix}\in\RR^{10\times5}$, the slow subspace is then $\uv=\cV \cv$ where we conveniently choose to use $\cv:=(c_0,\ldots,c_4)\in\RR^5$ to directly parametrise the slow subspace because of the form chosen for the eigenvectors~$\vec v_k$.
On this slow subspace, from the eigenvectors via $\uv=\cV \cv$\,, the difference variables 
\begin{equation*}
\dv:=(d_0,d_1,d_2,d_3,d_4)=(c_1-c_3,c_2-c_4,c_3,c_4,0).
\end{equation*}
Further, on this slow subspace the evolution is guided by a toeplitz matrix, namely
\begin{equation*}
\de t{\cv}=\cA\cv=\begin{bmatrix} 0&0&1&0&-1\\ 
0&0&0&1&0\\ 0&0&0&0&1\\ 0&0&0&0&0\\ 0&0&0&0&0 \end{bmatrix}\cv
%\label{eq:}
\end{equation*}
However,  the system~\eqref{eq:odes} is perturbed from this slow subspace by the `forcing' of the uncertain coupling.
We next use a time dependent coordinate transform to account for the uncertain coupling.

\subsection{Time dependent normal form}
\label{sec:tdnf}

Near identity coordinate transforms underpin modelling dynamics.
In particular, time dependent coordinate transforms empower  understanding of the modelling of non-autonomous, and stochastic, dynamical systems \cite[e.g.]{Aulbach99, Arnold98, Roberts06k}.
This section analogously uses a time dependent coordinate transformation to separate exactly the slow and fast modes of the system~\eqref{eq:odes} in the presence of the uncertain `forcing'.

The coordinate transform introduces new dependent variables~\Cv\ and~\Dv.
In some sense, the new variables \(\Cv\approx \cv\) and \(\Dv\approx \dv\) so the coordinate transform is `near identity'.
Let's choose to parametrise precisely the slow subspace of the system~\eqref{eq:odes} by the variables~\(\Cv\): that is, on the subspace where the new stable variables \(\Dv=\vec 0\), then we insist on the exact identity \(\cv=\Cv\).
This choice simplifies subsequent construction of slowly varying models such as~\eqref{eq:appdiff}.

In the coordinate transform, the effects of the uncertain remainders appear as integrals over their past history.
In this problem we need to invoke the convolution
\begin{equation}
\ou\big(w(t),tt,-\big):=\int_0^te^{s-t}w(s)\,ds\,.
\label{eq:con}
\end{equation}
Then a key property is the time derivative 
\(\de t{(\ou\big(w,tt,-\big))}=-\ou\big(w,tt,-\big)+w\).

To construct the coordinate transformation one uses well established iteration described elsewhere \cite[e.g.]{Roberts06k}.
The details are not significant here, all we need are the results.
The computer algebra code of Appendix~\ref{sec:camhe}, for the case \(N=4\), produces the exact coordinate transform~\eqref{eq:cxC}--\eqref{eq:dxD}.%
\footnote{My web service \cite[]{Roberts07d} generates an analogous normal form transformation of the system~\eqref{eq:odes}.  
The only difference is that the web service chooses a parametrisation that avoids history integrals in the evolution on the slow subspace.} 
The coordinate transformation is exact because there is no neglect of any `small' terms.

Invoke the following time dependent, coordinate transform, \((\Cv,\Dv)\mapsto(\cv,\dv)\):
\begin{subequations}%
\begin{eqnarray}
c_0&=&C_0-D_1+D_3\,,
\\c_1&=&C_1-D_2+D_4\,,
\\c_2&=&C_2-D_3\,,
\\c_3&=&C_3-D_4\,,
\\c_4&=&C_4\,;
\end{eqnarray}\label{eq:cxC}%
\end{subequations}%
\begin{subequations}%
\begin{eqnarray}
d_0&=&D_0+C_1-C_3+5\ou\big(\ou\big(c_{4x},tt,-\big),tt,-\big)
+5\ou\big(\ou\big(\ou\big(c_{4x},tt,-\big),tt,-\big),tt,-\big),
\\d_1&=&D_1+C_2-C_4+5\ou\big(d_{4x},tt,-\big)+5\ou\big(\ou\big(d_{4x},tt,-\big),tt,-\big),
\\d_2&=&D_2+C_3-5\ou\big(\ou\big(c_{4x},tt,-\big),tt,-\big),
\\d_3&=&D_3+C_4- 5\ou\big(d_{4x},tt,-\big),
\\d_4&=&D_4+5\ou\big(c_{4x},tt,-\big).
\end{eqnarray}\label{eq:dxD}%
\end{subequations}
In these new variables~\((\Cv,\Dv)\) the original system~\eqref{eq:odes} is identically the separated system
\begin{subequations}%
\begin{eqnarray}
{\dot D_0}&=&-D_0-D_2+D_4\,,
\\{\dot D_1}&=&-D_1-D_3\,,
\\{\dot D_2}&=&-D_2-D_4\,,
\\{\dot D_3}&=&-D_3\,,
\\{\dot D_4}&=&-D_4\,.
\end{eqnarray}\label{eq:dDdt}%
\end{subequations}%
\begin{subequations}%
\begin{eqnarray}
{\dot C_0}&=&C_2-C_4+5\ou\big(d_{4x},tt,-\big) + 5\ou\big(\ou\big(d_{4x},tt,-\big),tt,-\big),
\\{\dot C_1}&=&C_3- 5\ou\big(\ou\big(c_{4x},tt,-\big),tt,-\big),
\\{\dot C_2}&=&C_4- 5\ou\big(d_{4x},tt,-\big),
\\{\dot C_3}&=&5\ou\big(c_{4x},tt,-\big),
\\{\dot C_4}&=&5d_{4x}\,.
\end{eqnarray}\label{eq:dCdt}%
\end{subequations}
In this separated system of these new variables, one immediately sees from~\eqref{eq:dDdt} that the new stable variables \(D_n\to0\) as \(t\to\infty\); moreover, they decay exponentially quickly, \(D_n=\Ord{e^{-\gamma t}}\) for any chosen rate \(0<\gamma<1\).
That is, \(\Dv=\vec 0\) is the exact slow subspace for the `forced' system~\eqref{eq:odes}.

\paragraph{The slowly varying model}
Recall the exact Taylor polynomial~\eqref{eq:ctrt}.
Given the exact coordinate transform~\eqref{eq:cxC}, and that \(D_n=\Ord{e^{-\gamma t}}\), the polynomial~\eqref{eq:ctrt} asserts the mean field
\begin{eqnarray}
c(x,t)&=&
C_0(X,t)
+\pt1C_1(X,t)
+\pt2C_2(X,t)
\nonumber\\&&{}
+\pt3C_3(X,t)
+\pt4C_4(X,t)
+\Ord{e^{-\gamma t}}.
\label{eq:slowcid}
\end{eqnarray}
Crucially, the left-hand side is independent of the station~\(X\).
If the right-hand side was just a local approximation, then the field it generates would depend upon the station~\(X\).
But the right-hand side is exact (with its unknown but exponentially quickly decaying transients).
This exactness is maintained because we keep the remainder terms in the analysis.
Consequently, the mean field expression~\eqref{eq:slowcid} is independent of the station~\(X\).

To obtain an exact \pde\ of the slow variations in the mean field~\(c\), take the time derivative of~\eqref{eq:slowcid} and evaluate at \(x=X\).  Remembering that the derivative of the history convolution \(\de t{}(\ou\big(w,tt,-\big))=-(\ou\big(w,tt,-\big))+w\), we derive
\begin{eqnarray*}
\D tc&=&
\D t{C_0}+\Ord{e^{-\gamma t}}
\\&=&C_2-C_4
+5 \ou\big(d_{4x},tt,-\big)
+5\ou\big(\ou\big(d_{4x},tt,-\big),tt,-\big)
+\Ord{e^{-\gamma t}}
\\&=&\left[C_2+\ou\big(d_{4x},tt,-\big)\right]-C_4
+\ou\big(\ou\big(d_{4x},tt,-\big),tt,-\big)+ \ou\big(d_{4x},tt,-\big)
+\Ord{e^{-\gamma t}}
\\&=&c_2-c_4
+5(1+\ou\big({},tt,-\big))\ou\big(d_{4x},tt,-\big)
+\Ord{e^{-\gamma t}}
\\&=&\DD xc-\Dn x4c
+5(1+\ou\big({},tt,-\big))\ou\big(d_{4x},tt,-\big)
+\Ord{e^{-\gamma t}}.
\end{eqnarray*}
Consequently, an exact statement of the mean field~\(c\) is thus
\begin{equation}
\D tc=\DD xc-\Dn x4c
+5(1+\ou\big({},tt,-\big))\ou\big(d_{4x},tt,-\big)
+\Ord{e^{-\gamma t}}.\label{eq:pdeC}
\end{equation}
In principle, equation~\eqref{eq:pdeC} is an exact integro-differential equation for the system: the integral part coming from the history convolutions of the coupling~\(d_{4x}\) with other stations~\(X\).
In practice, we read off an approximate model from this transformed version of the original heat exchanger system~\eqref{eq:hecan}.
The rigorous slowly varying model is then the \pde~\eqref{eq:pdeC} with \(\Ord{e^{-\gamma t}}\) neglected as a quickly decaying transient, and  the uncertain~$(1+\ou\big({},tt,-\big))\ou\big(d_{4x},tt,-\big)$ neglected as its error.

To characterise the magnitude of this error, recall from~\eqref{eq:dxD} that at all stations \(d_4 =5\ou\big(c_{4x},tt,-\big)+\Ord{e^{-\gamma t}}\).
We thus estimate that $5d_{4x}=\Ord{c_{4xx},e^{-\gamma t}}=\Ord{\Dn x6c,e^{-\gamma t}}$.

The \pde~\eqref{eq:pdeC}, with its second and fourth order derivatives of the mean field~\(c\), is an example of so-called mixed order models.
The extant mathematical methodologies of homogenisation and multiple scales promote an aversion to such mixed order models.
Our analysis shows that such models are rigorously justifiable.

\section{A PDE models interior cylindrical dynamics}
\label{sec:pmid}

Inspired by the successful exact modelling of the heat exchanger in section~\ref{sec:mdhx}, this section establishes analogous exact modelling in more general linear systems.
This section forms a foundation for the nonlinear, centre manifold, theory of section~\ref{sec:mndcd}.

This section develops models of the macroscale dynamics of any \pde\ in the linear class
\begin{equation}
\D tu=\fL_0u+\fL_1\D xu+\fL_2\DD xu
+\cdots
\label{eq:upde}
\end{equation}
on a cylindrical domain~\(\XX\times\YY\) for some field~\(u(x,y,t)\) in a given Banach space~\UU\ (finite or infinite dimensional), where \(u:\XX\times\YY\times\RR\to\UU\) is a function of 1D longitudinal position \(x\in\XX\subset \RR\), cross-sectional position \(y\in\YY\subset\RR^\ydim\), and time \(t\in\RR\).
The longitudinal domain~\(\XX\) (open) may be finite, say~\((0,L)\), or infinite~(\RR), or \(L\)-periodic.
The cross-section~\YY\ may be as simple as the index set~\(\{1,2\}\) as for the heat exchanger~\eqref{eq:hepde}, or the whole of~\(\RR^\ydim\) as in application to the modelling of marginal probability distributions by Fokker--Planck equations \cite[e.g.]{Knobloch83}.
The operators~\(\fL_\ell\) are assumed autonomous and independent of longitudinal position~\(x\); they only operate in the `microscale' cross-section~\(y\).%
\footnote{Nonetheless, cross-sectional operators that depend upon longitudinal position~\(x\) and time~\(t\) are of interest in a range of applications and are the subject of further research.}
In applications, the sum of terms in the \pde~\eqref{eq:upde} often truncate at the second order derivatives. 
However, our analysis caters for arbitrarily high order \pde{}s, such as the fourth order truncation invoked in the pattern formation example of subsection~\ref{sec:egpes}.
%It appears straightforward to generalise~\eqref{eq:upde} to higher order longitudinal derivatives as invoked in the pattern formation example of subsection~\ref{sec:egpes}.??

\begin{example}[shear dispersion] \label{eg:sheardisp}
As an example threaded through the discourse, consider classic shear dispersion in a 2D channel \cite[e.g.]{Smith83b}.  
The system has non-dimensional mechanisms parametrised by a Peclet number~\(\Pe\), the longitudinal advection along the channel occurs with velocity \(w(y):=\Pe(1-y^2)\), and diffusion of strength one: for a concentration field~\(u(x,y,t)\) the non-dimensional governing conservative advection-diffusion equation is
\begin{equation*}
\D tu=-w(y)\D xu+\DD xu+\DD yu\,.
\end{equation*}
This shear dispersion system fits into our framework by the following choices: operator \(\fL_0:=\DD y{}\) with  Neumann boundary conditions; operator \(\fL_1:=-w(y)\); operator \(\fL_2:=1\); and \(\fL_\ell:=0\) for \(\ell>2\).  
The channel cross-section restricts~\(y\) to the non-dimensional domain \(\YY=\{y\in\RR: |y|<1\}\) and associates the operator~\(\fL_0\) with conservative Neumann boundary conditions of \(\D yu=0\) at \(y=\pm 1\).  
The channel typically stretches from an inlet at \(x=0\) to an outlet at \(x=L\) (notionally large) so that the longitudinal domain \(\XX=\{x\in\RR: 0<x<L\}\).
\end{example}

For \pde{}s in the general form~\eqref{eq:upde}, assume the field~\(u\) is smooth enough to have continuous \(2N\)~derivatives in~\(x\), \(u\in C^{2N}(\XX\times\YY\times\RR,\UU)\), for some pre-specified Taylor series truncation~\(N\).

This section establishes the following proposition.
\begin{proposition}[slowly varying \pde] \label{thm:svpde}
Let \(u(x,y,t)\) be governed by a \pde\ of the form~\eqref{eq:upde} satisfying Assumption~\ref{ass:spec}.  
Define the `mean field' \(c(x,t):=\langle Z_0(y),u(x,y,t)\rangle\) for \(Z_0(y)\) and inner product of Definition~\ref{def:adj}.
Then, in the regime of `slowly varying solutions' the mean field~\(c\) satisfies the \pde
\begin{equation}
\D tc=\sum_{n=0}^NA_n\Dn xnc\,, \quad x\in\XX\,,
\label{eq:dcdta}
\end{equation}
in terms of matrices~\(A_n\) given by~\eqref{eq:vzadef}--\eqref{eq:toeport}, 
to an error quantified by~\eqref{eq:pderemain}, and upon ignoring transients decaying exponentially quickly in time.
\end{proposition}

\subsection{Rewrite the local field}
\label{sec:rlf}

Choose a cross-section at longitudinal station $X\in\XX$\,.
Then invoke Taylor's Remainder Theorem to write the field~\(u\) in terms of a local polynomial about the cross-section \(x=X\):
\begin{equation}
u(x,y,t)=\sum_{n=0}^{N-1}u_n(X,y,t)\pt{n}
+u_N(X,x,y,t)\pt N\,,
\label{eq:utrtn}
\end{equation}
where \(u_n:=\Dn xnu\) evaluated at the station \(x=X\), except for the last term \(u_N:=\Dn xNu\) which is evaluated at some point \(x=\xey(X,x,y,t)\) that is some function of station~\(X\), longitudinal position~\(x\), cross-section position~\(y\), and time~\(t\). 
By Taylor's Remainder Theorem, the location~\(\xey\) satisfies \(X \lessgtr\xey \lessgtr x\).
%\begin{example} 
%If the function \(u:=1/(1+xy)\), near station \(X=0\) and expanded to \(N=2\) terms, then \(u=1-xy+x^2y^2/(1+\xey y)^3\) for function \(\xey=[(1+xy)^{1/3}-1]/y\).
%\end{example}
However, although the function~\(\xey(X,x,y,t)\) in principle exists, in our modelling \(\xey\)~appears as an implicit uncertain part of the modelling closure.  
The location~\(\xey\) is implicit because it is hidden in the dependency upon~\(x\) of the last factor~\(u_N(X,x,y,t)\), and also implicit in some of the dependence upon \(y\) and~\(t\).
The uncertainty of~\(\xey\) is reflected in uncertainty about where the \(N\)th derivative~\(u_N\) is `located', albeit known to be between \(x\)~and~\(X\).

\paragraph{Derive exact local ODEs}
Let's derive some exact \ode{}s for the the evolution of the coefficients~\(u_n(X,y,t)\) and~\(u_N(X,x,y,t)\).
The \pde~\eqref{eq:upde} invokes various derivatives of the field~\(u\):  the Taylor polynomial gives, after a little rearrangement, the \(\ell\)th~derivative
\begin{equation}
\Dn x\ell u=\sum_{n=0}^{N-\ell}u_{n+\ell}\pt n
+\sum_{n=N-\ell+1}^N\binom{\ell}{N-n}\Dn x{n+\ell-N}{u_N}\pt n\,.
\label{eq:dudxl}
\end{equation}
Consequently, substituting the Taylor polynomial~\eqref{eq:utrtn} into the \pde~\eqref{eq:upde} gives, after rearrangement,
\begin{eqnarray}
\sum_{n=0}^N \D t{u_n}\pt n
&=&\sum_{n=0}^N\left(\sum_{\ell=0}^{N-n}\fL_\ell u_{n+\ell}\right)\pt n
\nonumber\\&&{}
+\sum_{n=0}^{N}\pt n \sum_{k=1}^\infty 
\binom{N-n+k}{N-n}\fL_{N-n+k}\Dn x{k}{u_N}\,.
%\nonumber\\&&{}
%+\fL_1u_{Nx}\pt N
%+2\fL_2u_{Nx}\pt{(N-1)}
%+\fL_2u_{Nxx}\pt N\,.
\qquad\label{eq:sumpde}
\end{eqnarray}
Be careful about details of this and subsequent equation:
\begin{itemize}
\item partial derivatives in \(X\), \(x\), \(y\) and~\(t\) are done keeping constant the other three variables in the foursome;
\item whereas for index \(n=0,\ldots,N-1\) the time derivative \(\D t{u_n}\) is straightforward to interpret, for index \(n=N\) the time derivative implicitly contains effects due to the dependency  upon time of the uncertain locations~\(\xey\);
%\item when, as is typical, the field~\(u\) has multiple components, the products \(u_{NX}\xey_x\), \(u_{NXX}\xey_x^2\) and~\(u_{NX}\xey_{xx}\) are all done component-wise, that is, in finite dimensions for example, using \(u_j\)~and~\(\xey_j\) to denote their \(j\)th~components then the \(j\)th~component of \(u_{NX}\xey_x\) is simply \(u_{jNX}\xey_{jx}\);
\item and, lastly, equation~\eqref{eq:sumpde} is exact for all \(x,X\in\XX\) as the Taylor polynomial~\eqref{eq:utrtn} is exact (but regions of rapid variation will have large uncertain remainder terms~\(\Dn xk{u_N}\)).
\end{itemize}
Since equation~\eqref{eq:sumpde} is exact, we differentiate equation~\eqref{eq:sumpde} with respect to~\(x\) up to \(N\)~times, and evaluate each derivative at \(x=X\) to obtain valid exact equations.  
This differentiation of equation~\eqref{eq:sumpde} \(n\)~times and evaluating at \(x=X\) is nearly equivalent to the heuristic of equating coefficients of~\((x-X)^n\)---the difference lies in the `remainder' terms involving extra \(x\)~dependence implied by the uncertain location~\(\xey\).
Proceeding to  differentiate equation~\eqref{eq:sumpde} \(n\)~times with respect to~\(x\) and evaluating as \(x\to X\) gives the set of  \(N+1\)~\ode{}s
\begin{equation}
\D t{u_n}=
\fL_0u_n+\fL_1u_{n+1}+\cdots+\fL_{N-n}u_{N}+r_n
\,,\quad\text{for }n=0,1,\ldots,N\,,
\label{eq:uodes}
\end{equation}
where, after some rearrangement, the \emph{remainder} 
%\ajr{Should this binomial coefficient be C_N^{k+n}? Checked it is OK.}
\begin{equation}
r_n(X,y,t):=\sum_{k=1}^\infty \binom{k+N}{N}\fL_{k+N-n}\Dn xk{u_N}\,,
\quad\text{for } n=0,1,\ldots,N\,.
\label{eq:remain}
\end{equation}
The formal infinite sum in~\eqref{eq:remain} typically truncates depending upon the truncation of the \pde~\eqref{eq:upde}.
\begin{example}[shear dispersion continued] %\label{eg:default}
For example, when \(\fL_\ell=0\) for \(\ell>2\)---a common second-order truncation of the \pde~\eqref{eq:upde}---the remainder
\begin{equation*}
r_n=\begin{cases}
0\,,&n=0,1,\ldots,N-2\,,\\
(N+1)\fL_2u_{Nx}\,, &n=N-1\,,\\
(N+1)\fL_1u_{Nx}
+\frac{(N+1)(N+2)}2\fL_2u_{Nxx}\,, &n=N\,.
\end{cases}
\end{equation*}
In shear dispersion, Example~\ref{eg:sheardisp}, this remainder is specifically
\begin{equation*}
r_n=\begin{cases}
0\,,&n=0,1,\ldots,N-2\,,\\
(N+1)u_{Nx}\,, &n=N-1\,,\\
-(N+1)w(y)u_{Nx}
+\frac{(N+1)(N+2)}2u_{Nxx}\,, &n=N\,.
\end{cases}
\end{equation*}
\end{example}

Equation~\eqref{eq:uodes} forms a system of \ode{}s for the local field derivatives~\(u_n\).
Denote the (meta-)vector of coefficients~\(u_n\) by~\(\uv:=(u_0,u_1,\ldots,u_N)\in\UU^{N+1}\), and similarly for the remainders, \(\vec r:=(r_0,r_1,\ldots,r_N)\in\UU^{N+1}\).
Then rewrite equation~\eqref{eq:uodes} as the apprently `forced' linear system \(\D t{\uv}=\cL\uv+\rv\) for a  block Toeplitz matrix\slash operator~\(\cL:\UU^{N+1}\to\UU^{N+1}\); that is,
\begin{equation}
\D t{\uv}=\underbrace{\begin{bmatrix} 
\fL_0&\fL_1&\fL_2&\cdots&\fL_N
\\ &\fL_0&\fL_1&\ddots&\vdots
\\ &&\fL_0&\ddots&\fL_2
\\ &&&\ddots&\fL_1
\\ &&&&\fL_0
\end{bmatrix}}_{\cL}
\uv+
\underbrace{\begin{bmatrix} r_0\\r_1\\\vdots\\r_{N-1}\\r_N 
\end{bmatrix}}_{\rv} .
\label{eq:odesn}
\end{equation}
This system of \ode{}s~\eqref{eq:odesn} is an exact statement of the dynamics in the locale of the station~\(X\).
System~\eqref{eq:odesn} might appear closed, but it is actually coupled by the derivatives~\(\Dn xk{u_N}\), \(k\geq1\), through the remainder~\eqref{eq:remain}, to the dynamics of neighbouring stations. 
%Such coupling may also occur implicitly in the terms~\(\fL_\ell u_N\) when the \(y\)-variations in~\(u_N\) also have a component due to any \(y\)-variations in the location~\(\xey\). % Omit as not correct because evaluate at x=X now.
Thus system~\eqref{eq:odesn} is two faced: when viewed globally as the union over all stations~\(X\in\XX\) it is a deterministic autonomous system; but when viewed locally at any one station~\(X\in\XX\) the inter-station coupling implicit in the remainder~\(\vec r\) appears as time dependent~`forcing'.

Our plan is to treat the remainders as `uncertainities' and derive models where the effects of the uncertain remainders can be bounded into the precise error statement~\eqref{eq:pderemain} for the models.
Roughly, since the remainder is linear in \(\Dn xk{u_N}\propto \Dn x{N+k}u\), for slowly varying fields~\(u\) these high derivatives are small and so the errors due to the uncertain remainder will be small.
If the field~\(u\) has any localised internal or boundary layers, then in these locales the errors due to the uncertain remainder will be appropriately large.

\subsection{Model the local `autonomous' system}
\label{sec:mlas}

To analyse the uncertainly `forced' system~\eqref{eq:odesn} we must first understand the autonomous local system 
\begin{equation}
\D t{\uv}=\cL\uv\,.
\label{eq:odesna}
\end{equation}
The invariant subspaces of~\cL\ are a key part of our understanding of the autonomous system~\eqref{eq:odesna}.
The linear operator~\(\cL\) is `block' upper triangular so the spectrum of~\(\cL\) is the same as each `block' on the diagonal, namely that of the eigenproblem \(\fL_0 v=\lambda v\) (subject to any boundary conditions on~\(\partial\YY\) implicit in the symbol~\(\fL_0\)).  
\begin{assumption} \label{ass:spec}
The Banach space~\(\UU\) is the direct sum of two closed \(\fL_0\)-invariant subspaces, \(\EE_c^0\) and~\(\EE_s^0\), and the corresponding restrictions of~\(\fL_0\) generate strongly continuous semigroups \cite[]{Gallay93, Aulbach96}.
Further, assume that the operator~\(\fL_0\) has a discrete spectrum of eigenvalues~\(\lambda_1,\lambda_2,\ldots\) (repeated according to multiplicity) with corresponding and \emph{complete} set of linearly independent (generalised) eigenvectors~\(v_1,v_2,\ldots\)\,. 
We assume the first \(m\)~eigenvalues \(\lambda_1,\ldots,\lambda_m\) of~\(\fL_0\) all have real part satisfying \(|\Re\lambda_j|\leq\alpha\) and hence span the \(m\)-dimensional \emph{centre subspace}~\(\EE_c^0\) \cite[Chapt.~4, e.g.]{Chicone2006}.%
\footnote{Potentially, the centre subspace~\(\EE_c^0\) could be an infinite-D Banach space, appropriate to pattern forming models with spanwise structures, but we leave this potential for future research.}
Also, assume that there is no unstable subspace: that is, all other eigenvalues~\(\lambda_{m+1},\lambda_{m+2},\ldots\) have real part negative and well separated from the centre eigenvalues, namely \(\Re\lambda_j\leq-\beta<-N\alpha\) for \(j=m+1,m+2,\ldots\)\,, and that there is a complete set of corresponding eigenvectors~\(v_{m+1},v_{m+2},\ldots\) which span the stable space~\(\EE_s^0\).
\end{assumption}
\begin{example}[shear dispersion continued] %\label{eg:default}
Here the cross-channel diffusion eigenproblem is \(\lambda v=\fL_0v=\DD yv\) with  Neumann boundary conditions at \(y=\pm 1\).
Here the Banach space \(\UU=\{v(y)\in H^2[-1,1]\mid \D yv=0 \text{ at }y=\pm1\}\).
This eigenproblem is straightforward giving,  for \(j=1,2,3,\ldots\), eigenfunctions \(v_j=\cos[(j-1)\pi(y+1)/2]\) with corresponding eigenvalues \(\lambda_j=-(j-1)^2\pi^2/4\)\,.
There is one eigenvalue of zero (hence \(\alpha=0\)) corresponding to the 1D centre subspace~\(\EE_c^0\) of fields constant across the channel.
The countably infinite other eigenvalues are all\({}\leq-\beta=-\pi^2/4<0\).
\end{example}

However, much of the following derivation and discussion applies to other cases that may be of interest in other circumstances.  
One may be interested in a centre subspace among both stable and unstable modes, or in a \emph{slow subspace} corresponding to pure zero eigenvalues, or in some other `normal mode' \emph{subcentre subspace} \cite[e.g.]{Lamarque2011}, or in the \emph{centre-unstable subspace}, and so on.
We focus on the centre subspace among otherwise decaying modes as then the centre subspace contains the long term dynamics from general initial conditions (\cite{Robinson96} called it asymptoticly complete).
  
Recall that the operator~\cL\ is block upper triangular with~\(\fL_0\) repeated \((N+1)\)~times on the diagonal blocks.
Thus the spectrum of~\cL\ is the spectrum of~\(\fL_0\) repeated \((N+1)\)~times.
As there are \(m\)~centre eigenvalues for each block~\(\fL_0\) on the diagonal, the operator~\cL\ has an \(m(N+1)\)-dimensional centre subspace, denoted~\(\EE_c^N\). 
Further, all other eigenvalues of~\cL\ have real part negative~(\(\leq-\beta<0\)).
Hence this \(m(N+1)\)-dimensional centre subspace is exponentially quickly attractive from all initial conditions: the longest lasting transients decay roughly like~\(e^{-\beta t}\).   
The evolution on the centre subspace thus forms a long term model of the autonomous system~\eqref{eq:odesna}.

\paragraph{Generalised eigenvectors span the centre subspace}
Because of its block Toeplitz structure, operator~\cL\ is generally non-normal and its eigenspaces involve many generalised eigenvectors.
Typically, the only `pure' centre eigenvectors (corresponding to the centre eigenvalues) of the non-normal~$\cL$ are $\vec v_k:=(v_k,0,\dots,0)\in\UU^{N+1}$ for \(k=1,\ldots,m\).  
Recall that Assumption~\ref{ass:spec} supposes a complete set of linearly independent eigenvectors is~\(\{v_1,\ldots,v_m\}\) (generalised if necessary) to form a basis for the centre subspace of~\(\fL_0\).
In applications, these \(m\)~eigenvectors correspond to well established neutral or oscillatory modes of the cross-sectional dynamics at station \(x=X\).  
The difference here is that we now explore longitudinal structures, via the generalised eigenvectors, without invoking the scaling heuristics of traditional slowly varying methodologies. 

The other centre eigenvectors of~\cL\ are (typically) generalised eigenvectors which straightforward linear algebra finds will form a toeplitz-like structure.
\footnote{Perhaps one reason why a rigorous justification of models `slowly varying' in space is difficult is that such modelling needs to invoke generalised eigenmodes and their dynamics.  
Physically, we need generalised eigenmodes to cope with, for example, initial conditions that transiently `feed' into organised structures before cross-sectional dissipation fully acts.}  
%To proceed, f, and then we must have \(\fL_0V_0=V_0 A_0\) for some \(m\times m\)~matrix~\( A_0\in\RR^{m\times m}\) which has only zero or pure imaginary eigenvalues.
%Further, upon , by basic linear algebra there must exist a projection `matrix' of left\slash adjoint eigenvectors~\(Z_0\in\UU^{m}\) orthogonal to~\(V_0\), \(\langle Z_0,V_0\rangle=I_m\), such that \(\tr\fL_0Z_0=Z_0\tr A_0\) where \(\tr{}\) denotes the appropriate adjoint.
\begin{definition} \label{def:adj}
  Define an inner product \(\langle,\rangle:\UU\times\UU\to\RR\)\,, so there exists a corresponding adjoint~\(\tr\fL_0\).
  Also use this inner product symbol \(\langle,\rangle:\UU^{1\times m}\times\UU^{1\times m}\to\RR^{m\times m}\) to mean \(\langle[z_j],[v_j]\rangle:=\begin{bmatrix} \langle z_i,v_j \rangle \end{bmatrix}\) (a matrix of inner products).
  Form the `matrix'  \(V_0:=\begin{bmatrix} v_1&\cdots&v_m \end{bmatrix}\in\UU^{1\times m}\) of centre eigenvectors of~\(\fL_0\).
  Then elementary algebra assures us that there exists a projection `matrix' of left\slash adjoint eigenvectors~\(Z_0\in\UU^{1\times m}\) orthogonal to~\(V_0\), and there exists a matrix~\( A_0\in\RR^{m\times m}\), with eigenvalues~\(\{\lambda_1,\ldots,\lambda_m\}\), such that
  \begin{equation}
\fL_0V_0=V_0 A_0\,,\quad
\tr\fL_0Z_0=Z_0\tr A_0\,,\quad
\langle Z_0,V_0\rangle=I_m
\label{eq:vzadef}
\end{equation}
\end{definition}

\begin{example}[shear dispersion continued] %\label{eg:default}
Define the natural inner product to be the cross-channel average \(\langle z,v\rangle=\frac1{2}\int_{-1}^1 zv\,dy\).  Then here~\(\fL_0=\DD y{}\) is self-adjoint in this inner product, and with left and right centre (slow) eigenfunctions \(z_1=v_1=1\) corresponding to eigenvalue \(\lambda_1=0\)\,.  Consequently, \(Z_0=V_0=1\), and \(A_0=0\).
\end{example}

\paragraph{Recursively define generalised eigenvectors}
After solving the basic eigenproblem~\eqref{eq:vzadef} for~\( A_0\), \(V_0\) and~\(Z_0\), now recursively solve the following sequence of problems for \( A_n\in\RR^{m\times m}\)~and~\(V_n\in\UU^{1\times m}\), \(n=1,2,\ldots,N\), 
%(upon setting \(V_{-1}:=0\in\UU^{1\times m}\)),
\begin{subequations}\label{eqs:toep}%
\begin{eqnarray}&&
 A_n:=
 \sum_{k=1}^n\langle Z_0,\fL_k V_{n-k}\rangle,
% \langle Z_0,\fL_1 V_{n-1}\rangle
%+\langle Z_0,\fL_2 V_{n-2}\rangle,
\label{eq:toeplam}
\\&&
\fL_0V_n-V_n A_0=
-\sum_{k=1}^n\fL_kV_{n-k}
%-\fL_1 V_{n-1}-\fL_2 V_{n-2}
+\sum_{k=1}^nV_{n-k} A_k\,,
\label{eq:topevn}
\\&&\langle Z_0,V_n\rangle=0_m\,.
\label{eq:toeport}
\end{eqnarray}%
\end{subequations}
In applications, the \(m\)~columns of each of these~\(V_n\) contain information about the interactions between longitudinal gradients of the field~\(u\), as felt through the mechanisms encoded in \(\fL_1,\fL_2,\ldots\), and the cross-sectional out-of-equilibrium dynamics encoded in~\(\fL_0\).

\begin{example}[shear dispersion continued] %\label{eg:default}
Via some tedious algebra, here the recursion~\eqref{eqs:toep} gives the well established structures
\begin{eqnarray*}
A_1 &=& 
-\Pe
,\\
V_1 &=& \Pe (
-\sfrac{7}{120}
+\sfrac{1}{4} y^2
-\sfrac{1}{8} y^4)
,\\
A_2 &=& 1
+\sfrac{2}{105} \Pe^2
,\\
V_2 &=& \Pe^2 (
-\sfrac{29}{201600}
-\sfrac{17}{3360} y^2
+\sfrac{17}{960} y^4
-\sfrac{7}{480} y^6
+\sfrac{3}{896} y^8)
,\\
A_3 &=& \sfrac{4}{17325} \Pe^3
%,\\
%V_3 &=& \Pe^3 (\sfrac{16019}{164736000}
%-\sfrac{4223}{8870400} y^2
%+\sfrac{47}{230400} y^4
%+\sfrac{3}{6400} y^6
%-\sfrac{13}{21504} y^8
%+\sfrac{211}{806400} y^{10}
%-\sfrac{3}{78848} y^{12})
%,\\
%A_4 &=& 
%-\sfrac{32}{1126125} \Pe^4
%,\\
%V_4 &=& \Pe^4 (\sfrac{434314477}{197604126720000}
%+\sfrac{391}{195686400} y^2
%-\sfrac{6821723}{193729536000} y^4
%+\sfrac{39901}{1064448000} y^6
%-\sfrac{23}{3686400} y^8
%-\sfrac{7087}{677376000} y^{10}
%+\sfrac{3347}{425779200} y^{12}
%-\sfrac{4867}{2152550400} y^{14}
%+\sfrac{3}{12615680} y^{16})
,
\end{eqnarray*}
and so on.  Then, as the derivation of equation~\eqref{eq:ssmod} asserts, and choosing truncation \(N=3\), the evolution on the local slow subspace becomes 
\begin{equation*}
\D t{}\begin{bmatrix} c_0\\c_1\\c_2\\c_3 \end{bmatrix}
=\begin{bmatrix} 
0&-\Pe&1+\sfrac{2}{105} \Pe^2&\sfrac{4}{17325} \Pe^3 \\
0&0&-\Pe&1+\sfrac{2}{105} \Pe^2\\
0&0&0&-\Pe\\
0&0&0&0
\end{bmatrix}
\begin{bmatrix} c_0\\c_1\\c_2\\c_3 \end{bmatrix}
\end{equation*}
Then the next section proves that, in essence, the first line of this evolution supports the slowly varying model
\begin{equation*}
\D tc\approx 
-\Pe\D xc+(1+\sfrac{2}{105} \Pe^2)\DD xc
+\sfrac{4}{17325} \Pe^3\Dn x3c
\end{equation*}
for the long time dispersion of material along the channel.
\end{example}

\begin{lemma} 
  The recursive equation~\eqref{eq:topevn} is solvable for \(n=1,2,\ldots,N\).
\end{lemma}
\begin{proof} 
   By the choice~\eqref{eq:toeplam}, as seen by considering \(\langle Z_0,\eqref{eq:topevn}\rangle\), the left-hand side of~\eqref{eq:topevn}, using the orthogonality~\eqref{eq:toeport}, becomes
\begin{eqnarray*}&&
\langle Z_0,\fL_0V_n\rangle-\langle Z_0,V_n A_0\rangle
=\langle \tr\fL_0Z_0,V_n\rangle-\langle Z_0,V_n\rangle A_0
\\&&{}
=\langle Z_0\tr A_0,V_n\rangle-0_m A_0
= A_0\langle Z_0,V_n\rangle
= A_00_m
=0_m\,;
\end{eqnarray*}
whereas the right-hand side, also using the orthogonality~\eqref{eq:toeport}, becomes
\begin{eqnarray*}&&
-\sum_{k=1}^n\langle Z_0,\fL_kV_{n-k}\rangle
%-\langle Z_0,\fL_1V_{n-1}\rangle-\langle Z_0,\fL_2V_{n-2}\rangle
+\sum_{k=1}^{n-1}\langle Z_0,V_{n-k} A_k\rangle+\langle Z_0,V_0 A_n\rangle
\\&&{}
=-\sum_{k=1}^n\langle Z_0,\fL_kV_{n-k}\rangle
%-\langle Z_0,\fL_1V_{n-1}\rangle-\langle Z_0,\fL_2V_{n-2}\rangle
+\sum_{k=1}^{n-1}\langle Z_0,V_{n-k}\rangle A_k+\langle Z_0,V_0\rangle A_n
\\&&{}
=-\sum_{k=1}^n\langle Z_0,\fL_kV_{n-k}\rangle
%-\langle Z_0,\fL_1V_{n-1}\rangle-\langle Z_0,\fL_2V_{n-2}\rangle
+\sum_{k=1}^{n-1}0_m A_k+I_m A_n
\\&&{}
=-\sum_{k=1}^n\langle Z_0,\fL_kV_{n-k}\rangle
%-\langle Z_0,\fL_1V_{n-1}\rangle-\langle Z_0,\fL_2V_{n-2}\rangle
+ A_n
=0_m
\end{eqnarray*}
by the choice~\eqref{eq:toeplam}.
\end{proof}

\begin{lemma}
For the homogeneous system~\eqref{eq:odesna}, a basis for the centre subspace is the collective columns of 
\begin{equation}
\vec V_n:=(V_n,\ldots,V_0,0_m,\ldots,0_m)\in\UU^{(N+1)\times m},
\quad n=0,1,\ldots,N\,.
\label{eq:genevec}
\end{equation}
\end{lemma}
\begin{proof} 
  First prove the space spanned by~\(\{\vec V_0,\vec V_1,\ldots,\vec V_N\}\) is invariant. 
Define two important block Toeplitz `matrices': \(\cV:=\begin{bmatrix} \vec V_0&\vec V_1&\cdots&\vec V_N \end{bmatrix}\in\UU^{(N+1)\times m(N+1)}\), that is, 
\begin{equation}
\cV:=\begin{bmatrix} V_0&V_1&V_2&\cdots&V_N
\\0&V_0&V_1&\ddots&\vdots
\\0&0&V_0&\ddots&V_{2}
\\\vdots&\ddots& \ddots& \ddots&V_1
\\0& \cdots &0&0&V_{0}
\end{bmatrix},
\quad\text{and }
\cA:=\begin{bmatrix}  A_0& A_1& A_2&\cdots& A_N
\\0_m& A_0& A_1&\ddots&\vdots
\\0_m&0_m& A_0&\ddots& A_{2}
\\\vdots&\ddots& \ddots& \ddots& A_1
\\0_m& \cdots &0_m&0_m& A_{0}
\end{bmatrix}.
\label{eq:vamat}
\end{equation}
Consider the \(n\)th~block of~\(\cL\vec V_\ell\) (\(n=0,\ldots,N\)):
%it is  \(\fL_0V_{\ell-n}+\fL_1V_{\ell-n-1}+\fL_2V_{\ell-n-2}\), 
it is  \(\sum_{k=0}^{\ell-n}\fL_kV_{\ell-n-k}\) 
which by the recursion~\eqref{eq:topevn} is \(\sum_{k=0}^{\ell-n}V_{\ell-n-k} A_k\), and which in turn is the \((n,\ell)\)th~block of~\(\cV\cA\).
Hence \(\cL\cV\)~is in the space spanned by the columns of~\(\cV\).
Second, moreover, \(\cL\cV=\cV\cA\) so that the eigenvalues corresponding to the eigenspace spanned by~\cV\ are those of~\cA, which by its block Toeplitz structure are the centre eigenvalues of~\(A_0\) repeated \((N+1)\)~times.
Third, the columns of~\cV\ are linearly independent by its block Toeplitz form and the linear independence of the columns of~\(V_0\).
Lastly, there are \(m(N+1)\)~columns in~\cV\ to match the required number of centre eigenvalues of~\cL\ (counted according to multiplicity).
Denote the centre subspace of~\cL, spanned by columns of~\cV, by~\(\EE_c^N\).
\end{proof}

\paragraph{Parametrise evolution on the centre subspace}
To parametrise locations on the centre subspace~$\EE_c^N$ we use the columns of~\cV.  
Using variable name~\(c\) for `centre', let \(c_n\in\RR^m\) for \(n=0,\ldots,N\), and \(\cv:=(c_0,\ldots,c_N)\in\RR^{m(N+1)}\).
Then parametrise positions on the centre subspace as \(\uv=\cV\cv\).
In applications, the variables~\(c_n\) typically measure the \(n\)th~derivative in the longitudinal direction of the macroscopic components in~\(V_0\) at station~\(X\) at time~\(t\).

Evolution on the centre subspace~\(\EE_c^N\) is then characterised by evolving~\(\cv(t)\).  
From the autonomous system~\eqref{eq:odesna},
\(\cV\D t{\cv} =\D t{\uv}
=\cL\uv =\cL\cV\cv =\cV\cA\cv
\)\,.
Since the columns of~\cV\ are linearly independent it follows that
\begin{equation}
\D t{\cv}=\cA{\cv}\,.
\label{eq:ssmod}
\end{equation}
which then governs the evolution~\eqref{eq:odesna} within the centre subspace~\(\EE_c^N\). 
This system of \ode{}s forms a long term model of the dynamics of the autonomous~\eqref{eq:odesna}. 
These \ode{}s have no approximation, only neglect of transients: by the decay of the stable modes we know that all solutions of the autonomous~\eqref{eq:odesna} approach solutions of~\eqref{eq:ssmod} exponentially quickly.  
The decay is like~\(e^{-\gamma t}\) for any rate \(\gamma\in(\alpha,\beta)\) because of possible effects due to the generalised eigenvectors of the non-normal~\cL.

\subsection{How do we project uncertain forcing?}
\label{sec:hruf}

Our aim is not to model the autonomous~\eqref{eq:odesna}, but the exact system~\eqref{eq:odesn} with its  uncertain `forcing' by the coupling~\(\rv\) with neighbouring locales.
%If the remainder was independently specified, then straightforward arguments project it onto the centre subspace.
%But the remainder involves the unknowns through coupling with neighbouring locales: nonetheless even in this case a coordinate transform exists that justifies the straightforward projection.
Let's proceed to project the uncertain forcing as if it was arbitrary.

\paragraph{Change basis to centre and stable variables}
Write \(\uv=\cV\cv+\cW\dv\) where the centre variables~\(\cv\) parametrise the centre subspace, and the variables~\(\dv\) parametrise the stable subspace.
Just like~\cV, the (block Toeplitz) operator~\cW\ is associated with the following properties:
\begin{itemize}
\item \cW\ spans the stable subspace~\(\EE_s^N\) of~\cL;
\item there exists a (block Toeplitz) operator~\(\cB:\EE_s^N\to\EE_s^N\) such that \(\cL\cW=\cW\cB\) and all eigenvalues of~\cB\ have real part\({}\leq-\beta\);
%\item \cW\ is orthogonal to~\cZ, \(\llangle \cZ,\cW\rrangle=0\);
\item there exist projection operators~\cP\ and~\cZ\ such that \(\llangle \cP,\cW\rrangle=I\), \(\llangle \cP,\cV\rrangle=0\), \(\llangle \cZ,\cW\rrangle=0\), and \(\llangle \cZ,\cV\rrangle=I\).
\end{itemize}
Then writing the `forced' system~\eqref{eq:odesn} in separated variables \(\cv(X,t)\) and~\(\dv(X,t)\), by projecting with \(\llangle \cZ,\rrangle\) and~\(\llangle \cP,\rrangle\) respectively, we deduce
%\ajr{Need to change r and R earlier?? or here??}
\begin{subequations}\label{eqs:ecde}%
\begin{eqnarray}&&
\D t{\cv}=\cA\cv+\Rv
\quad\text{where }\Rv=\llangle\cZ,\rv\rrangle\in\RR^{m(N+1)},
\label{eq:ecdec}
\\&&
\D t{\dv}=\cB\dv+\rvd
\quad\text{where }\rvd=\llangle\cP,\rv\rrangle.
\label{eq:ecdeq}
\end{eqnarray}%
\end{subequations}

Now consider the stable variables.
Since \cL\ generates a continuous semigroup, so does its restriction~\cB, and so we rewrite~\eqref{eq:ecdeq} in the integral equation form
\begin{equation}
\dv(t)
=e^{\cB t}\dv(0)+\int_0^t e^{\cB (t-s)}\rvd(s)\,ds
=e^{\cB t}\dv(0)+e^{\cB t}\star\rvd,
\label{eq:qfntime}
\end{equation}
as convolutions \(f(t)\star g(t)=\int_0^tf(t-s)g(s)\,ds\)\,.  Since all eigenvalues of~\cB\ have real part\({}\leq-\beta\), then for some decay rate~\(\gamma\in(\alpha,\beta)\)
\begin{equation}
\dv(t)=e^{\cB t}\star\rvd+\Ord{e^{-\gamma t}},
\quad\text{written}\quad
\dv(t)\simeq e^{\cB t}\star\rvd
\label{eq:rforcedq}
\end{equation}
upon invoking the following definition that \(f\simeq g\) to mean that \(f\)~and~\(g\) are equal apart from ignored exponentially rapid decaying transients. 
\begin{definition}
Define \(f(t)\simeq g(t)\) to mean \(f-g=\Ord{e^{-\gamma t}}\) as \(t\to\infty\) for some exponential rate \(\alpha<\gamma <\beta\)\,.
\end{definition}  
Consequently, equation~\eqref{eq:rforcedq} determines how the local stable variables~\(\dv\) are forced by the coupling with neighbouring stations via the remainder effects in~\(\rvd\).

\paragraph{The centre subspace dynamics with remainder}
Define the amplitude field of slowly varying solutions by the projection
\begin{equation}
c(x,t):=\langle Z_0,u(x,y,t)\rangle,
\label{eq:cssf}
\end{equation}
which as yet is distinct from the local centre variables~\(\cv\).
In order to discover how the amplitude field~\(c(x,t)\) evolves, our task is to now relate the field~\(c(x,t)\) to the local centre subspace variables~\(\cv\).
Recall from~\eqref{eq:utrtn} that expanded about the station~\(X\) the original field
\begin{equation*}
u(x,y,t)=u_0(X,y,t)+u_1(X,y,t)\pt1+\cdots+u_N(X,x,y,t)\pt N\,.
\end{equation*}
By projecting this expression, the centre field
\begin{equation}
c(x,t)=\sum_{n=0}^{N-1}\langle Z_0,u_n(X,y,t) \rangle\pt{n}
+\langle Z_0,u_N(X,x,y,t) \rangle\pt N\,.
\label{eq:csumzu}
\end{equation}
But \(\uv=\cV\cv+\cW\dv \simeq \cV\cv+\cW e^{\cB t}\star\rvd\).
Since equation~\eqref{eq:toeport} sets \(\langle Z_0,V_n\rangle=0\) for all \(n\neq0\), consequently
\begin{equation*}
\langle Z_0,u_n\rangle \simeq  
c_n(X,t)+\langle Z_0,\cW_{n:} e^{\cB t}\star\rvd\rangle
\quad\text{for }n=0,1,\ldots,N\,,
\end{equation*}
where \(\cW_{n:}\)~denotes the \(n\)th block-row of operator~\cW.
Thus equation~\eqref{eq:csumzu} becomes 
\begin{equation}
c(x,t) \simeq \sum_{n=0}^{N}\left[c_n(X,t)
+\langle Z_0,\cW_{n:} e^{\cB t}\star\rvd\rangle\right]\pt{n}\,.
\label{eq:csumc}
\end{equation}
This relates the centre field to the local centre variables: there is no approximation except the neglect of rapid transients.

Identity~\eqref{eq:csumc} together with evolution~\eqref{eq:ecdec} leads to the \pde\ governing the centre field.
Differentiating~\eqref{eq:csumc} \(\nu\)~times with respect to~\(x\), keeping constant time~\(t\) and station~\(X\), gives
\begin{eqnarray}
\left.\Dn x\nu c\right|_{x=X} &\simeq& 
\left.\sum_{n=\nu}^{N}\left[c_{n}
+\langle Z_0,\cW_{n:} e^{\cB t}\star\rvd\rangle\right]\pt{(n-\nu)}
\right|_{x=X}
\nonumber\\&&{}
=c_{\nu}
+\langle Z_0,\cW_{\nu:} e^{\cB t}\star\rvd\rangle
\quad\text{at }x=X\,.
\label{eq:rccnu}
\end{eqnarray}
Derivatives of the remainder factor~\(\rvd\) do not occur as the remainder is independent of~\(x\) by the necessary evaluation at \(x=X\) in its Definition~\eqref{eq:remain}:
sufficient information about spatial gradients are already encoded into the remainder through the derivatives \(\Dn xk{u_N}\) that appear in~\eqref{eq:remain}.
Consider the \(\nu=0\) instance of the identity~\eqref{eq:rccnu} and differentiate with respect to time to give
\begin{eqnarray}
\D tc&\simeq&\D t{c_0}
+\D t{}\langle Z_0,\cW_{0:} e^{\cB t}\star\rvd\rangle
\nonumber\\
&&{}=\cA_{0:}\cv+r_0
+\langle Z_0,\cW_{0:} \D t{}[e^{\cB t}\star\rvd]\rangle
\quad\text{[by~\eqref{eq:ecdec}]}
\nonumber\\
&&{}=\sum_{n=0}^NA_{n}c_n+r_0
+\langle Z_0,\cW_{0:} \cB e^{\cB t}\star\rvd\rangle
+\langle Z_0,\cW_{0:} \rvd\rangle.
\quad\text{[by \eqref{eq:qfntime}]}
\qquad\label{eq:cct}
\end{eqnarray}
From the identity~\eqref{eq:rccnu}, replace the local centre variables~\(c_n\) in favour of spatial gradients of the amplitude field~\(c\) to derive from~\eqref{eq:cct} that the amplitude field must satisfy the \pde\
\begin{equation}
\D tc\simeq\sum_{n=0}^N A_n\Dn xnc +\rho \,,
\label{eq:cpde}
\end{equation}
where the remainder term
\begin{equation}
\rho=r_0
+\langle Z_0,\cW_{0:} \cB e^{\cB t}\star\rvd\rangle
+\langle Z_0,\cW_{0:} \rvd\rangle
-\sum_{n=0}^NA_n\langle Z_0,\cW_{n:} e^{\cB t}\star\rvd\rangle.
\label{eq:pderemain}
\end{equation}
The \pde~\eqref{eq:cpde} applies at all stations~\(X\) in the domain~\(\XX\).
Strictly, the `\pde'~\eqref{eq:cpde} is actually a coupled differential-integral equation: the dynamics at each station~\(X\) being coupled by the gradients and their history convolution integrals occurring within the remainder~\eqref{eq:pderemain}.
But when the uncertain remainder term is negligible, as in slowly varying regimes where the remainder~\(\rho\) is~\Ord{\Dn x{N+1}u}, then equation~\eqref{eq:cpde} reduces to the longitudinal \pde\ closure~\eqref{eq:dcdta}.  
This completes the argument that establishes Proposition~\ref{thm:svpde}.

\subsection{Application: pattern diffusion in space}
\label{sec:egpes}

The Swift--Hohenberg equation is a well known, prototypical \pde\ for studying issues in pattern formation and evolution: nondimensionally it is 
\(\D tu=ru-(1+\delsq)^2u-u^3\) \cite[e.g.]{Cross93}.
Here just consider the long time evolution of small amplitude solutions of the Swift--Hohenberg system exactly at the borderline of instability and in one space dimension: a field~\(u(\ix,t)\) satisfies the linear \pde
\begin{equation}
\D tu=-(1+\partial_{\ix\ix})^2u
\label{eq:shpdem}
\end{equation}
on a domain~\(\XX\) of large extent in~\(\ix\).
The slow marginal modes are \(u\propto e^{\pm i\ix}\).
However, there are an infinity of modes arbitrarily close to marginal: the modes \(u\propto e^{\pm ik\ix}\) with spatial wavenumbers~\(k\) near one.
This infinity of modes means that physically we see the marginal modes~\(e^{\pm i\ix}\) being modulated in space over large distances.
The modelling challenge for this subsection is to establish a new approach that rigorously models the dynamics of these modulation patterns.

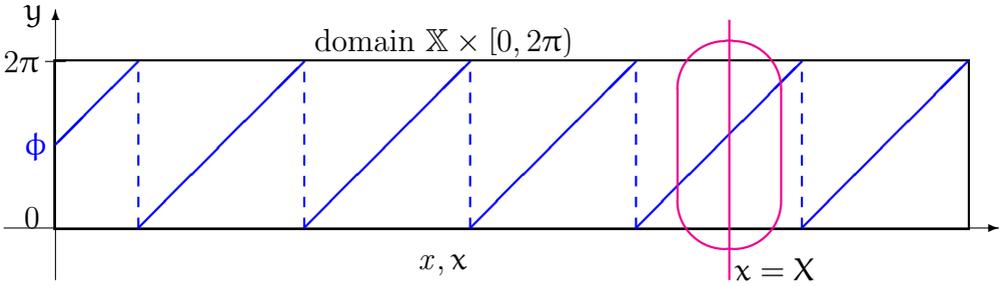
\begin{figure}
\centering
\setlength{\unitlength}{0.01\linewidth}
\begin{picture}(96,26)
%\put(0,0){\framebox(96,26){}}
\put(0,5){\vector(1,0){96}}
\put(40,1){\(\ix,\fx\)}
\put(5,0){\vector(0,1){26}}
\put(2,25){\(y\)}
\put(0,20){\(2\pi\)}\put(4,21){\line(1,0){2}}
\put(2,5){\(0\)}
\put(30,22){domain \(\XX\times[0,2\pi)\)}
\thicklines
\put(5,5){\framebox(88,16){}}
\color{blue}
\multiput(13,5)(16,0)5{
  \put(0,0){\line(1,1){16}}
  \multiput(0,0)(0,2)8{\line(0,1){1}}
  }
  \put(13,21){\line(-1,-1){8}}
  \put(2,12){\(\phi\)}
\color{magenta}
\put(70,0){\line(0,1){25}\,\color{black}\(\fx=X\)}
\put(70,13){\oval(10,20)}
\end{picture}
\caption{cylindrical domain of the embedding \pde~\eqref{eq:shembed} for field \(\fu(\fx,y,t)\).  Obtain solutions of the linear marginal Swift--Hohenberg \pde~\eqref{eq:shpdem} on the blue line as \(u(\ix,t)=\fu(\ix,\ix+\phi,t)\) for any constant phase~\(\phi\).}
\label{fig:pattensemble}
\end{figure}

Let's embed the \pde~\eqref{eq:shpdem} in a larger problem.
As indicated schematically in Figure~\ref{fig:pattensemble},
and in terms of a notionally new longitudinal variable~\(\fx\) and new phase variable~\(y\), consider a new field~\(\fu (\fx,y,t)\) satisfying the \pde
\begin{equation}
\D t\fu =-(1+\partial_{yy}+2\partial_{y\fx}+\partial_{\fx\fx})^2\fu ,
\quad \text{for }(\fx,y)\in \XX\times[0,2\pi),
\label{eq:shembed}
\end{equation}
where the field~\(\fu \) is \(2\pi\)-periodic in~\(y\).
Given any solution~\(\fu \) of the \pde~\eqref{eq:shembed}, elementary calculus shows that, for any chosen fixed phase~\(\phi\) and using that \(\fu \)~is \(2\pi\)-periodic in~\(y\), the field \(u(\ix,t)=\fu (\ix,\ix+\phi,t)\) is a solution of the linear marginal Swift--Hohenberg \pde~\eqref{eq:shpdem}, also indicated in Figure~\ref{fig:pattensemble}.
Thus modelling of the dynamics of the \pde~\eqref{eq:shembed} immediately leads to models for the dynamics of the linear marginal Swift--Hohenberg \pde~\eqref{eq:shpdem}.
Crucially, the rigorous embedding here replaces the heuristic multiple space and time scale assumptions traditionally employed in asymptotic analysis \cite[e.g.]{Cross93, Vandyke87}.

The techniques and results of this section apply to the embedding \pde~\eqref{eq:shembed}.
The \pde~\eqref{eq:shembed} is of the form of the general \pde~\eqref{eq:upde} with
\begin{eqnarray}&&
\fL_0=-(1+\partial_{yy})^2,\quad
\fL_1=-4(\partial_y+\partial_{yyy}),\quad
\fL_2=-2-6\partial_{yy},
\nonumber\\&&
\fL_3=-4\partial_y,\quad
\fL_4=-1\,.
\label{eq:sheln}
\end{eqnarray}
The basic eigenproblem at a station \(\fx=X\) is then \(\lambda v=\cL_0 v=-(1+\partial_{yy})^2v\).
Using the \(2\pi\)-periodicity in cross-sectional variable~\(y\), the eigenfunctions are~\(v_k=e^{\pm iky}\) for index \(k=0,1,2,3,\ldots\).
The corresponding eigenvalues are \(\lambda_k=-(1-k^2)^2\) giving a discrete spectrum of~\(\{-1,0,-9,-64,\ldots\}\).
Thus there are two eigenvalues of zero corresponding to the basic spatial pattern~\(e^{\pm iy}\), and all other eigenvalues are\({}\leq-\beta=-1<0\).
That is, the \pde~\eqref{eq:shembed} satisfies Assumption~\ref{ass:spec}.

Consequently, Proposition~\ref{thm:svpde} asserts there are models of the \pde~\eqref{eq:shembed} in the form~\eqref{eq:dcdta} that emerge exponentially quickly and to a quantifiable error.
Interpreting these results for the field \(u(\ix,t)=\fu (\ix,\ix+\phi,t)\), for any phase~\(\phi\), leads to predictions about the pattern evolution of the linear marginal Swift--Hohenberg \pde~\eqref{eq:shpdem}.%
\footnote{A similar argument could be given for the modelling of wave modulation.
Such an approach would discretise the wave spectrum into distinct oscillating modes and one would choose one wavenumber on which to base a subcentre manifold (as defined by \cite{Sijbrand85}).
However, there would be no straightforward guarantee that the described wave modulation would emerge from general initial conditions.}

The model here is particularly straightforward.
Let's use the complex exponentials~\(e^{\pm iy}\) as the two basis functions in~\(V_0\) to span the basic slow subspace corresponding to the eigenvalues of \(A_0=0_2\).
We need an inner product, Definition~\ref{def:adj}, so introduce the mean \(\langle z,v\rangle=\frac1{2\pi}\int_0^{2\pi} zv\,dy\).
Then the adjoint eigenfunctions \(Z_0=V_0\).
Recursively solving equation~\eqref{eqs:toep} leads to all \(V_n=0\) for \(n\geq 1\) (that is, here \cV~is block-diagonal).
Further, the evolution on the slow subspace is determined by
\(A_2=\diag(4,4)\), \(A_3=\diag(-4i,4i)\), \(A_4=\diag(-1,-1)\) and all others zero.
Hence Proposition~\ref{thm:svpde} assures us that to an error quantified by some remainder terms~\eqref{eq:pderemain}, the slow dynamics 
\begin{equation}
\D t{c_\pm}\simeq 4\DD \fx{c_\pm} \mp 4i\Dn \fx3{c_\pm} -\Dn \fx4{c_\pm}
\label{eq:shdiff}
\end{equation}
will emerge exponentially quickly from general initial conditions.
These \pde{}s, of course, match the dispersion relation of the marginal Swift--Hohenberg \pde~\eqref{eq:shpdem} near the critical wavenumbers.

The leading order model is that the spatial pattern diffuses: \(\D t{c_\pm}\simeq 4\DD \fx{c_\pm}\)\,.  The corresponding, slow subspace, embedding field is \(\fu(\fx,y,t)\simeq c_+(\fx,t)e^{iy}+c_-(\fx,t)e^{-iy}\) which predicts an emergent physical field of \(u(\ix,t)=\fu(\ix,\ix+\phi,t)\simeq c_+(\ix,t)e^{i\ix+i\phi}+c_-(\ix,t)e^{-i\ix-i\phi}\) for any constant phase~\(\phi\) (the phase~\(\phi\) could be absorbed into~\(c_\pm\)).

Initial conditions for the embedding \pde~\eqref{eq:shembed} appear paradoxical.
On the one hand, the embedding \pde~\eqref{eq:shembed} describes dynamics along lines \(y=\fx+\phi\pmod{2\pi}\) which are completely decoupled for different phase~\(\phi\): consequently, one could have completely disparate solutions on neighbouring~\(\phi\).
On the other hand, the spectrum of the operator~\(\fL_0\) appears to guarantee a rapid relaxation to an equilibrium structure with basis~\(e^{\pm iy}\).
This apparent paradox is rationalised by the uncertain coupling between neighbouring stations~\(X\):
a rapid relaxation to a smooth slowly varying field is only guaranteed to occur for initial conditions where the uncertain coupling in the remainder terms are small enough;
that is, only for initial conditions which are sufficiently smooth.
It is only when the ensemble of solutions over all phases~\(\phi\) are smooth enough that the errors in the modelling~\eqref{eq:shdiff} will be acceptable.
Thus we can only have acceptable errors when the ensemble is chosen to be not disparate.
The remainder terms~\eqref{eq:pderemain} quantify this error for us.

\section{Nonlinear heat exchanger modelling}
\label{sec:nhem}

Reconsider the heat exchanger of Figure~\ref{fig:hescheme}. 
Now we include a nonlinear (quadratic) reaction in each pipe.
This section uses this example to introduce how to adapt the approach of previous sections to model nonlinear dynamics in cylindrical domains.
Section~\ref{sec:mndcd} develops these ideas to nonlinear theory for general systems.

In the nonlinear heat exchanger suppose the governing \pde{}s are
\begin{equation}
\D ta=+U\D xa+\frac R2(b-a)-\sigma a^2,\quad
\D tb=-U\D xb+\frac R2(a-b)+\sigma b^2,
\label{er:hedim}
\end{equation}
for flow to the left and right at equal and opposite velocities~$\pm U$, for some inter-pipe exchange at rate~$R$, and some quadratic reaction in one pipe and corresponding quenching in the other pipe, both of strength~\(\sigma\).
Non-dimensionalise space and time by choosing the reference time~$1/R$, the reference length~$U/R$, and reference field value~\(R/(2\sigma)\) to give the non-dimensional \pde{}s
\begin{equation}
\D ta=+\D xa+\rat12(b-a)-\rat12a^2,\quad
\D tb=-\D xb+\rat12(a-b)+\rat12b^2.
\label{er:hepde}
\end{equation}
These \pde{}s would be modelled with boundary conditions, such as $a={}$hot at $x=L$, and $b={}$cold at $x=0$.
However, we leave appropriate boundary conditions for further research \cite[e.g.]{Roberts92c}, and here focus on the evolution in the interior.
This section finds the model that in the interior the mean temperature, $c(x,t)=\rat12(a+b)$ satisfies a Burgers'-like \pde 
\begin{equation}
\D tc\approx-2c\D xc+\DD xc+\rat12c^3;
\label{eq:dcdtblpde}
\end{equation}
further, the aim is to certify this approximation with a novel error estimate and as the emergent dynamics.

To make the analysis more straightforward, let's transform the non-dimensional \pde{}s~\eqref{er:hepde} to mean and difference fields~\eqref{eq:meandiff};
%\begin{equation}
%c(x,t)=\rat12(a+b),\quad d(x,t)=\rat12(a-b);
%\label{er:meandiff}
%\end{equation}
that is, \(a=c+d\) and \(b=c-d\)\,.
Rearranging the mean and difference of the \pde{}s~\eqref{er:hepde} gives the equivalent \pde\ system
\begin{equation}
\D tc=\D xd-cd\,,\quad
\D td=-d+\D xc-\rat12(c^2+d^2).
\label{er:hecann}
\end{equation}
In this form we readily see that the difference field~$d$ tends to decay exponentially quickly, albeit with the quadratic reaction forcing some difference, but that interaction between gradients of the mean and difference fields generates other effects.

\subsection{In the interior}

Fix upon any station along the pipe, say at $x=X$\,, and consider the mean and difference fields in the vicinity of $x=X$.  
In the vicinity express the fields as
%\footnote{For example, consider the function \(c(x,t)=1/(1+xt)\).  A Taylor Remainder Theorem about \(X=0\) could give \(c(x,t)=1-tx+t^2x^2/(1+\xec t)^3\) for some \(0 \gtrless\xec\gtrless x\).
%That is, after rearrangement \(\xec=[(1+xt)^{1/3}-1]/t\).  
%Then \(\xec\sim x/3\) as \(t\to 0\)\,;
%but also the derivative \(\D x\xec=\rat13(1+xt)^{-2/3}\to\infty\) as \(xt\to-1\).}
\begin{subequations}\label{ers:trt}%
\begin{eqnarray}
c(x,t)&=&c_0(X,t)+c_1(X,t)\pt1 +c_2(X,x,t)\pt2,
\label{er:ctrt}
\\
d(x,t)&=&d_0(X,t)+d_1(X,t)\pt1 +d_2(X,x,t)\pt2,
\label{er:dtrt}
\end{eqnarray}%
\end{subequations}
where by Taylor's Remainder Theorem $c_n(X,t)=\Dn xnc|_{x=X}$ except for the case $n=2$ where $c_2(X,x,t)=\Dn x2c|_{x=\xec}$ for some unknown~\(\xec\) satisfying \(X \lessgtr\xec(X,x,t) \lessgtr x\).
Similarly, $d_n(X,t)=\Dn xnd|_{x=X}$ except the case $d_2(X,x,t)=\Dn x2d|_{x=\xed}$ for some unknown~\(\xed\) satisfying \(X \lessgtr\xed(X,x,t) \lessgtr x\).
That is, \(c_2\)~and~\(d_2\) are second order derivatives but evaluated at some nearby but uncertain and typically moving locations (although soon we will evaluate them also at \(x=X\) and consequently thereafter \(c_2,d_2\) and their derivatives only depend upon \(X\)~and~\(t\)).

For definiteness and reasonable conciseness, in this section we truncate the Taylor series approximation to second order---the case \(N=2\).
Appendix~\ref{sec:camnhe} lists computer algebra code that not only generates the intermediate steps and  results here, but also does so for any truncation of the Taylor series---any~\(N\leq9\) was tested. 

\paragraph{Local ODEs}
As before, substitute the Taylor expansions~\eqref{ers:trt} into the governing \pde{}s~\eqref{er:hecann}.
The computed residuals of the \pde{}s are exact everywhere. 
But they are useful near the section \(x=X\)\,.
To find a set of linearly independent equations just differentiate the residuals and evaluate at $x=X$\,.
The first of the \pde{}s~\eqref{er:hecann} give three \ode{}s for the $c_n$~coefficients:
\begin{subequations}%
\begin{eqnarray}&&
\dot c_{0}=d_{1}-c_0d_0, \label{sq:c0}\\&&
\dot c_{1}=d_{2}-c_0d_1-c_1d_0, \label{sq:c1}\\&&
\dot c_{2}=3d_{2x}-c_0d_2-2c_1d_1-c_2d_0. \label{sq:c2}
\end{eqnarray}%
\end{subequations}
Analogously, the second of the \pde{}s~\eqref{er:hecann} give three  \ode{}s for the $d_n$~coefficients:
\begin{subequations}%
\begin{eqnarray}&&
\dot d_{0}=-d_0+c_{1}-\rat12(c_0^2+d_0^2), \label{sq:d0}\\&& 
\dot d_{1}=-d_1+c_{2}-c_0c_1-d_0d_1, \label{sq:d1}\\&&
\dot d_{2}=-d_2+3c_{2x}-c_1^2-c_0c_2-d_1^2-d_0d_2. \label{sq:d2}
\end{eqnarray}%
\end{subequations}
In this set of six coupled \ode{}s, and hereafter, the variables~\(c_2\) and~\(d_2\) are only a function of \(X\)~and~\(t\) as they have been evaluated at \(x=X\) (the uncertain locations~\xec\ and \xed\ have also been squeezed to \(\xec=\xed=X\) by this evaluation).  The uncertainty only appears via the occurrence of the coupling derivatives~\(c_{2x}\) and~\(d_{2x}\) at the station~\((X,t)\).

Define the state vector $\uv =(c_0,d_0,c_1,d_1,c_2,d_2)$ and group these six \ode{}s into the matrix-vector system, of the form $\de t{\uv }=\cL\uv +\vec f(\uv )+\vec r(t)$,
\begin{equation}
\de t{\uv }=\underbrace{\begin{bmatrix} 
0&0&0&1\\ 0&-1&1&0\\
&&0&0&0&1\\ &&0&-1&1&0\\
&&&&0&0\\ &&&&0&-1\\
\end{bmatrix}}_{\cL}\uv 
+\underbrace{\begin{bmatrix} -c_0d_0\\-\rat12(c_0^2+d_0^2)
\\-c_0d_1-c_1d_0\\-c_0c_1-d_0d_1
\\c_0d_2-2c_1d_1-c_2d_0 \\ -c_1^2-c_0c_2-d_1^2-d_0d_2 \end{bmatrix}}_{\vec f(\vec u)}
+\underbrace{\begin{bmatrix} 0\\0\\0\\0\\
3d_{2x}\\3c_{2x} \end{bmatrix}}_{\vec r}
\label{er:odes}
\end{equation}
where $d_{2x}$~and~$c_{2x}$ give some definite but uncertain inter-station coupling. 
%The functions $d_{2x}$~and~$c_{2x}$ are part of the closure problem:  the derivatives $d_{2x}$~and~$c_{2x}$ couple the dynamics at a station~$X$ with the dynamics at neighbouring stations.
%It is by treating the terms $d_{2x}$~and~$c_{2x}$ as `uncertain' inputs into the local dynamics that we make the vast simplification in reducing the problem from one of an infinite dimensional dynamical system to a tractable (small) finite dimensional system.
%However, the system only appears to be finite dimensional, strictly the system is coupled to neighbouring locales and thence to the whole domain, \(X\in(0,L)\).
Crucially this transformation pushes the coupling to as high order as required, is carried through the analysis, and then estimates an error.

\subsection{The slow manifold emerges}
\label{sec:sme}

The system~\eqref{er:odes} appears in the form of a `forced' nonlinear system.
So our first task is to understand the linear homogeneous system obtained by omitting the nonlinearity and the `forcing' (although here the the `forcing' is actually coupling with neighbouring dynamics).
Subsequently, we invoke centre manifold theorems to deduce existence and emergence of a slow manifold model for the `forced' nonlinear dynamics.

\paragraph{Slow subspace}
The linearised homogeneous system~\eqref{er:odes} is upper triangular (also block toeplitz), so its eigenvalues are the diagonal of~\(\cL\), namely $0$~and~$-1$ each with multiplicity three.
The eigenvalues~$-1$ indicate that after transients in time, \Ord{e^{-\gamma t}} for any \(\gamma\in(0,1)\), the evolution lies on the 3D slow subspace of the zero eigenvalue.
Two eigenvectors corresponding to the zero eigenvalue are straightforward to find, namely
\begin{equation*}
\vec v_0=(1,0,0,0,0,0),\quad \vec v_1=(0,1,1,0,0,0).
\end{equation*}
Another eigenvector is generalised and come from solving $\cL \vec v_2=\vec v_0$ (and more generalised eigenvectors in the cases of truncations \(N>2\)):
\begin{align*}&
\vec v_2=(0,0,0,1,1,0).
\end{align*}
Letting the matrix $\cV =\begin{bmatrix} \vec v_0&\vec v_1&\vec v_2 \end{bmatrix}$, the slow subspace is then $\uv =\cV \vec c$ where we use $\vec c=(c_0,c_1,c_2)$ to directly parametrise the slow subspace (empowered by the form chosen for the eigenvectors~$\vec v_k$); denote the slow subspace by~\(\EE^2_{c}(X)\).
On this slow subspace~\(\EE^2_{c}(X)\) the evolution is guided by a toeplitz matrix, namely
\begin{equation}
\de t{\vec c}=\cA\vec c=\begin{bmatrix} 0&0&1\\ 
0&0&0\\ 0&0&0 \end{bmatrix}\vec c\,.
\label{er:linhomo}
\end{equation}

On this slow subspace~\(\EE^2_{c}(X)\), from the eigenvectors via $\uv =\cV \vec c$\,, the difference components 
\begin{equation*}
\vec d=(d_0,d_1,d_2)=(c_1,c_2,0).
\end{equation*}

\paragraph{There exists an emergent, infinite dimensional, slow manifold}
%Let's invoke theory by \cite{Aulbach96}.
%So far we have used \XX~to denote the open set of the physical domain (primarily because in linear dynamics we could address the dynamics at any station~\(X\) `independent' of other locales).
%In nonlinear systems we address the dynamics as a collective whole across all relevant space---all relevant stations.
%Thus we slightly change~\XX\ to now denote an open subset of the physical domain where boundary layers and internal shocks are excised from~\XX\ (if the domain is \(L\)-periodic, then \XX~could still be the entire spatial domain).
%This qualitative redefinition is to ensure that the uncertain coupling is `small' for all stations \(X\in\XX\)\,:
%the smaller a desired error in the modelling, the more restrictive the physical domain~\XX\ over which the modelling is valid.
%
Consider the system~\eqref{er:odes} over a set of stations~\XX: then system~\eqref{er:odes} over all stations \(X\in\XX\) is well-posed and autonomous, except for non-autonomous forcing across the boundary~\(\partial\XX\).
The system~\eqref{er:odes} has two closed \cL-invariant subspaces with a spectral gap: for example, \(\bigtimes_{X\in\XX}\EE^2_{c}(X)\)~is the slow subspace.
Thus the general Proposition~\ref{thm:gpdecm} of section~\ref{sec:mndcd} applies to ensure the existence of an emergent slow manifold \(\vec d=\vec h(\vec c;X)\), denoted~\(\cM_c^2\), representing the slow dynamics across the domain~\XX.

%The quadratic nonlinearity~\(f\) in system~\eqref{er:odes} is~\(C^{\infty}\) but not Lipschitz nor bounded, and we want to consider the inter-station coupling~\(\vec r\) as also a perturbing `nonlinearity'.
%To this end, introduce~\(\delta\) to parametrise both a nonlinear cut-off \cite[e.g.]{Vanderbauwhede88, Haragus2011, Chicone2006, Mielke86} and a `low pass filter', and let the subscript~\(x\) on coupling \(c_{2x}\) and~\(d_{2x}\) denote the low pass filtered version of the spatial derivatives involved in the coupling.
%Defining the filter and parameter appropriately we may ensure \(c_{2x},d_{2x}=\Ord{\delta |\vec u|}\) and are~\(C^\infty\).
%Consequently, adding parameter~\(\delta\) to the list of dynamical variables, and adjoining \(\de t\delta=0\) to the system~\eqref{er:odes}, then nonlinearities and the inter-station coupling~\(\vec r\) both appear as a \(C^\infty\), bounded and Lipschitz.
%% My gamma could be Gallay's beta.
%Then Theorem~6.1 of \cite{Aulbach96} applies to the system~\eqref{er:odes} 

Proposition~\ref{thm:gpdecm} assures us the slow manifold~\(\cM_c^2\) exists and emerges provided the resultant model is restricted to domains~\XX\ where the coupling derivatives \(c_{2x}\) and~\(d_{2x}\) are small enough.
It is in only this statement that we need make the slowly varying assumption of multiscale modelling.
This slowly varying restriction need not be imposed on the construction of the slow manifold model (section~\ref{sec:ufnsm}); it only need be a restriction on the domain~\XX\ to which the model is applied.
Thus the slowly varying nature only need restrict the regime of use of the model, not its construction.

%\paragraph{The slow manifold emerges}
%Recall that section~\ref{sec:tdnf} derived an exact, time dependent, coordinate transformation, \((\Cv,\Dv)\leftrightarrow(\cv,\dv)\), that decoupled slow and fast variables~\((\Cv,\Dv)\).
%Theory for non-autonomous systems asserts we can do the same for nonlinear systems \cite[Thm~4.1]{Aulbach2000}, even for random systems \cite[e.g.]{Arnold03, Roberts06k}: regarding the coupling \(c_{2x}\) and~\(d_{2x}\) as a time dependent input to the local dynamics at a station, a local coordinate transform exists so that the local fast variables evolve according to 
%\begin{equation*}
%\de t{\Dv}=\left[\cB+G(\Cv,t)\right]\Dv 
%\quad\text{where }
%\cB=\begin{bmatrix} -1&0&-1\\0&-1&0\\ 0&0&-1 \end{bmatrix},
%\end{equation*}
%and matrix \(G=\Ord{|\Cv|}\) (invoking the low band pass filter of the coupling).
%As \(\Dv=\vec 0\) is invariant and tangent to the slow subspace, it must give the slow manifold.
%Define the Lyapunov function \(\cE:=\rat12D_0^2+\rat12D_1^2+\rat12D_2^2\), then it is straightforward algebra to deduce that the time derivative
%\begin{equation*}
%\de t\cE\leq \tr\Dv[-\rat12I+G(\Cv,t)]\Dv\,.
%\end{equation*}
%By continuity of~\(G\), there exists a finite neighbourhood of \(\Dv=\vec 0\) such that \(\de t\cE\) is negative and hence \(\Dv\to0\) as \(t\to\infty\).
%%Then very generally, \cite{Arnold03} [Theorem 8.4.1(ii)--(iii)] proved that indeed \(\Dv=\vec 0\) is attractive and gives a slow manifold in the original system.
%That is, the slow manifold emerges from at least some finite neighbourhood of initial conditions.

\subsection{Uncertainly coupled nonlinear slow manifold}
\label{sec:ufnsm}

We need to construct the emergent slow manifold of the nonlinear local system~\eqref{er:odes} when the system is `forced' by the uncertain coupling and `bent' by the nonlinearity.

The slow manifold is to be constructed to some order in the variables and the uncertain forcing.
In principle, we could construct the slow manifold of the system~\eqref{er:odes} to arbitrarily high order and to a huge variety of relative weights of variables \cite[e.g.]{Zhenquan00}.
In practice, we want to construct an approximation consistent with the Taylor series truncation~\eqref{ers:trt}, and consistent with the notion that the solution fields \(c(x,t)\) and~\(d(x,t)\) are slowly varying in space.
To correspond to slow space variations, define the state vector~\(\uv \) to have amplitude (not a norm\footnote{The amplitude~\(\nuv\) is not a norm as it fails the absolute homogeneity property.})
\begin{equation}
\nuv:=|c_0|+|d_0|+|c_1|^{1/2}+|d_1|^{1/2}+|c_2|^{1/3}+|d_2|^{1/3}.
\label{er:unorm}
\end{equation}
Some consequences of this definition are that  
\begin{equation}
c_0,d_0=\Ord{\nuv},\quad
c_1,d_1=\Ord{\nuv^2},\quad 
c_2,d_2=\Ord{\nuv^3},
\quad\text{as }\nuv\to 0\,.
\label{er:cdmag}
\end{equation}
Because \(c_n,d_n\) represent \(n\)th~space derivatives, this choice of amplitude corresponds to the traditional conventional assumption that each space derivative is roughly of the same order of magnitude as the amplitude of the field itself (although I do use the order symbol in its strict sense that the left-hand side could be also vanishing relative to the right-hand side).\footnote{
By defining different amplitudes for the state vector~\(\uv \) we could make quite different assumptions about the relative order  of spatial derivatives, and even different assumptions about the relative magnitude of the fields \(c\) and~\(d\).  
Different choices correspond to adopting different views of the dynamics in the state space~\(\uv \).
The choice of amplitude~\eqref{er:unorm} appears the simplest and with the strongest connection to other methodologies.}
But in our approach the interpretation is fundamentally different to the traditional: 
here we recognise that the dynamics of the system~\eqref{er:odes} is what it is; 
our choice of amplitude merely affects how we describe geometric objects in the state space;
the choice~\eqref{er:unorm} corresponds to us choosing to describe the dynamics to multinomial terms of high order in~\(c_0,d_0\), intermediate order in~\(c_1,d_1\), and low order in~\(c_2,d_2\).
That is all that is implied by the amplitude.

Lastly, to be consistent with the Taylor series truncation~\eqref{ers:trt}, here we construct the slow manifold to an \emph{absolute error}~\Ord{\nuv^5}:
the exponent \(5=N+2+1\) since  \(N\)~orders are due to the \(N=2\) space derivatives in the truncation~\eqref{ers:trt}, two orders due to the quadratic nonlinearity in this particular problem, and the last one order to move to the leading error rather than the least significant order. 

The details of the construction of the slow manifold approximation are left to the computer algebra of Appendix~\ref{sec:camnhe}.
The computer algebra iteratively refines the description of the time dependent, nonlinear, slow manifold until the governing equations~\eqref{er:odes} are satisfied to the specified order of error, here the residuals are~\Ord{\nuv^5}.
Then \cite{Potzsche2006} [Proposition~3.6] assure us that the slow manifold is approximated to the same order of error.

We choose to parametrise the slow manifold in terms of the mean field variables~\(c_n\) as then the relation to the physical mean field is most direct.
As in the linear dynamics, the description involves convolutions,~\(\ou\big(,tt,-\big)\), over the past history of the uncertain coupling where the convolution is defined by~\eqref{eq:con}.
Appendix~\ref{sec:camnhe} then finds the local slow manifold to be
\begin{subequations}
\begin{eqnarray}
d_0&=&-\rat12c_0^2+c_1
+\rat38c_0^4-3c_0^2c_1+\rat32c_1^2+3c_0c_2
\nonumber\\&&{}
-3\Z\Z c_{2x}
-9c_0\Z d_{2x}-9c_0\Z{\Z d_{2x}},
\label{eqq:smd00}
\\
d_1&=&-c_0c_1+c_2-3\Z d_{2x}
+6c_0\Z{\Z c_{2x}},
\\
d_2&=& 
-c_1^2-c_0c_2
+3\Z c_{2x} +3c_0\Z d_{2x}.
\end{eqnarray}\label{eqq:smd0}
\end{subequations}
On this slow manifold the evolution is
\begin{subequations}
\begin{eqnarray}
\dot c_0&=& c_2-2c_0c_1+\rat12c_0^3-3\Z d_{2x}
+9c_0\Z{\Z c_{2x}},
\label{eqq:smdcdt0}
\\
\dot c_1&=&  
-2c_0c_2-2c_1^2+\rat32c_0^2c_1 
+3\Z c_{2x} +6c_0\Z d_{2x},
\\
\dot c_2&=& 3d_{2x}-3c_0\Z c_{2x}.
\end{eqnarray}\label{eqq:smdcdt}%
\end{subequations}

\subsection{The slow manifold represents a slowly varying model}

As established by section~\ref{sec:sme}, the slow manifold emerges exponentially quickly from all nearby initial conditions.  
%In principle there exists a near identity coordinate transform, \((\vec c,\vec d)\leftrightarrow(\vec C,\vec D)\), such that the transformed system has normal form  \(\de t{\vec D}=(-I+\Ord{\nuv^2})\vec D\). 
%By continuity in perturbations, there exists a finite domain in which \(\vec D=\Ord{e^{-\gamma t}}\) as \(t\to\infty\) for any chosen \(0<\gamma<1\)\,.  
%That is, the slow manifold generally emerges from transients~\Ord{e^{-\gamma t}}. 
To find the evolution on the slow manifold, recall the exact Taylor polynomial~\eqref{er:ctrt} for the mean field:
\(c(x,t)=c_0(X,t)+c_1(X,t)(x-X) +c_2(X,x,t)\rat12(x-X)^2\).
To obtain a \pde\ of the slow variations in the mean field~\(c\), first take the time derivative of~\eqref{er:ctrt} (keeping constant \(x\) and~\(X\)) and evaluate at \(x=X\), and second take the space derivatives and evaluate at \(x=X\):\footnote{Remember that the definition of~\(c_N(X,x,t)\) accounts for the uncertain variation of~\xec\ in time~\(t\).}
then
\begin{align}
&\left.\D tc\right|_{x=X}=\D t{c_0}\,,
&&\left.c\right|_{x=X}={c_0}\,,
&&\left.\D xc\right|_{x=X}={c_1}\,,
&&\left.\DD xc\right|_{x=X}={c_2}\,.
\label{er:cdiffid}
\end{align}
Substitute into the slow manifold evolution~\eqref{eqq:smdcdt0} for~\(c_0\) and obtain
\begin{align*}
\left.\D tc\right|_{x=X}
={}&\left.\left(\DD xc-2c\D xc+\rat12c^3
-\sigma\ou\big(w_{d},tt,-\big)
+9 c\ou\big(\ou\big(w_{c},tt,-\big),tt,-\big)
\right)\right|_{x=X}
\\&{}
+\Ord{\nuv^5,e^{-\gamma t}}.
\end{align*}
Recall that \(x=X\) is a generic station in the interior of the domain, thus the above evolution holds everywhere in the interior giving the model for the mean field to be the reaction-advection-diffusion \pde
\begin{equation}
\D tc
=\DD xc-2c\D xc+\rat12c^3
-3\ou\big(d_{2x},tt,-\big)
+9 c\,\ou\big(\ou\big(c_{2x},tt,-\big),tt,-\big)
+\Ord{\nuv^5,e^{-\gamma t}}.
\label{er:mpdemf}
\end{equation}
The rigorous slowly varying model is then the \pde~\eqref{er:mpdemf} with \(\Ord{e^{-\gamma t}}\) neglected as a quickly decaying transient, with \Ord{\nuv^5}~neglected as a nonlinear error, and  the unknown $-3\ou\big(d_{2x},tt,-\big)+9c\,\ou\big(\ou\big(c_{2x},tt,-\big),tt,-\big)$ neglected as the leading coupling error.

The reaction modified Burgers' \pde~\eqref{er:mpdemf} is the \pde\ one would obtain via a variety of systematic methods.  
What is new is the rigorous emergence at every interior locale (away from boundary layers and shocks) from a finite domain of initial conditions, and the novel leading order estimate of the spatial coupling error.

To find the slow manifold itself, recall the exact Taylor polynomial~\eqref{er:dtrt} for the difference field:
\(d(x,t)=d_0(X,t)+d_1(X,t)(x-X) +d_2(X,x,t)\rat12(x-X)^2\) so that
\(\left.d\right|_{x=X}={d_0}\)\,.
Substitute this and the expressions~\eqref{er:cdiffid} for~\(c_n\)  into the slow manifold expression~\eqref{eqq:smd00} for~\(d_0\) and obtain
\begin{align*}
\left.d\right|_{x=X}
={}&\left.\left[\D xc-\rat12c^2
+3c\DD xc+\rat32\left(\D xc\right)^2-3c^2\D xc+\rat38c^4
\right.\right.\\&\left.\left.\phantom{\D xc}
-3\ou\big(\ou\big(w_{c},tt,-\big),tt,-\big)
-9 c\ou\big(w_{d},tt,-\big)
-9 c\ou\big(\ou\big(w_{d},tt,-\big),tt,-\big)
\right]\right|_{x=X}
\\&{}
+\Ord{\nuv^5,e^{-\gamma t}}.
\end{align*}
Recall that \(x=X\) is a generic station in the interior of the domain, thus the above equation holds everywhere in the interior giving the difference field 
\begin{align}
d(x,t)&=
\D xc-\rat12c^2
+3c\DD xc+\rat32\left(\D xc\right)^2-3c^2\D xc+\rat38c^4
\nonumber\\&\quad{}
-3\ou\big(\ou\big(c_{2x},tt,-\big),tt,-\big)
-9 c\,\ou\big({},tt,-\big)(1+\ou\big({},tt,-\big))d_{2x}
%\nonumber\\&\quad{}
+\Ord{\nuv^5,e^{-\gamma t}}.
\label{er:msmmf}
\end{align}
The rigorous slow manifold is then~\eqref{er:msmmf} with \(\Ord{e^{-\gamma t}}\) neglected as a quickly decaying transient, with \Ord{\nuv^5} and the unknown coupling via~$d_{2x}$ and~$c_{2x}$ neglected as errors.

Importantly, in any particular situation we are now empowered to estimate the local errors by constructing to higher orders in~\(\nuv\), and we can bound the spatial coupling errors in terms of \(c_{Nx}\) and~\(d_{Nx}\).

\subsection{The generating function simplifies}
\label{sec:gfs}

To empower dealing with the hierarchy of \ode{}s~\eqref{er:odes} in a compact form, and making a direct connection with the method of multiple scales, let's introduce two generating functions (polynomials) that encapsulate the three local derivatives within the data structure of a quadratic polynomial:
\begin{subequations}\label{er:genfun}%
\begin{eqnarray}&&
\tc(\xi ,X,t):=c_0(X,t)\xit 0 +c_1(X,t)\xit 1 +c_2(X,t)\xit 2\,,\quad
\\&&
\td(\xi ,X,t):=d_0(X,t)\xit 0 +d_1(X,t)\xit 1 +d_2(X,t)\xit 2 
\end{eqnarray}
\end{subequations}
(recalling \(c_2(X,t):=c_2(X,X,t)\) and \(d_2(X,t):=d_2(X,X,t)\)).
Then by the sums \(\eqref{sq:c0}+\xi\eqref{sq:c1}+\rat12\xi^2\eqref{sq:c2}\) and \(\eqref{sq:d0}+\xi\eqref{sq:d1}+\rat12\xi^2\eqref{sq:d2}\), the system of six \ode{}s~\eqref{er:odes} are precisely the pair of coupled \ode{}s
\begin{eqnarray*}&&
\D t{\tc}=+\D \xi {\td}-\tc\td+\xit23d_{2x}
%\\&&\qquad{}
+\rat12\xi ^3(c_1d_2+c_2d_1)+\rat14\xi ^4c_2d_2 \,,
\\&&
\D t\td=-\td+\D \xi \tc-\rat12(\tc^2+\td^2)+\xit23c_{2x}
%\\&&\qquad{}
+\rat12\xi ^3(c_1c_2+d_1d_2)+\rat18\xi ^4(c_2^2+d_2^2).
\end{eqnarray*}
The explicit cubic and quartic terms in~\(\xi\) exactly cancel with the cubic and quartic terms in~\(\xi \) that are implicit in the nonlinear terms \(\tc\td\), \(\tc^2\) and~\(\td^2\).
We write an equivalent version of the above form by noting that \(c_1=\tc_\xi-\xi\tc_{\xi\xi}\), \(c_2=\tc_{\xi\xi}\) and similarly for~\(d_n\), then closed exact statements of the coupled \ode{}s are
\begin{eqnarray*}&&
\D t{\tc}=+\D \xi {\td}-\tc\td+\xit23d_{2x}
\\&&\qquad{}
+\rat12\xi ^3\left[(\tc_\xi-\xi\tc_{\xi\xi})\td_{\xi\xi}
+(\td_\xi-\xi\td_{\xi\xi})\tc_{\xi\xi}\right]+\rat14\xi ^4\tc_{\xi\xi}\td_{\xi\xi} \,,
\\&&
\D t\td=-\td+\D \xi \tc-\rat12(\tc^2+\td^2)+\xit23c_{2x}
\\&&\qquad{}
+\rat12\xi ^3\left[(\tc_\xi-\xi\tc_{\xi\xi})\tc_{\xi\xi}
+(\td_\xi-\xi\td_{\xi\xi})\td_{\xi\xi}\right]
+\rat18\xi ^4(\tc_{\xi\xi}^2+\td_{\xi\xi}^2).
\end{eqnarray*}
The generating polynomial transform~\eqref{er:genfun} maps from the vector~\(\uv \) of variables in the state space into (quadratic) polynomials in~\(\xi \).
Differentiation~\(\Dn \xi n{}\) and evaluation at \(\xi =0\) transforms back from the generating polynomials to the state space vector of variables. 
For example, the cubic and quartic terms disappear when differentiating up to twice and evaluating at \(\xi=0\).
%Two consequences are: 
%firstly, that we only need to know \(\tc\)~and~\(\td\) in an almost infinitesimal domain around \(\xi =0\) ---we just need to be able to differentiate and evaluate;
This back transform is impervious to any terms of higher order than quadratic in~\(\xi\) as we only address dynamics up to \(c_2\) and~\(d_2\),
thus let's lump the explicit higher order terms into one qualitative order term:
\begin{subequations}\label{ers:tt}%
\begin{eqnarray}&&
\D t{\tc}=+\D \xi {\td}-\tc\td+\xit23d_{2x}
+\Ord{\xi ^3} \,,
\label{er:tct}
\\&&
\D t\td=-\td+\D \xi \tc-\rat12(\tc^2+\td^2)+\xit23c_{2x}
+\Ord{\xi ^3}.
\label{er:tdt}
\end{eqnarray}%
\end{subequations}
It may be useful to remember that these order terms are not errors: instead within the nonlinearities there are implicit cubic and quartic terms in~\(\xi \) that these order terms cancel.

Amazingly, this generating polynomial form~\eqref{ers:tt} is nearly identical to the original non-dimensional physical \pde{}s~\eqref{er:hecann}.
The differences are:
\begin{itemize}
\item symbolically \(\tc,\td\) replace \(c,d\), and `artificial' \(\D \xi {}\) replaces spatial \(\D x{}\);
%\item \eqref{ers:tt} identifies the order to which analysis is to be carried out, the degree of the Taylor polynomials~\eqref{ers:trt}, through the terms~\Ord{\xi ^3}; and 
%\ajr{do I want to make this comment??}
\item \eqref{ers:tt} identifies the exact remainder terms, from Taylor's Remainder Theorem, through the terms~\(\xit23c_{2x}\) and~\(\xit23d_{2x}\).

\end{itemize}

The nonlinear analysis needs to be careful with the magnitude of variables and effects.
Via the definition of the amplitude~\eqref{er:unorm}, recognising the order of magnitudes~\eqref{er:cdmag}, given that the artificial~\(\xi\) is finite, and all as \(\nuv\to0\)\,,
\begin{subequations}%
\begin{eqnarray}
|\tc|&\leq& |c_0|+|\xi||c_1|+\rat12|\xi|^2|c_2|
\nonumber\\&&{}
=\Ord{\nuv+|\xi|\nuv^2+|\xi|^2\nuv^3}=\Ord{\nuv},
\label{er:qnormt}
\\
\left|\D\xi\tc\right|
&=&|c_1+\xi c_2|
\leq |c_1|+|\xi||c_2|
=\Ord{\nuv^2+|\xi|\nuv^3}
=\Ord{\nuv^2},
\qquad\\
\left|\DD\xi\tc\right|
&=&|c_2|
=\Ord{\nuv^3},
\end{eqnarray}
\end{subequations}
and similarly for the \(d_n\)~variables.
Let's exploit these orders of magnitude in the \pde{}s~\eqref{ers:tt} by labelling each term in the \pde{}s with its \emph{relative} order in~\(\nuv\).
Invoke this labelling explicitly in the \pde{}s by introducing an artificial parameter~\(\epsilon\) that counts the relative order of each term: for example, \(\D\xi\tc=\Ord{\nuv^2}\) and so is labelled with a multiplication by~\(\epsilon^2\) as it is second order, but then the whole \pde\ is divided by~\(\epsilon\) so that the term appears as \(\epsilon\D\xi\tc\) in the \pde{}s.
The \pde{}s~\eqref{ers:tt} then appear as%
\footnote{The uncertain coupling terms are unlabelled in~\eqref{ers:tte} as the coupling needs to match the implicit unlabelled components \(\xit2\D t{c_2}\) and \(\xit2\D t{d_2}\) on the left-hand sides.}
\begin{subequations}\label{ers:tte}%
\begin{eqnarray}&&
\D t{\tc}=+\epsilon\D \xi {\td}-\epsilon\tc\td+\xit23d_{2x}
+\Ord{\epsilon\xi ^3},
\label{er:tcte}
\\&&
\D t\td=-\td+\epsilon\D \xi \tc-\epsilon\rat12(\tc^2+\td^2)+\xit23c_{2x}
+\Ord{\epsilon\xi ^3}.
\label{er:tdte}
\end{eqnarray}%
\end{subequations}
This form corresponds closely to the classic scaled equations used in multiple scale modelling \cite[e.g.]{Roberts88a, Vandyke87, Nayfeh71b, Nayfeh2005}: in applying the method of multiple scales to the \pde~\eqref{er:hecann} one would 
\begin{itemize}
\item introduce a `slow space scale' \(\xi=\epsilon x\),
\item focus on small amplitude solutions by scaling fields \(c=\epsilon\tc(\xi,t)\) and \(d=\epsilon\td(\xi,t)\),
\end{itemize}
and then straightforward change of variables derives the `local' parts of the system~\eqref{ers:tte}, symbolically identically.
One difference is that we now include the uncertain coupling terms that form the leading error in a multiple scale analysis.  
Further, the interpretation is quite different to classic multiple scale modelling: 
here this form arises as a consequence of the convenient data structure of the generating polynomial, with the \(\epsilon\)~factors just doing some bookkeeping for us.
The data structure is convenient because it greatly simplifies, compared to sections~\ref{sec:sme}--\ref{sec:ufnsm}, the details of slow manifold modelling.

Note that in constructing the slow manifold, derivatives~\(\D\xi{}\) are always multiplied by~\(\epsilon\) so although lower powers of~\(\xi\) are generated by the unwanted~\Ord{\xi^3} terms, albeit implicit in~\eqref{ers:tte}, such lower powers come with higher powers of~\(\epsilon\).  
Since variable~\(\epsilon\) just counts order, such lower powers of~\(\xi\) remain of higher order in the construction.

\subsubsection{Establish the slow manifold model}

The practical procedure to construct a slow manifold model of the system~\eqref{ers:tte} follows a straightforward formal procedure first detailed decades ago \cite[]{Roberts88a}, but modified to now include the novel explicit uncertain coupling using techniques developed for non-autonomous deterministic  \cite[e.g.]{Potzsche2006} or stochastic systems  \cite[e.g.]{Arnold93, Roberts06k}.
However, in this application of the approach, the interpretation and justification of the formal procedure is different.

Firstly, the system~\eqref{ers:tte} looks like \pde{}s because of the derivatives~\(\D\xi{}\). 
But the system is not a \pde\ because these derivatives just access different components in the generating polynomials: the system is a set of \ode{}s.
The system is a set of \ode{}s at each station~\(X\), uncertainly coupled by \(c_{2x}\) and~\(d_{2x}\) to \ode{}s at all other stations.
The `infinite dimensionality' of the original physical \pde{}s~\eqref{er:hecann} arises via the uncertain coupling between locales in the system~\eqref{ers:tte}.

\paragraph{Equilibrium}
The slow manifold is based upon the equilibrium at the origin \(\tc=\td=0\) for the system~\eqref{ers:tte}.
When zero throughout the domain~\XX, then the uncertain coupling is also zero giving an equilibrium over the whole interior.

\paragraph{Linearisation}
In the system~\eqref{ers:tte}, the terms~\(\D\xi{}\) just represent the off-diagonal blocks in the block upper triangular matrices~\cL\ of~\eqref{er:odes}.
The terms~\(\D\xi{}\) are negligible in the sense that they do not affect the eigenvalues: not that they are multiplied by~\(\epsilon\) (which is only a convenient counter); nor that they are of `higher order' (as in multiple scales).
That is, the spectrum of the linearisation about the origin is the same as that for 
\begin{equation*}
\D t\tc=0\,,\quad\D t\td=-\td\,,
\end{equation*}
namely, eigenvalues \(\lambda\in\{0,-1\}\), each of multiplicity three (once for each component in~\(\xi^n\)).

But this spectrum only accounts for the local dynamics at a station.  
It is only in accounting for the uncertain coupling between neighbouring stations that we make the slowly varying assumption (as section~\ref{sec:sme} discusses): we are only interested in regimes where the uncertain coupling is a negligible influence.
That is, we assume that solutions vary smooth enough in the domain that the terms of~\eqref{ers:tte} in~\(c_{2x}\) and~\(d_{2x}\) are a negligible perturbing influence---quantified by tracking their perturbative effects.
In this approach, the `slowly varying' assumption only directly involves these gradients of the highest resolved derivatives (section~\ref{sec:sme}); we do not need to make restrictive assumptions about the magnitude of the other derivatives.
Returning to the spectrum, and upon recognising the coupling, the eigenvalues \(\lambda\in\{0,-1\}\) are repeated an `infinite number' of times for all the stations in the domain~\XX. 
As established by Proposition~\ref{thm:gpdecm} in section~\ref{sec:cmtsm}, theory by \cite{Aulbach96, Aulbach2000} assures us that an emergent slow manifold then exists for the generating polynomial system~\eqref{ers:tte}.
%The reason is that the system~\eqref{ers:tte} is an invertible transformation of the system of section~\ref{sec:sme}.
%Because system~\eqref{ers:tte} is symbolically identical to that adopted by the method of multiple scales, our reformulation thus provides new rigorous support for the method of multiple scales.
%More generally our reformulation also supports the previously proposed, and more flexible, formal multiscale modelling procedure based upon centre manifold ideas \cite[]{Roberts88a}.
%But further, our reformulation extends these multiscale modelling procedures by providing a precise description of the uncertain coupling which previously has been an unknown error.

\paragraph{Construct the slow manifold model}
Appendix~\ref{sec:smvgf} lists and describes computer algebra code that, in essence, implements the earlier formal procedure \cite[]{Roberts88a, Roberts96a}. 
But to cater for the uncertain coupling in~\eqref{ers:tte}, the procedure is extended using techniques developed for non-autonomous and stochastic systems  \cite[e.g.]{Chao95} and validated by Proposition~3.6 of \cite{Potzsche2006}.

Appendix~\ref{sec:smvgf} iteratively improves a description of the slow manifold and evolution thereon until the residuals of the system~\eqref{ers:tte} are~\Ord{\nuv^{4}+\xi^{4}}, relative to~\(\nuv\).
That is, the residuals are~\Ord{\nuv^{5}+\xi^{5}}, in absolute terms, to correspond to the order of error adopted by section~\ref{sec:ufnsm}.
Appendix~\ref{sec:ctsmv} then lists code that unpacks from this new description the Taylor series description and confirms that it is identical to the previously derived slow manifold~\eqref{eqq:smd0} and evolution~\eqref{eqq:smdcdt}.
This agreement holds for all tested truncations, namely \(N\leq9\).

However, the derivation here is more compact (as well as directly connecting to and extending previous methodologies).
Here, Appendix~\ref{sec:smvgf} a slow manifold in the form
\begin{align}
\td={}&\left[-\rat12\tc^2+\D\xi\tc\right]
+\left[\rat38\tc^4-3\tc^2\D\xi\tc+\rat32\left(\D\xi\tc\right)^2
+3\tc\DD\xi\tc-\Dn\xi3\tc\right]
\nonumber\\&{}
+\left[\xit23\Z c_{2x}
-\xit13\Z d_{2x}+\xit23\Z d_{2x}
-3\Z{\Z c_{2x}}
\right.\nonumber\\&\left.\vphantom{\xit2}{}
+\xit1\tc\Z{\Z c_{2x}}
-9\tc\Z d_{2x}-9\tc\Z{\Z d_{2x}}
\right]
+\Ord{\nuv^{5}+\xi^{5}}.
\label{eq:hesmgf}
\end{align}
The first bracketed terms form the leading, second order, quasi-equilibria, estimate of the difference field, where \(\xi\)~derivatives correspond to spatial derivatives.
The second bracketed terms give fourth order corrections in this mixed order description.
The third bracketed terms form an estimate of the error induced by coupling with neighbouring stations: the different powers in~\(\xi\) label the different errors for the various spatial derivatives of the field~\(\td\).
The computer algebra simultaneously finds that on the slow manifold~\eqref{eq:hesmgf} the evolution is
\begin{eqnarray}
\D t\tc&=&
\left[\rat12\tc^3-2\tc\D\xi\tc+\DD\xi\tc\right]
+\left[\xit23d_{2x}
+\xit13\Z c_{2x}-\xit23\Z c_{2x}
\right.\nonumber\\&&\left.\vphantom{\xit2}{}
-3\Z d_{2x}+\xit16\tc\Z d_{2x} 
+9\tc\Z{\Z c_{2x}}
\right]
+\Ord{\nuv^{5}+\xi^{5}}.
\quad\label{eq:hesmegf}
\end{eqnarray}
The first bracketed terms gives the leading, third order, model~\eqref{eq:dcdtblpde} of Burgers'-like advection-diffusion with a cubic reaction.
The second bracketed terms additionally estimate the error induced by coupling with neighbouring stations: again, the different powers in~\(\xi\) label the different errors for the various spatial derivatives of the time derivative~\(\D t\tc\).

In short, and because of the symbolic identity between \(\xi\) and~\(x\) derivatives, and because of the general emergence of slow manifolds in some domain, the generating polynomial approach directly, compactly and efficiently derives the slow manifold model~\eqref{er:mpdemf}--\eqref{er:msmmf}.
%\begin{eqnarray*}
%d&=&-\rat12c^2+\D xc-3\Z{\Z c_{2x}}
%+\Ord{\nuv^{4},e^{-\gamma t}},
%\end{eqnarray*}
%\begin{eqnarray*}
%\D tc&=&
%\rat12c^3-2c\D xc+\DD xc
%-3\Z d_{2x}
%+\Ord{\nuv^{4},e^{-\gamma t}}.
%\end{eqnarray*}

\section{Model nonlinear dynamics in cylindrical domains}
\label{sec:mndcd}

Inspired by the modelling of the nonlinear heat exchanger (section~\ref{sec:nhem}), this section extends the general linear analysis of section~\ref{sec:pmid} to general nonlinear dynamics in cylindrical domains.

Adding nonlinearity to the class of \pde{}s~\eqref{eq:upde}, this section develops a rigorous approach, Proposition~\ref{thm:gpdecm}, to modelling the dynamics of \pde{}s in the class
\begin{eqnarray}
\D tu&=&\fL[u]+f[u]
\nonumber\\&=&
\fL_0u+\fL_1\D xu+\fL_2\DD xu+\cdots
+f\left(u,\D xu,\DD xu,\ldots\right),
\label{eq:upden}
\end{eqnarray}
where, as in section~\ref{sec:pmid}, the \pde{} holds on a cylindrical domain~\(\XX\times\YY\) for some field~\(u(x,y,t)\) in a given Banach space~\UU, where \(u\)~is a function of 1D longitudinal position \(x\in\XX\subset \RR\), cross-sectional position \(y\in\YY\subset\RR^\ydim\), and time \(t\in\RR\).
The square brackets notation on functions such as~\(f[u]\) denotes a dependence upon values of the field~\(u\) locally in~\(x\), namely upon~\(u\) and its derivatives (although it may be nonlocal in~\(y\)), as alternatively explicitly expressed in the parentheses of~\(f(u,\D xu,\DD xu,\ldots)\).%
\footnote{Further research aims to generalise this scenario to nonlocal operators~\(\fL_n\), nonlocal nonlinearity~\(f\) and nonautonomous systems.}
The nonlinearity function \(f[]:\UU\to\UU\) has no linear terms, formally \(f[u]=\Ord{u^2}\) as \(u\to0\)\,.

\begin{assumption}\label{ass:nonlin}
The operators~\(\fL_\ell\) continue to satisfy Assumption~\ref{ass:spec}. 
The nonlinearity~\(f()\) is autonomous and independent of longitudinal position~\(x\).
Extending section~\ref{sec:pmid}, the nonlinear function~\(f\) must
be~\Ord{|u|^p} as \(u\to0\)\,, \(p\geq2\)\,, and sufficiently smooth to have at least \(N+1+p\)~derivatives in a suitable domain about \(u=0\).
\end{assumption}

So far we have used \XX~to denote the open set of the physical domain (primarily because in linear dynamics we could address the dynamics at any station~\(X\) `independent' of other locales).
In nonlinear systems we address the dynamics as a collective whole across all relevant space---all relevant stations.
Thus we slightly change~\XX\ to now denote an open subset of the physical domain where boundary layers and internal shocks are excised from~\XX\ (if the domain is \(L\)-periodic, then \XX~could still be the entire spatial domain).
This qualitative redefinition is to ensure that the uncertain coupling is `small' for all stations \(X\in\XX\)\,:
the smaller a desired error in the modelling, the more restrictive the physical domain~\XX\ over which the modelling is valid.

\subsection{The generating function has equivalent dynamics}
\label{sec:gfhed}

This section establishes the following proposition.
The next section~\ref{sec:cucncm} then uses this form to establish a practical approach to constructing models of slow space-time evolution.

\begin{proposition}[nonlinear equivalence] \label{thm:nsv}
Let \(u(x,y,t)\) be governed by a \pde\ of the form~\eqref{eq:upden}.  
Then the dynamics at all locales \(X\in\XX\) are equivalently governed by the equation
\begin{equation}
\D t\tu=\sum_{\ell}\fL_\ell\Dn\xi\ell\tu
+f\left(\tu,\D\xi\tu,\DD\xi\tu,\ldots\right)
+r[u],
\label{eq:updeng}
\end{equation}
for the generating function polynomial~\(\tu(X,\xi,y,t)\) defined in~\eqref{eq:ugen}, and
for the `uncertain' coupling term~\(r[u]:\UU\to\UU\) given by~\eqref{eq:grun}.
%Define the `mean field' \(c(x,t):=\langle Z_0(y),u(x,y,t)\rangle\) for \(Z_0(y)\) and inner product of Definition~\ref{def:adj}.
%Then, in a suitable regime the mean field~\(c\) is determined by the formal slowly varying solutions of ??
%satisfies the \pde
%\begin{equation}
%\D tc=\sum_{n=0}^NA_n\Dn xnc\,, \quad x\in\XX\,,
%\label{eq:dcdta}
%\end{equation}
%in terms of matrices~\(A_n\) given by~\eqref{eq:vzadef}--\eqref{eq:toeport}, 
%to an error quantified by~\eqref{eq:pderemainn}, and upon ignoring transients decaying exponentially quickly in time.
\end{proposition}

As in section~\ref{sec:pmid} for linear \pde{}s, for nonlinear \pde{}s in the general form~\eqref{eq:upden}, assume the field~\(u\) is smooth enough to have continuous \(2N\)~derivatives in~\(x\) for some pre-specified Taylor series truncation~\(N\).  
Choose an arbitrary cross-station \(X\in\XX\).
Then write the field~\(u\) in terms of a local polynomial~\eqref{eq:utrtn} about the cross-section \(x=X\)\,.
As in section~\ref{sec:pmid}, \(u_N(X,x,y,t)\) is the \(N\)th~derivative at some implicit uncertain location~\(\xey\).
Define the generating polynomial
\begin{equation}
\tu(X,\xi,y,t):=\sum_{n=0}^{N-1}\xit n u_n(X,y,t)+\xit Nu_N(X,X,y,t),
\label{eq:ugen}
\end{equation}
\(\tu:\XX\times\Xi\times\YY\times\RR\to\UU\) for an arbitrary open interval~\(\Xi\subset\RR\) containing zero.
The first aim of this section is to prove that systematic modelling of the \pde~\eqref{eq:upden} is equivalent to well-known heuristic procedures expressed in terms of this generating polynomial.

A key task is to relate fields in physical space with their corresponding field in the `generating polynomial space'.
Define the operator
\begin{equation}
\cG:=\left[\sum_{n=0}^N\xit n\Dn xn{}\right]_{x=X}
=\left[1+\xit1\D x{}+\cdots+ \xit N\Dn xN{}\right]_{x=X},
\label{eq:Gn}
\end{equation}
where these brackets denote evaluation.
This operator is denoted by~\cG\ to signify it determines the generating polynomial corresponding to a given field: for example, it is straightforward to deduce from the Taylor polynomial~\eqref{eq:utrtn} and the Definition~\eqref{eq:ugen} that
\begin{equation}
\cG u(x,y,t)=\tu(X,\xi,y,t).
\label{eq:}
\end{equation}
But to use operator~\(\cG\) observe from~\eqref{eq:Gn} that 
\begin{equation}
\cG=\left[e^{\xi \partial_x}+\Ord{\xi^{N+1}}\right]_{x=X}
=\left[\cdot\vphantom{e^\xi}\right]_{x=X+\xi}+\Ord{\xi^{N+1}}
\label{eq:Gnshft}
\end{equation}
\cite[p.65, e.g.]{npl61}; that is, the generating polynomial is equivalent, to errors~\Ord{\xi^{N+1}}, to evaluation a distance~\(\xi\) from the chosen cross-section~\(X\in\XX\).
This equivalence of~\(\xi\) and space~\(x\) is the key to the equivalence between our rigorous approach to modelling and the well established heuristic of slow scaling of the space variables.

Crucially, differences arise between the equivalence, and these differences lead to our derivation of remainder terms that combine to form a systematic description of the modelling error.
The differences arise in spatial gradients.
\begin{lemma} \label{lem:dudxlg}
Use \(u_N^{(p)}\) to denote the \(p\)th derivative~\(\Dn xp{u_N}\). Then for \(\ell=0,\ldots,N\),
\begin{equation}
\cG\Dn x\ell u
=\Dn\xi\ell\tu 
+\sum_{n=N-\ell+1}^N\binom{n+\ell}{N}u_N^{(n+\ell-N)}(X,X,y,t)\xit n
+\Ord{u_N\xi^{N+1}}.
\label{eq:dudxlg}
\end{equation}
\end{lemma}
This sum of spatial derivatives of~\(u_N\) induce the remainders~\eqref{eq:remain} observed in the linear modelling of section~\ref{sec:pmid}: here the factor of~\(\xi^n\) determines the corresponding remainder~\(r_n\).

\begin{proof} 
Using \(u_N^{(p)}\) to denote the \(p\)th derivative~\(\Dn xp{u_N}\), from the \(\ell\)th~derivative~\eqref{eq:dudxl}
\begin{eqnarray*}
\cG\Dn x\ell u
&=&\sum_{n=0}^{N-\ell-1}u_{n+\ell}(X,y,t)\xit n
\nonumber\\&&{}
+\sum_{n=N-\ell}^N\binom{\ell}{N-n}u_N^{(n+\ell-N)}(X,X+\xi,y,t)\xit n
+\Ord{u_N\xi^{N+1}}
\nonumber\\&=&
\Dn\xi\ell\tu %-\xit{(N-\ell)}u_N(X,X,y,t)
+\sum_{n=N-\ell+1}^N\binom{n+\ell}{N}u_N^{(n+\ell-N)}(X,X,y,t)\xit n
+\Ord{u_N\xi^{N+1}},
%\qquad\label{eq:dudxlg}
\end{eqnarray*}
upon using~\eqref{eq:Gnshft}, expanding in~\(\xi\), rearranging sums, and invoking a binomial identity.
This derives~\eqref{eq:dudxlg}.
%\begin{verbatim}
%clear thesum check
%N=4
%for l=0:N
%for n=N-l:N
%ks=0:n+l-N;
%sm=0; for k=ks, sm=sm+nchoosek(n,k)*nchoosek(l,N+k-n); end
%thesum(l+1,n+1)=sm;
%check(l+1,n+1)=nchoosek(l+n,N);
%end, end
%thesum=[thesum check]
%\end{verbatim}
\end{proof}

Now we establish that the operator~\cG\ distributes through nonlinearities with small remainder.
For a preliminary suggestive example, and upon setting the truncation \(N=1\) for simplicity, from definition and~\eqref{eq:dudxl} we find the generating polynomial corresponding to a cubic nonlinearity as follows:
\begin{eqnarray*}
\cG (u^3)&=&\left[u^3+\xi\D x{u^3}\right]_{x=X}
=\left[u^3+\xi3u^2\D x{u}\right]_{x=X}
=u_0^3+\xi3u_0^2u_1
\\&=&
(u_0+\xi u_1)^3-\xi^23u_0u_1^2-\xi^3u_1^3
=(\cG u)^3+\Ord{\xi^2\nuv^5}
\\&=&
\tu^3+\Ord{\xi^7+\nuv^7}.
\end{eqnarray*}
%\ajr{Do we want the norms here? or the original u??}

\begin{lemma}
Under the conditions of Assumption~\ref{ass:nonlin},
\begin{equation}
\cG f[u]=f[\cG u]+\Ord{\xi^{N+1}u^p,u^{N+p+1}}.
\label{eq:cGfup}
\end{equation}
\end{lemma}
\begin{proof} 
For general multinomial nonlinearities, proceed by induction.
First, it is trivial that \(\cG u^{(\ell)}=\cG u^{(\ell)}+\Ord{\xi^{N+1}u}\), where we continue to use superscripts in parentheses to denote \(x\)~derivatives.
Second, assume that 
\begin{equation}
\cG g[u]=g[\cG u]+\Ord{\xi^{N+1}u^q}
\label{eq:cGgup}
\end{equation}
for any \(q\)th~order multinomial term~\(g[u]\).
Third, consider a \((q+1)\)th~order multinomial term~\(u^{(\ell)}g[u]\) where \(g[u]\)~is \(q\)th~order.
Then, starting from Definition~\eqref{eq:Gn},
\begin{eqnarray*}
\cG \big(u^{(\ell)}g[u]\big)
&=& \left[\sum_{n=0}^N\xit n\Dn xn{}\left(u^{(\ell)}g[u]\right)\right]_{x=X}
\\
&=& \left[\sum_{n=0}^N\xit n
\sum_{k=0}^n \binom nk u^{(n-k+\ell)}\Dn xkg\right]_{x=X}
\\
&=& \left[\sum_{k=0}^N \sum_{n=k}^N 
\xit{(n-k)} u^{(n-k+\ell)}\xit k\Dn xkg \right]_{x=X}
\\
&=& \left[\sum_{k=0}^N \xit k\Dn xkg \sum_{n=0}^{N-k} 
\xit{n} u^{(n+\ell)} \right]_{x=X}
\\
&=& \left[\sum_{k=0}^N \xit k\Dn xkg \left\{\sum_{n=0}^{N} 
\xit{n} u^{(n+\ell)}+\Ord{\xi^{N-k+1}u} \right\} \right]_{x=X}
\\
&=& \left[\left(\sum_{k=0}^N \xit k\Dn xkg\right) 
\left(\sum_{n=0}^{N} 
\xit{n} u^{(n+\ell)}\right)+\Ord{\xi^{N+1}ug[u]} \right]_{x=X}
\\&=&\cG g[u]\cdot\cG u^{(\ell)}+\Ord{\xi^{N+1}u^{q+1}}
\\&=&g[\cG u]\cdot\cG u^{(\ell)}+\Ord{\xi^{N+1}u^{q+1}}.
\end{eqnarray*}
By induction, \eqref{eq:cGgup} holds for all multinomial terms of all orders~\(q\in\NN\).
By linearity, \eqref{eq:cGgup}~holds for all multinomial sums~\(g[u]\) where the order~\(q\) is then determined from the lowest order terms in~\(g\); that is, if \(g=\Ord{u^p}\) as \(u\to 0\), then \(q=p\)\,.

By the smoothness of the nonlinearity~\(f\), Assumption~\ref{ass:nonlin}, \(f\)~has a multivariate Taylor series to order~\((N+p+1)\) and hence \eqref{eq:cGgup}~ensures that \eqref{eq:cGfup}~holds.
\end{proof}

\begin{lemma}\label{lem:gftu}
Under the conditions of Assumption~\ref{ass:nonlin},
\begin{equation}
\cG f[u]=f\left(\tu,\D\xi\tu,\DD\xi\tu,\ldots\right)
+\Ord{\nuv^{N+p+1}+\xi^{N+p+1}},
\label{eq:cGdtu}
\end{equation}
where in terms of a norm \(\|\cdot\|:\UU\to\RR\), we define the derivative weighted amplitude (not a norm)
\begin{equation}
\nuv:=\left[\sum_{n=0}^N \left\|\Dn xnu\right\|^{1/(n+1)}\right]_{x=X}.
\label{eq:nuvg}
\end{equation}
\end{lemma}

\begin{proof} 
As in the preceding proof, regard nonlinearity~\(f\) as a linear combination of multinomial terms.
The lowest order terms generate the largest errors. 
Since \(f=\Ord{u^p}\), the lowest order terms are of the form~\(u^{p-1}\Dn x\ell u\)\,.
Using~\eqref{eq:cGfup} and~\eqref{eq:dudxlg}, consider
\begin{eqnarray*}
\cG\left(u^{p-1}\Dn x\ell u\right)
&=&(\cG u)^{p-1}\cG\left(\Dn x\ell u\right)
+\Ord{\xi^{N+1}u^p,u^{N+p+1}}
\\
&=&\tu^{p-1}\left(\Dn\xi\ell\tu +\Ord{u_{Nx}\xi^{N-\ell+1}}\right)
+\Ord{\xi^{N+1}u^p,u^{N+p+1}}
\\
&=&\tu^{p-1}\Dn\xi\ell\tu +\Ord{\nuv^{N+p}\xi^{N-\ell+1}}
+\Ord{\xi^{N+1}u^p,u^{N+p+1}}
\\
&=&\tu^{p-1}\Dn\xi\ell\tu +\Ord{\nuv^{N+p+1}+\xi^{N+p+1}},
\end{eqnarray*}
because \(N-\ell+1\geq 1\) as \(\ell+p\leq N+p\) (as otherwise the term is neglected).
Writing the nonlinearity~\(f\) in terms of its multivariate Taylor series, and using the linearity of operator~\cG, we therefore derive~\eqref{eq:cGdtu}. 
\end{proof}

\paragraph{Establish Proposition~\ref{thm:nsv}}
Recall that we decide on an order~\(N\) of Taylor series truncation.
Then by Taylor's Theorem we write the field~\(u\) as the local expansion~\eqref{eq:utrtn} about the cross-section \(x=X\) in terms of functions~\(u_n\).
Consider the \pde~\eqref{eq:upde} for the field~\(u\) in the polynomial form~\eqref{eq:utrtn}.
The operator~\cG\ when applied to the \pde~\eqref{eq:upde} performs the complete process of \begin{enumerate}
\item finding all the derivatives of the \pde, 
\item evaluating at the station \(x=X\), and lastly 
\item forming into an equation for the generating polynomial~\tu.
\end{enumerate}
This process works because although operator~\cG\ does not commute with \(x\)~derivatives (Lemma~\ref{lem:dudxlg}), from the Definition~\eqref{eq:Gn} operator~\cG\ does commute with \(\D t{}\) and with cross-sectional \(y\)~operators.
Invoking the Taylor expansion~\eqref{eq:utrtn} for the field~\(u\), applying~\cG, and using Lemmas~\ref{lem:dudxlg}--\ref{lem:gftu} the governing \pde~\eqref{eq:upde} becomes the equation~\eqref{eq:updeng}
%\ajr{We really need to define the truncation of linear expansion??}
where, from~\eqref{eq:dudxlg}, the remainder term defined as 
\begin{equation}
r[u]:=
\sum_{\ell=1}
\sum_{n=N-\ell+1}^N\binom{n+\ell}{N}\fL_\ell {u_N^{(n+\ell-N)}}\xit n
+\Ord{\nuv^{N+p+1}+\xi^{N+p+1}},
\label{eq:grun}
\end{equation}
where~\(u_N\), defined by~\eqref{eq:utrtn}, has derivatives evaluated at \(x=X\)\,,
and where the error term in~\eqref{eq:dudxlg} is absorbed in the remainder term here (provided the lowest order of the nonlinearity \(p\leq N+1\)).
This completes the proof of Proposition~\ref{thm:nsv}.
%\ajr{Could change the order of the error to N+1+min(p,N+1) ??}

\subsection{Construct nonlinear models of slow spatial variations}
\label{sec:cucncm}

Anticipating the existence and emergence results of the subsequent section~\ref{sec:cmtsm}, this section shows how the generating polynomial leads to established, direct, practical constructions of a centre manifold model of slowly varying solutions.
%Further, this section shows how the approach relates to other methods, and extends them by incorporating the leading error estimate.

\begin{corollary}[multiple scales methodology] \label{cor:msm}
Choosing truncation~\(N\) to give only the leading order nontrivial dynamics, the method of multiple scales applied to \pde~\eqref{eq:upden} is symbolically equivalent to constructing a centre manifold model of~\eqref{eq:updeng}  for the generating polynomial~\tu\ to an error~\Ord{u_{Nx}}.
\end{corollary}

\begin{proof} 
In the method of multiple scales \cite[e.g.]{Nayfeh85}, consider the \pde~\eqref{eq:upden} and seek solutions \(u(x,y,t)=\tu(\xi,y,\tau;\epsilon)\) for some slow variables \(\xi=\epsilon x\) and \(\tau=\epsilon^Nt\).
The \pde~\eqref{eq:upden} then becomes
\begin{equation}
\epsilon^{N+1}\D \tau\tu=\sum_{\ell}\fL_\ell\epsilon^{\ell+1}\Dn\xi\ell\tu
+ f\left(\epsilon\tu,\epsilon^2\D\xi\tu,\epsilon^3\DD\xi\tu,\ldots\right).
\label{eq:ms}
\end{equation}
Then the method seeks a solution of this equation in a power series in~\(\epsilon\).
The solution satisfies the equation to errors~\Ord{\epsilon^{N+2}}, and gives the leading order evolution in terms multiplied by~\(\epsilon^{N+1}\).
Conversely, in the generating polynomial equation~\eqref{eq:updeng} let's label all~\tu\ with an~\(\epsilon\) and each derivative~\(\D\xi{}\) with an~\(\epsilon\).
Then apart from a trivial scaling of time, equation~\eqref{eq:updeng} is identical to the multiple scale approximation to~\eqref{eq:ms} provided we establish \(r[u]\equiv\Ord{\epsilon^{N+2}}\):
\begin{itemize}
\item first, the amplitude Definition~\eqref{eq:nuvg} implies that in the method of multiple scales, \(\nuv=\Ord\epsilon\), and thus, as \(p\geq2\)\,, the \Ord{\nuv^{N+p+1}}~terms in~\eqref{eq:grun} for~\(r[u]\) are \(\Ord{\epsilon^{N+p+1}}=\Ord{\epsilon^{N+2}}\);
\item second, in all the other terms of~\eqref{eq:grun}, the lowest order term is~\(u_{Nx}\) which is an \((N+1)\)th derivative of small~\(u\) and so in the multiple scales scheme is~\Ord{\epsilon^{N+2}}.
\end{itemize}
Thus the multiple scales method is equivalent to the leading order truncation of the centre manifold model of~\eqref{eq:updeng}.
\end{proof}

However, I do not see that the leading error term in~\(r[u]\) can be incorporated into multiple scales method as the method requires all effects to occur at the leading order (let's not explore extensions that invoke an indefinite hierarchy of super-slow space and time scales).
Further, our use of centre manifold theory supports arbitrarily high order modelling;
in particular, we can now provide rigorous support for practical mixed order models \cite[e.g.]{Roberts92c, Roberts96a}.

\begin{corollary} \label{cor:robformproc}
For any truncation~\(N\), the formal procedure proposed by \cite{Roberts88a}  is symbolically equivalent to constructing a centre manifold model of~\eqref{eq:updeng} for the generating polynomial~\tu\ to an error~\Ord{\nuv^{N+2}}.
\end{corollary}

\begin{proof} 
The formal procedure \cite[p.497]{Roberts88a} proposed to simply treat derivatives~\(\D x{}\) and field~\(u\) as small.  
The procedure counts an order of magnitude for each derivative and field variable.
Thus, from the definition of the amplitude~\eqref{eq:nuvg}, truncating the analysis of the \pde~\eqref{eq:upden} to errors of order~\(N+2\) is equivalent to solving the generating polynomial equation~\eqref{eq:updeng} to errors~\Ord{\nuv^{N+2}}.
As for corollary~\ref{cor:msm}, the lowest order term of~\(r[u]\) is~\(u_{Nx}\) which is an \((N+1)\)th derivative of small~\(u\) and so in this scheme is of order~\({N+2}\) and so included within the error.
\end{proof}

Furthermore, the formal approach \cite[]{Roberts88a} is sufficiently flexible to incorporate some of the coupling terms in~\(r[u]\) and hence quantify a leading order estimate of the modelling error. 
For example, the computer algebra of Appendix~\ref{sec:smvgf} analyses the heat exchanger~\eqref{er:hepde} and finds a slow manifold~\eqref{eq:hesmgf} and evolution thereon~\eqref{eq:hesmegf} complete with an estimate of the error induced by coupling with neighbouring stations.

Assuming we can treat the inter-station coupling, via the derivatives~\(u_N^{(\ell)}\), as time dependent forcing of the local system, then the following corollary immediately follows from Proposition~\ref{thm:nsv}.

\begin{corollary}
Constructing a centre manifold model for system~\eqref{eq:updeng} to errors~\Ord{\nuv^{N+p+1}} gives a slowly varying centre manifold model of \pde~\eqref{eq:upden} complete with a leading order estimate of the errors due to the slow space variations.
Further, in constructing the centre manifold, when finding corrections one may neglect \(\sum_{\ell=1}^N\fL_\ell\Dn\xi\ell{}\) acting on corrections, not because they are `small', but because the error in doing so is subsequently corrected anyway.
\end{corollary}

\begin{proof} 
The first part of the corollary follows from the equivalence of Proposition~\ref{thm:nsv} and that the leading order coupling terms in~\eqref{eq:grun} are of lower order.
Theory for non-autonomous systems asserts the errors in the slow manifold model are of the same order as the residuals in the governing nonlinear system \cite[Proposition~3.6]{Potzsche2006}, even for random systems \cite[e.g.]{Arnold03, Roberts06k}, and accounts for effects of time dependent coupling terms in~\(u_N^{(\ell)}\).
The second part follows because the linear~\(\D\xi{}\) terms signify generalised eigenvectors which are generally found iteratively, see Section~\ref{sec:mlas}.
\end{proof}

\subsection{Application: nonlinear pattern formation}
\label{sec:anpf}

Before proving the existence and emergence results of the next section~\ref{sec:cmtsm},
let's model the long time evolution of small amplitude solutions of the Swift--Hohenberg system in one space dimension: a field~\(u(\ix,t)\) satisfies the nondimensional nonlinear \pde
\begin{equation}
\D tu=ru-(1+\partial_{\ix\ix})^2u-u^3
\label{eq:shpde}
\end{equation}
on a domain~\(\XX\) of large extent in~\(\ix\).
For parameter \(r\)~small, the slow marginal modes are \(u\propto e^{\pm i\ix}\).
The aim is to derive the well-known Ginzburg--Landau \pde
\begin{equation}
\D tc\approx rc-3|c|^2c+4\DD \ix c\,,
\label{eq:gle}
\end{equation}
for the complex amplitude~\(c(\ix,t)\) of oscillatory patterns \(u(\ix,t)\approx ce^{i\ix}+\bar ce^{-i\ix}\)  \cite[e.g.]{Cross93}.

Significant theory exists to support the modelling of pattern formation by a Ginzburg--Landau equation.  For examples, \cite{Eckhaus93} proved it emerges from nearby initial conditions.
\cite{Mielke95} also proved attractors existed for a class of problems including the Swift--Hohenberg equation.
\cite{Schneider1999} developed the work further to find global existence results for pattern forming processes in applications to 3D Navier--Stokes problems. 
\cite{Blomker04} developed some theory for a stochastic Ginzburg--Landau model of a stochastic Swift--Hohenberg equation in large domains. 
This section provides new support for the Ginzburg--Landau approximation to complement such earlier work, but additionally quantifies the leading error in its slowly varying approximation.

Section~\ref{sec:egpes} establishes a basis for analysing the Swift--Hohenberg \pde~\eqref{eq:shpde}.
Recall we embed the \pde\ into larger problem, as illustrated by Figure~\ref{fig:pattensemble}: the linear \pde~\eqref{eq:shembed} for a field~\(\fu (\fx,y,t)\), \(2\pi\)-periodic in~\(y\), becomes here the nonlinear
\begin{equation}
\D t\fu =r\fu-(1+\partial_{yy}+2\partial_{y\fx}+\partial_{\fx\fx})^2\fu -\fu^3,
\quad \text{for }(\fx,y)\in \XX\times[0,2\pi).
\label{eq:shembedn}
\end{equation}
Then solutions of the Swift--Hohenberg \pde~\eqref{eq:shpde} are \(u(\ix,t)=\fu (\ix,\ix+\phi,t)\) for any phase~\(\phi\).
Equation~\eqref{eq:sheln} details the linear operators \(\fL_0,\ldots,\fL_4\)\,.
The \pde~\eqref{eq:shembedn} satisfies the necessary Assumption~\ref{ass:nonlin} on the linear and nonlinear parts, provided parameter \(|r|<1/N\).
In particular, \cite{Haragus2011} [\S2.4.3] show that \(\fL_0\)~satisfies the requisite properties for a local centre manifold to exist and be attractive.
We choose truncation \(N:=2\) to derive the Ginzburg--Landau~\pde~\eqref{eq:gle} and its leading error.

Then Proposition~\ref{thm:nsv} asserts that the dynamics of \pde~\eqref{eq:shembedn} near any station \(\fx=X\in\XX\) is governed by the following \pde\ for the generating polynomial~\(\tu(X,\xi,y,t)\):
\begin{eqnarray}
\D t\tu &=&r\tu-(1+\partial_{yy}+2\partial_{y\xi}+\partial_{\xi\xi})^2\tu -\tu^3
\nonumber\\&&{}
+\sum_{\ell=1}^3
\sum_{n=3-\ell}^2\binom{n+\ell}{2}
\fL_\ell {\fu_2^{(n+\ell-2)}}\xit n
+\Ord{\nuv^{4}+\xi^{4}}.
\label{eq:shgfn}
\end{eqnarray}
The first line of~\eqref{eq:shgfn} is the well-known form of the Swift--Hohenberg \pde~\eqref{eq:shpde} in terms of a `fast phase' variable~\(y\) and a `slow space' variable~\(\xi\).
The second line of~\eqref{eq:shgfn} explicitly gives the leading order coupling error in terms of uncertain `slow' variable derivatives~(\(\D\fx{}\)) (denoted by superscripts in paranetheses) of the second derivative~\(\fu_2\).
The nonlinear order of error term in the second line of~\eqref{eq:shgfn} (in terms of amplitude~\eqref{eq:nuvg}) could be of higher order, but quartic errors are sufficient to derive the Ginzburg--Landau \pde.

Upcoming theory of Section~\ref{sec:cmtsm} asserts that there exists a slow manifold for the system~\eqref{eq:shgfn}, global in the spatial domain~\XX. 
The slow manifold is exponentially quickly attractive, in that transients decay roughly like~\Ord{e^{-t}}, from all nearby initial conditions.
That is, the slow manifold model of the Ginzburg--Landau \pde\ is emergent.

To approximate the slow manifold model we solve the system~\eqref{eq:shgfn} asymptotically.
We find approximations to the autonomous system global in the space domain~\XX\ by invoking approximation theorems for the local `non-autonomous' system formed by treating the inter-station coupling~\(\Dn\fx n{\fu_2}\) as an arbitrary time dependent forcing of the local dynamics: by finding solutions of the system~\eqref{eq:shgfn} to errors~\Ord{\nuv^{4}+\xi^{4}}, the slow manifold is then known to errors~\Ord{\nuv^{4}} \cite[Proposition~3.6]{Potzsche2006}.
One further detail is that it is best to treat the bifurcation parameter~\(r\) as a `second order' quantity: that is, we modify the Definition~\eqref{eq:nuvg} of the amplitude~\nuv\ to include the extra term~\(+|r|^{1/2}\) so that the parameter \(r=\Ord{\nuv^2}\).

The computer algebra of Appendix~\ref{sec:campfshe} constructs the slow manifold model for us:
section~\ref{sec:dsuo} caters for cross-sectional structures and the time dependence in the uncertain coupling; 
section~\ref{sec:dlce} forms the leading order expression~\eqref{eq:grun} for the coupling;
and section~\ref{sec:ifsm} uses the residuals of \pde~\eqref{eq:shgfn} to iteratively correct a slow manifold approximation until the residuals are zero to the specified order of error.
For example, limiting the coupling to \(\fu_2=\sum_{k=-1}^1\fu_{2,k}e^{iky}\) for simplicity, the code finds the slow manifold is
\begin{eqnarray}
\tu&=&\tc_+e^{iy}+\tc_-e^{-iy}
-\rat1{64}\tc_+^3e^{i3y}-\rat1{64}\tc_-^3e^{-i3y}
\nonumber\\&&{}
+\Z{}\left[(-6+24\Z{})\fu_{2,0}^{(2)}+30\Z{\fu_{2,0}^{(4)}}\right]
-\xit1\Z{}\left( 6\fu_{2,0}^{(1)}+10\fu_{2,0}^{(3)}\right)
\nonumber\\&&{}
-\xit2\Z{}\left( 12\fu_{2,0}^{(2)}+15\fu_{2,0}^{(4)} \right)
+\Ord{\nuv^4}+\Ord{e^{-\gamma t}},
\label{eq:shsmgfu}
\end{eqnarray}
for some decay rate \(\gamma \in(|r|,1)\).
This equation is in terms of the generating polynomials that implicitly resolve the dynamics of the various derivatives of the local field:
to resolve the field itself, just set \(\xi=0\) to find the slow manifold
\begin{eqnarray}
\fu&=&c_+e^{iy}+c_-e^{-iy}
-\rat1{64}c_+^3e^{i3y}-\rat1{64}c_-^3e^{-i3y}
\nonumber\\&&{}
+\Z{}\left[(-6+24\Z{})\fu_{2,0}^{(2)}+30\Z{\fu_{2,0}^{(4)}}\right]
+\Ord{\nuv^4}+\Ord{e^{-\gamma t}}.
\qquad\label{eq:shsmu}
\end{eqnarray}
The first line is the classic cubic approximation to the Swift--Hohenberg field. 
The second line gives the errors including the leading coupling error (more terms appear when one resolves more wavenumbers in the coupling).
The computer algebra of Appendix~\ref{sec:campfshe} simultaneously determines the evolution on the slow manifold in terms of the the evolution of the spatial gradients implicit in the generating polynomials~\(\tc_\pm\).
Again, setting \(\xi=0\) and rewriting \(\xi\)-derivatives as \(\fx\)-derivatives recovers the evolution of the complex amplitudes themselves:
\begin{eqnarray}
\D t{c_\pm}&=&rc_\pm-3c_\mp c_\pm^2+4\DD\fx{c_\pm} 
\nonumber\\&&{}
-6u_{2,\pm1}^{(2)}\mp i12u_{2,\pm1}^{(1)}
+\Ord{\nuv^4}+\Ord{e^{-\gamma t}}.
\label{eq:glecc}
\end{eqnarray}
When the initial conditions are real, then the amplitudes~\(c_\pm\) are complex conjugate and the first line is the classic Ginzburg--Landau \pde~\eqref{eq:gle}.
In the second line, the two terms in \(\fx\)-derivatives of~\(u_{2,\pm1}\) are the leading estimate of the uncertain coupling via the cross-section mode~\(e^{\pm iy}\).
Thus monitoring the leading coupling terms in the second lines of~\eqref{eq:shsmu}--\eqref{eq:glecc} will quantitatively estimate the error due to the approximation of slow variations in space.

The next section proves the existence and emergence of such a slow manifold model, but in general.

\subsection{Centre manifold theory supports modelling}
\label{sec:cmtsm}

Given the equivalence between dynamics described by the general nonlinear \pde~\eqref{eq:upden} and the dynamics~\eqref{eq:updeng} of the local generating polynomial~\eqref{eq:ugen}, our next task is to establish the existence and emergence of a model reduction of these nonlinear dynamics.
This section establishes on how centre manifold theory applies to the local \ode{}s in generating polynomial form~\eqref{eq:updeng} when coupled to its neighbours across the domain~\XX\ via the high order derivatives in~\(r[u]\).
I call~\eqref{eq:updeng} a set of \ode{}s because the partial derivatives~\(\D\xi{}\) just access different components in the generating polynomial: in its \(\xi\)~dependence the system appears as just a finite set of equations, finite because the truncated terms~\Ord{\xi^{N+1}} are spurious in our chosen Taylor series truncation.%
\footnote{The \ode{}s~\eqref{eq:updeng} often contain partial derivatives in the cross-sectional variable~\(y\): this nomenclature overlooks such partial derivatives.}
The `infinite dimensionality' of the original physical \pde~\eqref{eq:upden} arises via the inter-station coupling of the \ode{}s~\eqref{eq:updeng} which then form a \emph{system} over the domain~\XX.
It is this system that we address.
In particular, this section establishes the following proposition.

\begin{proposition}[existence and emergence]\label{thm:gpdecm}
Under Assumptions~\ref{ass:spec} and~\ref{ass:nonlin}, and in any open domain~\(\XX\) where the gradients of~\(u_N\) are sufficiently small, 
\begin{enumerate}
\item the \pde~\eqref{eq:upden} has a \(C^N\)~centre manifold in some neighbourhood of \(u=0\), and globally in the domain~\XX.  
\item For as long as solutions stay in the neighbourhood, solutions are exponentially quickly attracted to solutions on the centre manifold.
\end{enumerate}
\end{proposition}

\begin{proof}
Proposition~\ref{thm:nsv} establishes the generating polynomial~\eqref{eq:updeng} is equivalent to the \pde~\eqref{eq:upden}.
Thus we prove Proposition~\ref{thm:gpdecm} via the generating polynomial \ode{}s~\eqref{eq:updeng}.
Section~\ref{sec:acs} establishes the bases for the centre and stable subspaces of the system~\eqref{eq:updeng} over domain~\XX\ which separates the linear dynamics, globally in~\XX.
Using extant theory, primarily that by \cite{Aulbach96, Aulbach2000}, section~\ref{sec:teacm} then establishes that there exists a slow manifold in some neighbourhood, and
section~\ref{sec:acme} establishes the emergence of the centre manifold.
\end{proof}

\subsubsection{Centre and stable subspaces separate}
\label{sec:acs}

A centre manifold is typically based on the subspaces of an equilibrium: here we assume the equilibrium is at the origin \(\tu=0\)\,, because the coupling~\(r\) is then also zero.
Recall that section~\ref{sec:mlas} establishes the existence and parametrisation of a centre subspace provided Assumption~\ref{ass:spec} holds.
This section also invokes Assumption~\ref{ass:spec} and hence all the results of section~\ref{sec:mlas} hold here: the difference being the symbolic representation now invokes the data structure of the generating polynomials and relevant derivatives~\(\D\xi{}\).

Under Assumption~\ref{ass:spec} and for each cross-section~\(X\in\XX\)\,: there are \(m(N+1)\)~centre eigenvalues of \ode{}s~\eqref{eq:updeng}; for the generating polynomial \ode{}s~\eqref{eq:updeng} the corresponding (generalised) eigenvectors are, from~\eqref{eq:vamat}, the \(m(N+1)\)~columns of polynomials
\begin{equation*}
\tV=\begin{bmatrix} V_0& V_1+\xit1V_0&
V_2+\xit1V_1+\xit2V_0&\cdots& V_N+\xit1V_{N-1}+\cdots+\xit NV_0 \end{bmatrix}.
\end{equation*}
The following argument establishes these are the centre eigenvectors.
For the \ode{}s~\eqref{eq:updeng} for the generating polynomial~\tu, define the linear operator \(\tL:=\sum_\ell\fL_\ell\Dn\xi\ell{}\)\,.
Correspondingly define the linear operator \(\tA:=\sum_\ell A_\ell\Dn\xi\ell{}\) for centre variables~\tc.
Then, from~\eqref{eq:vzadef} and the recursion~\eqref{eq:topevn},
\begin{eqnarray*}
\tL\tV&=&\begin{bmatrix} \fL_0V_0 & 
(\fL_0V_1+\fL_1V_0)+\xit1\fL_0V_0 &
%(\fL_0V_2+\fL_1V_1+\fL_2V_0)+\xit1(\fL_0V_1+\fL_1V_0)+\xit2\fL_0V_0 
\cdots \end{bmatrix}
\\&=&\begin{bmatrix} V_0A_0 & 
(V_1A_0+V_0A_0)+\xit1V_0A_0 &
\cdots \end{bmatrix}
\\&=&\tV\tA\,.
\end{eqnarray*}
Hence the subspace \(\tu=\tV\tc\) is invariant under the linear \pde\ \(\de t\tu=\tL\tu\), and the centre variables~\tc\ satisfy
\begin{equation*}
\de t\tc=\tA\tc=\sum_\ell A_\ell\Dn\xi\ell{\tc}\,,
\end{equation*}
which directly corresponds to~\eqref{eq:dcdta} and~\eqref{eq:ssmod}.
Since \(\xi\)~is a proxy for the local longitudinal coordinate, the eigenvectors in~\tV\ encapsulate the interaction between longitudinal gradients of the field and cross-sectional structures.
The columns of~\tV\ form a basis for the centre subspace~\(\EE_c^N(X)\) at any station \(X\in\XX\).
Identical results hold for all stations \(X\in\XX\)\,, so \(\EE_c^N(\XX)=\bigtimes_{X\in\XX}\EE^N_{c}(X)\) forms the centre subspace of the system~\eqref{eq:updeng} over the domain~\XX.

Exactly analogous arguments, as in section~\ref{sec:mlas}, also establish a similar basis for the collective stable space~\(\EE_s^N(\XX)\).
At each station~\(X\in\XX\), there is a subspace \(\tu=\tW\td\) which is invariant under the linear \pde\ \(\de t\tu=\tL\tu\), and the stable variables~\td\ satisfy
\begin{equation*}
\de t\td=\tB\td=\sum_\ell B_\ell\Dn\xi\ell{\td}\,.
\end{equation*}

\subsubsection{There exists a centre manifold}
\label{sec:teacm}

To establish Part~1 of Proposition~\ref{thm:gpdecm} we invoke theory by \cite{Aulbach96, Aulbach2000} and hence now establish its preconditions in the generating polynomial form~\eqref{eq:updeng}.
Consider the general system~\eqref{eq:updeng} over the set of stations~\(X\in\XX\): then system~\eqref{eq:updeng} over all stations in~\XX\ is well-posed and autonomous except for coupling at the boundary~\(\partial\XX\) providing effectively non-autonomous forcing.
The system~\eqref{eq:updeng} has two closed \tL-invariant subspaces \(\EE_c^N(\XX)\) and~\(\EE_s^N(\XX)\), with a spectral gap.
The restrictions of~\tL\ to these spaces generate strongly continuous semigroups as they are just the collection over~\XX\ of a block upper triangular operator with~\(\fL_0\) on the diagonal, which by Assumption~\ref{ass:spec} has the requisite strongly continuous semigroups \cite[]{Aulbach96}. %\cite[p.266]{Gallay93}.
Also under Assumption~\ref{ass:spec}, the spectrum has the requisite spectral gap: \(|\Re\lambda_s|\geq\beta>N\alpha\geq N|\Re\lambda_c|\) applies uniformly over domain~\XX.%\cite[p.267]{Gallay93}.

We want to consider the inter-station coupling~\(r[u]\) appearing in system~\eqref{eq:updeng} as a perturbing `nonlinearity'.
First, the (multinomial) nonlinear terms in~\(r\), gathered in the~\Ord{} term of~\eqref{eq:grun}, are spurious since they are only there to cancel with high order, nonlinear, multinomial terms implicit in~\(f[\tu]\) and thus not present in the dynamics of the \pde~\eqref{eq:upden} when expanded in its Taylor polynomial~\eqref{eq:utrtn}.
Section~\ref{sec:gfs} shows in the example how such nonlinear terms arise to cancel with other implicitly introduced terms.

Introduce~\(\delta\) to parametrise both a nonlinear cut-off of nonlinearity~\(f\) and a `low pass filter' of the coupling.
For any smooth enough function~\(h(x)\) with domain~\XX, let the Fourier transform~\(H(\kappa)\) of~\(h(x)\), in a suitably generalised sense to account for~\XX, be such that
\(h(x)=\int_{-\infty}^\infty e^{i\kappa x}H(\kappa)\,d\kappa\)\,.
By Parseval's theorem, \(\int_\XX|h|^2\,dx=L\int_{-\infty}^\infty |H|^2\,d\kappa\) where the length \(L:=\int_\XX 1\,dx\)\,.
For the purposes of this section, let the spatial derivative operator~\(\D x{}\) denote the low-pass filtered version of the usual derivative; that is, in this section
\(\D xh:=\int_{-\delta}^\delta i\kappa e^{i\kappa x}H(\kappa)\,d\kappa\)\,.
Then straightforward algebra derives the bound that
\begin{equation*}
\int_\XX \left|\D xh\right|^2dx
\leq \delta^2L\int_{-\delta}^\delta |H|^2\,d\kappa
\leq \delta^2\int_\XX|h|^2\,d\kappa\,,
\end{equation*}
and similarly for higher derivatives.
That is, this low-pass filtered derivative is bounded, \(\|\D x{}\|\leq\delta\) for a suitable norm.
%(analogous to a nonlinear mollifier often invoked in the theory of nonlinear, non-autonomous dynamics \cite[p.265, e.g.]{DaPrato96b}).
Consequently higher order derivatives are also suitably bounded, \(|\Dn x\ell{u_N}|\leq\delta^\ell |u_N|\), and \(\Dn x\ell{u_N}\) are~\(C^{2N-\ell}\).
These bounds decrease with parameter~\(\delta\).
%Adjoining \(\de t\delta=0\) to the system~\eqref{eq:updeng}, then the inter-station coupling~\(r[u]\) appears as a \(C^{N+1}\)~nonlinearity.

The nonlinearity~\(f\) in system~\eqref{eq:updeng} is required to be~\(C^{N+p+1}\) (Assumption~\ref{ass:nonlin}).
Since the derivatives~\(\Dn\xi n\tu\) operate only upon the generating polynomial~\tu, of \(N\)th~degree, then the derivative operator~\(\D\xi{}\) in~\(f\) is bounded.
With a suitable cut-off the nonlinearity becomes bounded and Lipschitz  \cite[e.g.]{Vanderbauwhede88, Haragus2011, Chicone2006, Mielke86}.
% My gamma could be Gallay's beta.
Theorem~6.1 of \cite{Aulbach96} then applies to the cut-off version of system~\eqref{eq:updeng}, for some small enough cut-off paramter \(\delta>0\)\,, to ensure the existence of a global \(C^N\)~centre manifold, tangent to the centre subspace~\(\EE_c^N(\XX)\) at the origin. 

The cut-off version of system~\eqref{eq:updeng} is the original in a finite neighbourhood proportional to parameter~\(\delta\), so the centre manifold of system~\eqref{eq:updeng} exists in such a neighbourhood, which establishes Part~1 of Proposition~\ref{thm:gpdecm}.
The restriction on the cut-off parameter~\(\delta\) means the resultant model is theoretically supported in regimes where the coupling derivatives~\(\Dn x\ell{u_N}\) are small enough to be in the low pass band of the filter, as required by Proposition~\ref{thm:gpdecm}.
It is in only this statement that we make the slowly varying assumption of multiscale modelling.

\subsubsection{A centre manifold emerges}
\label{sec:acme}

Given the conditions invoked in the previous section~\ref{sec:teacm}, Theorem~4.1 of \cite{Aulbach2000} asserts 
the (cut-off) system~\eqref{eq:updeng} is topologically equivalent to 
\begin{equation}
\de t\tC=\tA\tC+\tF(t,\tC),\quad
\de t\tD=\tB\tD\,,
\label{eq:dtctddt}
\end{equation}
for some centre and stable variables \tC\ and~\tD, where \(\tC\approx\tc\) and \(\tD\approx\td\)\,,  operators \tA\ and~\tB\ are given in section~\ref{sec:acs}, and for some perturbation~\tF.
Because the spectrum of~\tB\ satisfies \(\Re\lambda\leq-\beta<0\), these new stable variables \(\tD\to0\) as \(t\to\infty\)\,.
The centre manifold is \(\tD=0\)\,.
In the original system~\eqref{eq:updeng}, without the cut-off, there is the extra caveat that this decay is guaranteed to apply only as long solutions stay in the finite neighbourhood.
Because the evolution of~\tC\ under~\eqref{eq:dtctddt} is identical on the centre manifold \(\tD=0\) to off the centre manifold, solutions off the centre manifold approach solutions on the centre manifold.
This establishes Part~2 of Proposition~\ref{thm:gpdecm}---except for the exponential rate.
Thus the evolution on the centre manifold emerges as the long term dynamics global across the domain~\XX, albeit local in amplitude~\(\|\uv\|\).

The topological equivalence of \cite{Aulbach2000}, although continuous, may not be as smooth as needed.
To establish the exponential rate let's return to the `vector' form~\eqref{eq:odesn} which is more convenient here, albeit modified for nonlinearity.
Assume we have changed coordinates at each station \(X\in\XX\) to linearly separate the centre and fast variables, say \(\vec c(X,y,t)\) and \(\vec d(X,y,t)\) respectively, as in system~\eqref{eqs:ecde} but with nonlinearities.
Recall that theory for non-autonomous systems asserts there exists a \emph{smooth} coordinate transform for nonlinear non-autonomous systems that nonlinearly decouples centre and stable variables \cite[e.g.]{Roberts06k}, even for random systems \cite[e.g.]{Arnold03}. 
The procedures of sections~\ref{sec:gfs}, \ref{sec:cucncm}, and~\ref{sec:anpf} provide practical methods to construct approximations of such centre manifold models.
Thus, regarding the coupling \(\Dn x\ell{u_N}\) as a time dependent input to the local dynamics at a station,  a smooth coordinate transform exists, \((\Cv,\Dv)\leftrightarrow(\cv,\dv)\) for all stations \(X\in\XX\)\,, so that the local stable variables evolve according to 
\begin{equation}
\de t{\Dv}=\left[\cB+G(\Cv,t)\right]\Dv 
\quad\text{where }
\cB:=\begin{bmatrix}  B_0& B_1& B_2&\cdots& B_N
\\0_m& B_0& B_1&\ddots&\vdots
\\0_m&0_m& B_0&\ddots& B_{2}
\\\vdots&\ddots& \ddots& \ddots& B_1
\\0_m& \cdots &0_m&0_m& B_{0}
\end{bmatrix},
\label{eq:dddtb}
\end{equation}
and matrix \(G=\Ord{|\Cv|}\) where the filtered coupling with neighbouring stations leads to the notional time dependence in~\(G\).

Under the following assumption that characterises the spatial interactions of the stable modes, Lemma~\ref{lem:roc} completes the proof of the existence and emergence Proposition~\ref{thm:gpdecm} by bounding the rate of emergence of the centre manifold. 

\begin{assumption}
Recall the eigenvalues of~\(B_0\) have real-part\({}\leq-\beta\) (Assumption~\ref{ass:spec}).
Let the basis for \(\vec D=(D_0,D_1,\ldots,D_N)\) be chosen so that~\(B_0\) satisfies \(\tr DB_0D\leq-\gamma |D|^2\) for some \(0<\gamma \leq\beta\)\,. 
Assume the off-diagonal entries in~\cB\ satisfy \(\|B_n\|\leq \gamma /N\) for \(n=1,\ldots,N\)\,.
\end{assumption}

\begin{lemma}[rate of emergence] \label{lem:roc}
The centre manifold emerges from at least a surrounding neighbourhood of initial conditions, for as long as solutions stay in the neighbourhood, at a rate at least~\(\gamma'/2\) for any \(0<\gamma'<\gamma\).
\end{lemma}

\begin{proof}
To bound the rate of attraction to the centre manifold \(\vec D=\vec 0\), define the Lyapunov function \(\cE:=\rat12\tr{\vec D}\vec D\).
Then straightforward algebra deduces that the time derivative
\begin{eqnarray*}
\de t\cE &=&\sum_{n\leq\nu}\tr D_nB_{\nu-n}D_\nu+\tr{\vec D}G\vec D
\\&\leq&-\gamma \sum_{n=0}^N|D_n|^2+\frac{\gamma }N\sum_{n<\nu}|D_n||D_\nu|+\tr{\vec D}G\vec D
\\&\leq&-\frac{\gamma }2\sum_{n=0}^N|D_n|^2-\frac{\gamma }{2N}\sum_{n<\nu}\left(|D_n|-|D_\nu|\right)^2+\tr{\vec D}G\vec D
\\&\leq&-\frac{\gamma }2\sum_{n=0}^N|D_n|^2+\tr{\vec D}G\vec D\,.
\end{eqnarray*}
That is, 
\begin{equation*}
\de t\cE\leq \tr\Dv\left[-\rat12\gamma I+G(\Cv,t)\right]\Dv\,.
\end{equation*}
Since \(G(\vec0,t)=0\) and by continuity of~\(G\), there exists a finite neighbourhood of \(\Dv=\vec 0\) such that \(\de t\cE\leq-\gamma'\cE\) for any \(0<\gamma'<\gamma\leq\beta\) and hence \(\cE=\Ord{e^{-\gamma't}}\) as \(t\to\infty\) for as long as solutions stay in the neighbourhood.
That is, \(\Dv=\Ord{e^{-\gamma't/2}}\) which proves the lemma.
\end{proof}

\section{Conclusion}

This article develops a new general theoretical approach to supporting the much invoked practical approximation of slow variations in space.
The approach is to examine the dynamics in the locale around any cross-section.
We find that a Taylor series approximation to the dynamics is only coupled to neighbouring locales via the highest order resolved derivative. 
Treating this coupling as an `uncertain forcing' of the local dynamics we in essence apply non-autonomous centre manifold theory to prove the existence and emergence of a local model.
This support applies for all cross-sections and so establishes existence and emergence globally in the domain.
Sections~\ref{sec:mdhx}--\ref{sec:pmid} develop the approach for linear systems, and then sections~\ref{sec:nhem}--\ref{sec:mndcd} generalise the approach to nonlinear systems.

One result is that the new theory recovers a version of traditional multiple scale modelling as a special case (Corollary~\ref{cor:msm}), and justifies rigorously an established formal procedure (Corollary~\ref{cor:robformproc}).

In this theory there is no requirement for some small parameter to tend to zero. 
A centre manifold model exists for solutions up to at least some finite amplitude and up to at least some finite spatial gradients of the variables.

Because the `uncertain' coupling term accounts for errors in the  slowly varying assumption, this assumption need not be imposed on the construction of the slow manifold model (section~\ref{sec:cucncm}); it only need be a restriction on the regime of solutions to which the model is applied.
Indeed, the theory justifies the centre manifold model to exist and emerge over any open domain not including significant boundary layers or shocks.

This article focussed on the case of a centre manifold amongst centre-stable dynamics as this case is the most broadly useful in modelling dynamics.
The key required properties are the persistence of centre manifolds under perturbations by both nonlinearities and time dependent `forcing'.
Since this property of persistence is shared by other invariant manifolds, I expect the same approach will support the existence and perhaps relevance of other invariant manifolds with slow variations in space.

This approach opens much for future research.
It may be able to illuminate the thorny issue of providing boundary conditions to slowly varying models \cite[e.g]{Segel69, Roberts92c}.
One significant restriction on the analysis here is that the system is homogeneous in space: 
however, preliminary research suggests that we can adapt the approach to inhomogeneous systems, and to systems where the longitudinal operators are nonlocal rather than the local derivatives~\(\Dn xn{}\) invoked here.
Further, a generalisation to multiple slow dimensions should be valuable in order to model problems such as shells, plates and Turing patterns \cite[cf.]{Mielke91a}.

\paragraph{Acknowledgement} The Australian Research Council Discovery Project grant DP120104260 helped support this research.
I thank Arthur Norman and colleagues who maintain the Reduce software used.

\appendix
\section{Computer algebra models the heat exchanger}
\label{sec:camhe}

This section lists and describes computer algebra code to analyse the Taylor series approach to the slowly varying modelling of the heat exchanger~\eqref{er:hecann} of Figure~\ref{fig:hescheme}.
I invoked the free computer algebra package Reduce\footnote{\url{http://www.reduce-algebra.com/} gives full information about Reduce.} \cite[e.g.]{MacCallum91}.
Analogous code will work for other computer algebra packages.

An if-statement decides whether to execute this appendix, or not.  
\begin{reduce}
if 0 then begin
\end{reduce}
Then make the printing appears nicer.
\begin{reduce}
on div; on revpri; off allfac; linelength 60$
factor df,c,d;
\end{reduce}

\subsection{Substitute a Taylor series}
\label{sec:sts}

Define coefficients of the local expansion of the fields: they generally depend upon station~\(X\) (\verb|xx|) and time~\(t\).
\begin{reduce}
operator c; depend c,xx,t;
operator d; depend d,xx,t;
\end{reduce}

Choose to expand in a Taylor series to the order~\(N\) specified here; choose \(N=4\) to reproduce the modelling discussed in section~\ref{sec:mdhx}. 
The last coefficient being at an unknown location so make it additionally a function of position~\(x\) as well as station~\(X\) and time~\(t\).
\begin{reduce}
nn:=4;
depend c(nn),x;
depend d(nn),x;
\end{reduce}

Form the Taylor series~\eqref{er:ctrt} and~\eqref{er:dtrt} of the fields.
\begin{reduce}
ct:=(for n:=0:nn sum c(n)*(x-xx)^n/factorial(n));
dt:=(for n:=0:nn sum d(n)*(x-xx)^n/factorial(n));
\end{reduce}

Find residuals~\eqref{eq:ctay4}--\eqref{eq:dtay4} of the \pde{}s~\eqref{er:hecann} when the fields are expanded in this Taylor series.
\begin{reduce}
resc:=-df(ct,t)+df(dt,x);
resd:=-df(dt,t)-dt+df(ct,x);
\end{reduce}

\subsection{Local ODEs}
\label{sec:lode}

Derive a set of linearly independent equations~\eqref{eqs:5pde} simply by differentiation and evaluation at $x=X$\,:
\begin{reduce}
array odec(nn),oded(nn);
for n:=0:nn do begin
  write odec(n):=sub(x=xx,df(resc,x,n));
  write oded(n):=sub(x=xx,df(resd,x,n));
end;
\end{reduce}

\subsection{Time dependent coordinate transform}
Now derive the time dependent normal form transform of section~\ref{sec:tdnf}.
For convenience, change the name of the forcing by the uncertain coupling terms.
Invoke a time,~\verb|tt|, that is notionally independent of the `slow' time evolution of variables so that we can treat the time dependence in variables \Cv\ and~\Dv\ separately from the time dependence in the uncertain coupling \(c_{Nx}\) and~\(d_{Nx}\).
\begin{reduce}
operator w; depend w,tt;
subw:={ df(d(nn),x)=>w(d), df(c(nn),x)=>w(c) };
for n:=0:nn do begin
    write odec(n):=(odec(n) where subw);
    write oded(n):=(oded(n) where subw);
end;
depend tt,t;
\end{reduce}

Store the current transform in~\verb|cx| and~\verb|dx|, and the time derivatives of the new variables as \(\dot C_n=\verb|dcdt(n)|\) and \(\dot D_n=\verb|dddt(n)|\).
\begin{reduce}
operator cc; depend cc,xx,t; 
operator dd; depend dd,xx,t;
array dx(nn),cx(nn),dcdt(nn),dddt(nn);
let { df(dd(~n),t)=>dddt(n)
    , df(cc(~n),t)=>dcdt(n)
    , d(~n)=>dx(n), c(~n)=>cx(n) };
\end{reduce}
Let's choose to parametrise the slow subspace by the \(C_n=\verb|cc(n)|\) variables as we welcome history integrals appearing in the slow subspace evolution as encoding the uncertain coupling between neighbouring stations.

The initial approximation to the coordinate transform is the identity, with decay of stable variables~\(D_n=\verb|dd(n)|\).
\begin{reduce}
for n:=0:nn do cx(n):=cc(n);
for n:=0:nn do dx(n):=dd(n);
for n:=0:nn do dcdt(n):=0;
for n:=0:nn do dddt(n):=-dd(n);
\end{reduce}

Need to express the uncertain remainders as history integrals so use well established operators \cite[e.g.]{Roberts06k, Roberts06j}.
\begin{reduce}
operator z; linear z;
let { df(z(~f,tt,~mu),t)=>-sign(mu)*f+mu*z(f,tt,mu)
    , z(1,tt,~mu)=>1/abs(mu)
    , z(z(~r,tt,~nu),tt,~mu) =>
      (z(r,tt,mu)+z(r,tt,nu))/abs(mu-nu) when (mu*nu<0)
    , z(z(~r,tt,~nu),tt,~mu) =>
      -sign(mu)*(z(r,tt,mu)-z(r,tt,nu))/(mu-nu)
      when (mu*nu>0)and(mu neq nu)
    };
\end{reduce}

Define an operator to separate out terms in stable variables~\(D_k\).
\begin{reduce}
operator only; linear only;
let { only(dd(~k),dd)=>dd(k) , only(1,dd)=>0 };
\end{reduce}

Iterate to separate the slow  and stable subspaces: 
this algorithm takes six iterations to construct the \(N=4\) case discussed in section~\ref{sec:tdnf}.
\begin{reduce}
for iter:=1:99 do begin
    ok:=1;
    for n:=0:nn do begin 
      resd:=oded(n);
      dddt(n):=dddt(n)+(gd:=only(resd,dd));
      dx(n):=dx(n)+z(resd-gd,tt,-1); 
      resc:=odec(n);
      cx(n):=cx(n)-(fd:=only(resc,dd));
      dcdt(n):=dcdt(n)+(resc-fd);
      ok:=if {resc,resd}={0,0} then ok else 0;
    end;
    showtime;
    if ok then write iter:=iter+10000;
end;
\end{reduce}

Write the resultant slow subspace~\eqref{eq:cxC}, stable subspace~\eqref{eq:dxD} and their corresponding evolution~\eqref{eq:dDdt}--\eqref{eq:dCdt}.
\begin{reduce}
for n:=0:nn do write cx(n):=cx(n);
for n:=0:nn do write dx(n):=dx(n);
for n:=0:nn do write dcdt(n):=dcdt(n);
for n:=0:nn do write dddt(n):=dddt(n);
\end{reduce}

End the if-statement that chooses whether to execute the code of this appendix.
\begin{reduce}end;\end{reduce}

\section{Computer algebra models the nonlinear heat exchanger}
\label{sec:camnhe}

This section lists and comments on computer algebra code to analyse the Taylor series and generating function approaches to the slowly varying modelling of the nonlinear heat exchanger~\eqref{er:hecann}.
As in the preceding section, it uses the free computer algebra package Reduce.\footnote{\url{http://www.reduce-algebra.com/}}
Analogous code will work for other computer algebra packages.
Almost exactly the same code will analyse a variety of `heat exchanger' \pde{}s simply by modifying the advection and nonlinear terms.

An if-statement decides whether to execute this appendix, or not.
\begin{reduce}
if 0 then begin
\end{reduce}
Make printing prettier.
\begin{reduce}
on div; on revpri; off allfac; linelength 60$
\end{reduce}

Choose to expand in a Taylor series to the order specified here; choose \(N:=2\) for illustration, but I have confirmed the algorithm for all \(N\leq9\). 
\begin{reduce}
nn:=2;
\end{reduce}

\subsection{In the interior}

Define coefficients of the local expansion~\eqref{ers:trt} of the fields.
They depend upon time~\(t\) and the station~\(X\).
\begin{reduce}
operator c; depend c,xx,t;
operator d; depend d,xx,t;
\end{reduce}

The last term in the Taylor series~\eqref{ers:trt}, being at an unknown location, is additionally a function of position~\(x\) as well as time~\(t\) and station~\(X\).
\begin{reduce}
depend c(nn),x;
depend d(nn),x;
\end{reduce}

Form the Taylor series~\eqref{ers:trt} of the mean and difference fields.
\begin{reduce}
cc:=(for n:=0:nn sum c(n)*(x-xx)^n/factorial(n));
dd:=(for n:=0:nn sum d(n)*(x-xx)^n/factorial(n));
\end{reduce}

Find residuals of the \pde{}s~\eqref{er:hecann}.
One may modify these advection and nonlinear terms to analyse variations to the \pde{}s.
\begin{reduce}
resc:=-df(cc,t)+df(dd,x)-cc*dd;
resd:=-df(dd,t)-dd+df(cc,x)-(cc^2+dd^2)/2;
\end{reduce}

\subsection{Exact local nonlinear ODEs}
The derived expressions for the residuals are exact everywhere. 
But they are useful near the section \(x=X\)\,.
To find a set of linearly independent equations repeatedly differentiate the residuals and evaluate at $x=X$\,:
\begin{reduce}
array odec(nn),oded(nn);
for n:=0:nn do begin
  write odec(n):=sub(x=xx,df(resc,x,n));
  write oded(n):=sub(x=xx,df(resd,x,n));
end;
\end{reduce}

%Check up a Lyapunov function?
%\begin{reduce}
%array rhsc(nn),rhsd(nn);
%for n:=0:nn do rhsc(n):=odec(n)+df(c(n),t);
%for n:=0:nn do rhsd(n):=oded(n)+df(d(n),t);
%lyap:=(d(1)-c(2))^2/2+(d(0)-c(1))^2/2+d(2)^2/2;
%write 
%dlydt:=(df(lyap,t) where {df(c(~n),t)=>rhsc(n), df(d(~n),t)=>rhsd(n)});
%write aa:=mat
%(( 0, 0, 0, 0, 0, 0)
%,( 0,-2, 2, 0, 0,-1)
%,( 0, 2,-2, 0, 0, 1)
%,( 0, 0, 0,-2, 2, 0)
%,( 0, 0, 0, 2,-2, 0)
%,( 0,-1, 1, 0, 0,-2))/2;
%write uu:=tp mat((c(0),d(0),c(1),d(1),c(2),d(2)));
%write uu:=(tp uu)*aa*uu;
%write dlydt:=dlydt-uu(1,1);
%end;end;
%\end{reduce}
%[ 0, 0, 0, 0, 0, 0
% 0,-2, 2, 0, 0,-1
% 0, 2,-2, 0, 0, 1
% 0, 0, 0,-2, 2, 0
% 0, 0, 0, 2,-2, 0
% 0,-1, 1, 0, 0,-2]/2

In computer algebra we want a variable to count the order of each of the terms in all of the equations. 
In a general truncation the Definition~\eqref{er:unorm} of the amplitude becomes
\begin{equation*}
\nuv:=\sum_{n=0}^N\left(|c_n|^{1/(n+1)}+|d_n|^{1/(n+1)}\right),
\end{equation*}
then \(c_n,d_n=\Ord{\nuv^{n+1}}\) as \(\nuv\to0\).
Introduce~\verb|small|, and count variables according to this amplitude  so that a factor labelled through multiplication by~\(\verb|small|^p\) denotes a factor~\Ord{\nuv^p}.
The procedure~\verb|wsmall| encodes this choice (as it may be varied in other contexts).
Do not explicitly count the order of the \(d_n\)~variables as on the slow manifold they will naturally be counted:
it is only if we were to perform a normal form, near identity, coordinate transform that we would want to explicitly count the \(d_n\)~variables.
\begin{reduce}
factor small;
procedure wsmall(n); small^(n+1);
weighting:=for n:=0:nn collect c(n)=wsmall(n)*c(n);
\end{reduce}

Also decide on the level of detail resolved in the influence of the coupling terms \(c_{Nx}\) and~\(d_{Nx}\).
Here propose that the coupling terms \(c_{Nx},d_{Nx}=\Ord{\nuv^{N+1}}\).
Reducing this proposed order to~\Ord{\nuv^{N}} generates quadratic terms in these uncertain factors which appears to increase complication for insignificant benefit. 
One might argue that the coupling terms should be one order higher, \(c_{Nx},d_{Nx}=\Ord{\nuv^{N+2}}\), as they both involve an extra spatial derivative: 
however, such a view is unnecessarily redolent of the multiple scales straightjacket; 
instead let's allow the domain of validity of our analysis to be larger than this by assuming the coupling to be~\(\Ord{\nuv^{N+1}}\) as is consistent with the measure of \(c_N\) and~\(d_N\) in the amplitude. 

For convenience in the computer algebra, rename the coupling terms as \(w_c\) and~\(w_d\).
Like \(c_{Nx}\) and~\(d_{Nx}\) these abbreviations~\(w\) depend upon station~\(X\) and time~\(t\), but we invoke a separate time symbol,~\verb|tt|, in order to separate the time dependence in the coupling from the other slow time evolution on the slow manifold.
\begin{reduce}
operator w; depend w,xx,tt;
coupling:={ df(d(nn),x)=>wsmall(nn)*w(d) 
          , df(c(nn),x)=>wsmall(nn)*w(c) };
depend tt,t;
\end{reduce}

Implement the accounting of order in the \ode{}s.
\begin{reduce}
for n:=0:nn do begin
    write odec(n):=sub(weighting,(odec(n)where coupling));
    write oded(n):=sub(weighting,(oded(n)where coupling));
end;
\end{reduce}

The governing \ode{}s~\eqref{er:odes} then look like the following with the explicit accounting of the orders of both~\(c_n\) and the coupling.  
These equations use the symbol~\(\epsilon\) to denote the order counting variable~\verb|small|.
\begin{eqnarray*}&&
\epsilon\dot c_{0}=d_{1}-\epsilon c_0d_0,\\&&
\epsilon^2\dot c_{1}=d_{2}-\epsilon c_0d_1-\epsilon^2c_1d_0,\\&&
\epsilon^3\dot c_{2}=-\epsilon c_0d_2-2\epsilon^2c_1d_1-\epsilon^3c_2d_0 +\epsilon^3 3w_d,
\\&&
\dot d_{0}=-d_0+\epsilon^2[c_{1}-\rat12c_0^2]-\rat12d_0^2,\\&& 
\dot d_{1}=-d_1+\epsilon^3[c_{2}-c_0c_1]-d_0d_1,\\&&
\dot d_{2}=-d_2+\epsilon^4[-c_1^2-c_0c_2]
-d_1^2-d_0d_2 +\epsilon^3 3w_c.
\end{eqnarray*}

\subsection{Time dependent slow manifold}
\label{sec:tdsm}

In the computer algebra, store the current slow manifold in variables~\verb|d0|, and the evolution of the slow variables in~\verb|g0|: the zero denoting quantities of the slow manifold.
Initially both are approximated by the zero initialisation of this array declaration.
\begin{reduce}
array d0(nn),g0(nn);
let { d(~n)=>d0(n)
    , df(c(~n),t)=>g0(n) };
\end{reduce}

Need to express the uncertain remainders as integrals so use well established operators from non-autonomous and stochastic slow manifold theory \cite[e.g.]{Roberts06k}:
\begin{equation}
\verb|z(f,tt,mu)|:=\int_0^t e^{\mu(t-s)}f(s)\,ds
\quad\text{for }\mu<0\,.
\label{eq:conv}
\end{equation}
\begin{reduce}
operator z; linear z;
let { df(z(~f,tt,~mu),t)=>-sign(mu)*f+mu*z(f,tt,mu)
    , z(1,tt,~mu)=>1/abs(mu)
    , z(z(~r,tt,~nu),tt,~mu) =>
      (z(r,tt,mu)+z(r,tt,nu))/abs(mu-nu) when (mu*nu<0)
    , z(z(~r,tt,~nu),tt,~mu) =>
      -sign(mu)*(z(r,tt,mu)-z(r,tt,nu))/(mu-nu)
      when (mu*nu>0)and(mu neq nu)
    };
\end{reduce}
Let's choose to parametrise the slow manifold by the \verb|c(n)| variables, precisely, as we are not worried by history integrals appearing in the slow manifold evolution.
This choice simplifies analysis.

Truncate to an order determined by the number of terms in the original Taylor series: errors~\Ord{\nuv^{N+2}} may be best in general---the errors being one order higher than the smallest resolved term, but in this problem it appears that \Ord{\nuv^{N+3}}~errors also gives good answers.
\begin{reduce}
write "Truncate to errors O(small^",nn+3,")";
for o:=nn:nn do let small^(o+3)=>0;
\end{reduce}

Iterate to find the slow manifold.
Modify the evolution updates by the weight of the variable~\(c_n\) as we have already counted its weight.
\begin{reduce}
for iter:=1:99 do begin
    ok:=1;
    for n:=0:nn do begin 
      d0(n):=d0(n)+z(resd:=oded(n),tt,-1); 
      g0(n):=g0(n)+(resc:=odec(n))/wsmall(n); 
      ok:=if {resc,resd}={0,0} then ok else 0;
    end;
    showtime;
    if ok then write iter:=iter+10000;
end;
\end{reduce}

Write the resultant slow manifold, and note the convolutions are only over the past history.
\begin{reduce}
for n:=0:nn do write d0(n):=d0(n);
for n:=0:nn do write g0(n):=g0(n);
\end{reduce}

This code deduces the slow manifold~\eqref{eqq:smd0} and evolution~\eqref{eqq:smdcdt} thereon, with absolute errors~\Ord{\nuv^5}.

\subsection{The slow manifold via the generating function}
\label{sec:smvgf}

Start by confirming the order of the Taylor polynomial.
Factorize \verb|small| for clarity.
\begin{reduce}
nn:=nn;
factor small;
\end{reduce}

Introduce two generating function polynomials that encapsulate the \((N+1)\) local derivatives within an \(N\)th~degree polynomial, generalising~\eqref{er:genfun}:
\begin{equation}
\tc(\xi ,X,t)=\sum_{n=0}^N c_n(X,t)\xit n\,,\quad
\td(\xi ,X,t)=\sum_{n=0}^N d_n(X,t)\xit n \,.
\label{er:ggenfun}
\end{equation}

Omit the higher order terms, as with correct absolute error truncation they now have no effect on the results, and only complicate the details of the construction.
Anyway, the \verb|hot| labelled terms are only appropriate for the case \(N=2\) (and only for this specific nonlinearity).
\begin{reduce}
hot:=0;
\end{reduce}

Parametrise the slow manifold by~\tc\ which evolves in time according to~\eqref{er:tcte} with explicit count of order in~\verb|small|, denoted by~\(\epsilon\), to control asymptotic truncation.
The modelling involves two, time dependent, `uncertain' terms called~\verb|w(c)| and~\verb|w(d)| for no good reason.
The following appears to be compatible with the earlier slow manifold.
\begin{reduce}
depend tc,t,xi;
let df(tc,t)=>small*df(td,xi)-small*tc*td
    +(nn+1)*xi^nn/factorial(nn)*w(d)
    +hot*small*(
        xi^3/2*((df(tc,xi)-xi*df(tc,xi,2))*df(td,xi,2)
            +(df(td,xi)-xi*df(td,xi,2))*df(tc,xi,2))
        +xi^4/4*df(tc,xi,2)*df(td,xi,2) )
    ;
\end{reduce}

Now iterate to construct the slow manifold starting from the initial approximation that \(\td=0\).  
Find that truncating to \emph{relative} error~\Ord{\epsilon^{N+2}} is the same as the slow manifold construction of section~\ref{sec:tdsm}.
Also truncate to~\Ord{\xi^{N+1}}, corresponding to the finite generating polynomial, because the neglected terms do not change the results we extract. 
But actually implement truncation to~\Ord{\epsilon^{N+2}+\xi^{N+2}} for three reasons: because it is more efficient; because differentiation by~\(\xi\) is always accompanied by a multiplication by~\(\epsilon\); and because the leading order term in~\(\xi^{N+1}\) already has a factor of~\(\epsilon\).
\begin{reduce}
td:=0$
for o:=nn+2:nn+2 do let { 
    small^o=>0, xi^o=>0,
    xi*small^(o-1)=>0, small*xi^(o-1)=>0,
    small^~p*xi^~q=>0 when p+q>=o
    };
for iter:=1:99 do begin
\end{reduce}

Compute the residual of the \ode~\eqref{er:tdte}, and use the residual to update~\td.
The evolution of~\(\tc\) is then automatically updated by Reduce via the earlier \verb|let|-rule.
\begin{reduce}
    resd:=-df(td,t)-td+small*df(tc,xi)
    -small/2*(tc^2+td^2)+(nn+1)*xi^nn/factorial(nn)*w(c)
    +hot*small*(
        xi^3/2*((df(tc,xi)-xi*df(tc,xi,2))*df(tc,xi,2)
            +(df(td,xi)-xi*df(td,xi,2))*df(td,xi,2))
        +xi^4/8*(df(tc,xi,2)^2+df(td,xi,2)^2) )
    ;
    td:=td+z(resd,tt,-1);
\end{reduce}

Exit the iteration when the residual is zero to the specified order. 
\begin{reduce}
    showtime;
    if resd=0 then write iter:=iter+10000;
end;
\end{reduce}

Upon finishing the construction, find its version of the slow manifold evolution.
\begin{reduce}
dcdt:=df(tc,t)$
\end{reduce}

\subsection{Compare the two slow manifold views}
\label{sec:ctsmv}

Recover and compare the evolution and slow manifold of the generating polynomial results with that of the previous detailed Taylor series analysis.
Truncate to one higher order of error to match the \emph{absolute} error used in sections~\ref{sec:sme}--\ref{sec:ufnsm}.
\begin{reduce}
for o:=nn:nn do let small^(o+3)=>0;
\end{reduce}
Do not need to count the order of~\(c_n\) in~\(\tc\) as the various derivatives in the expansion are already counted, but we do need to multiply the various components by the appropriate absolute order when extracting the components from the generating polynomial.
\begin{reduce}
array dcndt(nn),dnn(nn);
tc:=for n:=0:nn sum xi^n/factorial(n)*c(n);
for n:=0:nn do write 
    dcndt(n):=wsmall(n)*coeffn(dcdt,xi,n)*factorial(n);
for n:=0:nn do write 
    dnn(n):=wsmall(n)*coeffn(td,xi,n)*factorial(n);
\end{reduce}

Passes the comparison check beautifully to confirm the generating polynomial approach is precisely equivalent to the specified order.
\begin{reduce}
for n:=0:nn do begin
    write "cerror",n,":=",dcndt(n)-wsmall(n)*g0(n);
    write "derror",n,":=",dnn(n)-d0(n);
end;
\end{reduce}
End the if-statement.
\begin{reduce}end;\end{reduce}

\section{Computer algebra models pattern formation in the Swift--Hohenberg PDE}
\label{sec:campfshe}

This section lists and comments on computer algebra code to analyse the generating function approach to the slowly varying modelling of the Swift--Hohenberg \pde~\eqref{eq:shpde}.
As in the preceding sections, it invokes the free computer algebra package Reduce.\footnote{\url{http://www.reduce-algebra.com/}}
Analogous code will work for other computer algebra packages.
Almost exactly the same code will analyse a variety of similar \pde{}s simply by modifying the nonlinear and perturbative terms.

An if-statement decides whether to execute this appendix, or not.
\begin{reduce}
if 1 then begin
\end{reduce}
Make printing prettier.
\begin{reduce}
on div; on revpri; off allfac; linelength 60$
\end{reduce}

Choose to analyse to the order specified here; choose \(N:=2\) for illustration, but have confirmed the algorithm works for all orders \(N\leq6\).
\begin{reduce}
nn:=2;
\end{reduce}

\subsection{Define some useful operators}
\label{sec:dsuo}

We expand the pattern solution in a complex Fourier series in the `fast' variable~\(y\), so here define operator \(\verb|cis|\,\theta=e^{i\theta}\).
Do not simplify \verb|cis(0)| as we want it for later pattern matching.
\begin{reduce}
operator cis;
let { df(cis(~a),~y)=>cis(a)*i*df(a,y)
    , cis(~a)*cis(~b)=>cis(a+b)
    , cis(~a)^~p=>cis(p*a)
    };
\end{reduce}

In the local slow manifold we need to account for the time variation of the uncertain coupling as history integrals.
I invoke established convolution operators~\eqref{eq:conv} from non-autonomous and stochastic slow manifold theory \cite[e.g.]{Roberts06k}.
Need to use a `fast' time,~\verb|tt|, that is notionally independent of the `slow' time evolution of variables.
\begin{reduce}
depend tt,t,cis;
operator z; linear z;
let { df(z(~f,tt,~mu),t)=>-sign(mu)*f+mu*z(f,tt,mu)
    , z(1,tt,~mu)=>1/abs(mu)
    , z(z(~r,tt,~nu),tt,~mu) =>
      (z(r,tt,mu)+z(r,tt,nu))/abs(mu-nu) when (mu*nu<0)
    , z(z(~r,tt,~nu),tt,~mu) =>
      -sign(mu)*(z(r,tt,mu)-z(r,tt,nu))/(mu-nu)
      when (mu*nu>0)and(mu neq nu)
    };
\end{reduce}

To find structures in the cross-section, define the operator~\verb|linv| to generate updates in the `fast' time and cross-section variables.
\begin{reduce}
operator linv; linear linv;
let { linv(cis(~m*y),cis)=>cis(m*y)/(1-m^2)^2
    , linv(~~a*cis(0),cis)=>z(a,tt,-1)*cis(0) 
    , linv(~~a*cis(~m*y),cis)=>z(a,tt,-(1-m^2)^2)*cis(m*y) 
    };
\end{reduce}

\subsection{Derive the leading coupling expression}
\label{sec:dlce}

One novel aspect of our approach is we quantify the leading order estimate of error in the slowly varying approximation.
To do so we need various terms in the highest order derivative of the notional Taylor series expansion: thus introduce~\verb|un| to denote~\(\fu_N(X,\fx,y,t)\), and use~\verb|un(p,k)| to denote the \(p\)th \(\fx\)-derivative of the \(k\)th~mode in the cross-section,~\(e^{iky}\).

In the computer algebra we prefer a variable to count the order of  each of the terms in all of the equations. 
Introduce~\verb|small|, and count variables according to the Definition~\eqref{eq:nuvg} of amplitude.
\begin{reduce}
factor small;
\end{reduce}

Decide how many modes of the `uncertain' coupling that we resolve in the cross-section by setting~\verb|kk|, although because we only resolve the linear effects so only modes \(k=\pm1\) affect the slow manifold evolution.
\begin{reduce}
kk:=2;
operator un; depend un,x,xx,tt;
tu:=small^(nn+1)*(for k:=-kk:kk sum un(0,k)*cis(k*y));
coupling:={ df(un(~p,~k),x)=>un(p+1,k) }$
\end{reduce}
Construct the uncertain coupling for the generating function approach.
Code into~\verb|ru|, the known terms in equation~\eqref{eq:grun} for~\(r[u]\) and invoke the linear operators~\eqref{eq:sheln} for the Swift--Hoheberg \pde.
When we later differentiate with respect to~\(\xi\) we automatically multiply by~\verb|small|: which means that here we have to compensate by dividing by~\verb|small| for each power of~\(\xi\).
\begin{reduce}
factor xi;
write
ru:=for ell:=1:4 sum
    for n:=max(nn-ell+1,0):nn sum xi^n/factorial(n)/small^n
    *factorial(ell+n)/factorial(nn)/factorial(ell+n-nn)
    *(df(if ell=1 then -4*df(tu,y)-4*df(tu,y,3)
    else if ell=2 then -2*tu-6*df(tu,y,2)
    else if ell=3 then -4*df(tu,y)
    else if ell=4 then -tu
    ,x,ell+n-nn) where coupling)$
\end{reduce}

\subsection{Initialise the slow manifold}

Parametrise the slow manifold by~\(c_\pm\) which evolves in time according to \(\D t{c_\pm}=g_\pm\) for some right-hand side to find.
\begin{reduce}
depend cp,t,xi;
depend cm,t,xi;
let { df(cp,t)=>gp, df(cm,t)=>gm };
\end{reduce}

The linear approximation is the slow subspace of the span of~\(e^{\pm iy}\), which are approximately equilibria.
\begin{reduce}
tu:=small*(cp*cis(y)+cm*cis(-y))$
gp:=gm:=0$
\end{reduce}

\subsection{Iteration finds the slow manifold}
\label{sec:ifsm}

Now iterate to construct the slow manifold.  
Implement truncation to residuals~\Ord{\nuv^{N+2}+\xi^{N+2}} because it is efficient, and because differentiation by~\(\xi\) is always accompanied by a multiplication by~\verb|small|.
\begin{reduce}
for o:=nn+2:nn+2 do let { 
    small^o=>0, xi^o=>0,
    xi*small^(o-1)=>0, small*xi^(o-1)=>0,
    small^~p*xi^~q=>0 when p+q>=o
    };
for iter:=1:99 do begin
\end{reduce}

In each iteration, compute the residual of the Swift--Hohenberg \pde~\eqref{eq:shpde}, including the leading `uncertain' coupling as in equation~\eqref{eq:shgfn}.
The multiplication by~\verb|small| that counts order according to amplitude~\eqref{eq:nuvg}, corresponds symbolically to the multiplication by~\(\epsilon\) that arise in the method of multiple scales, as established by  Corollary~\ref{cor:msm}.
\begin{reduce}
    v:=tu+df(tu,y,y)+small*2*df(tu,xi,y)+small^2*df(tu,xi,xi);
    resu:=-df(tu,t) +small^2*rr*tu -tu^3
    -(v+df(v,y,y)+small*2*df(v,xi,y)+small^2*df(v,xi,xi))
    +ru;
    write lengthres:=length(resu);
\end{reduce}

Use the residual to update the evolution on the slow manifold in~\(g_\pm\) and the slow manifold itself~\(\tu(X,\xi,y,c_+,c_-)\).
\begin{reduce}
    gp:=gp+(gpd:=coeffn(resu,cis(+y),1))/small;
    gm:=gm+(gmd:=coeffn(resu,cis(-y),1))/small;
    tu:=tu+linv(resu-gpd*cis(y)-gmd*cis(-y),cis);
\end{reduce}

Exit the iteration when the residual is zero to the specified order of errors. 
\begin{reduce}
    showtime;
    if resu=0 then write iter:=iter+10000;
end;
\end{reduce}

Upon finishing the construction, find its version of the slow manifold evolution.
\begin{reduce}
write dcpdt:=gp;
\end{reduce}

End the if-statement and the execution.
\begin{reduce}end;end;\end{reduce}

\raggedright
\bibliographystyle{agsm}
\bibliography{ajr,bib}

\end{document}